\documentclass[12pt]{amsart}
\usepackage{a4wide}
\usepackage{amssymb}
\usepackage{amsthm}
\usepackage{amsmath}
\usepackage{amscd}
\usepackage{verbatim}
\usepackage{bm}
\usepackage[all]{xy}
\usepackage{hyperref}
 \usepackage{dsfont, xtab, stmaryrd}
 \usepackage{diagbox, caption}
\numberwithin{equation}{section}

\theoremstyle{plain}
\newtheorem{theorem}{Theorem}[section]

\newtheorem{lemma}[theorem]{Lemma}
\newtheorem{proposition}[theorem]{Proposition}

\newtheorem{conjecture}[theorem]{Conjecture}
\theoremstyle{definition}

\newtheorem{remark}[theorem]{Remark}

\theoremstyle{remark}

\newcommand{\OO}{\mathcal O}
\newcommand{\A}{\mathbb{A}}
\newcommand{\R}{\mathbb{R}}
\newcommand{\G}{\mathbb{G}}
\newcommand{\Q}{\mathbb{Q}}
\newcommand{\Z}{\mathbb{Z}}
\newcommand{\N}{\mathbb{N}}
\newcommand{\C}{\mathbb{C}}

\renewcommand{\H}{\mathbb{H}}
\newcommand{\F}{\mathbb{F}}
\newcommand{\D}{\mathbb{D}}

\newcommand{\diag}{\operatorname{diag}}
\newcommand{\ord}{\operatorname{ord}}

\newcommand{\GSpin}{\operatorname{GSpin}}
\newcommand{\Gspin}{\operatorname{GSpin}}

\newcommand{\SL}{\operatorname{SL}}
\newcommand{\cha}{\operatorname{Char}}

\newcommand{\pmat}[4]{\begin{pmatrix}
                 #1 & #2\\
                 #3 & #4
\end{pmatrix}}
\newcommand{\smat}[4]{\left(\begin{smallmatrix}
                 #1 & #2\\
                 #3 & #4
\end{smallmatrix}\right)}
\newcommand{\kzxz}[4]{\left(\begin{smallmatrix} #1 & #2 \\ #3 & #4\end{smallmatrix}\right) }

\newcommand{\oo}{\mathbf{0}}
\newcommand{\lp}{\left (}
\newcommand{\rp}{\right )}
\newcommand{\Ac}{{\mathcal{A}}}

\newcommand{\Zb}{\mathbb{Z}}
\newcommand{\Qb}{\mathbb{Q}}
\newcommand{\SO}{{\mathrm{SO}}}
\newcommand{\GL}{{\mathrm{GL}}}
\newcommand{\af}{\mathfrak{a}}

\newcommand{\ef}{\mathfrak{e}}
\newcommand{\ff}{\mathfrak{f}}
\newcommand{\wf}{\mathfrak{w}}
\newcommand{\uf}{\mathfrak{u}}
\newcommand{\vf}{\mathfrak{v}}
\newcommand{\pf}{\mathfrak{p}}
\newcommand{\Pf}{\mathfrak{P}}

\newcommand{\Nm}{{\mathrm{Nm}}}
\newcommand{\Ab}{\mathbb{A}}
\newcommand{\Nb}{\mathbb{N}}
\newcommand{\Hb}{\mathbb{H}}
\newcommand{\Rb}{\mathbb{R}}
\newcommand{\Cb}{\mathbb{C}}

\newcommand{\ebf}{{\mathbf{e}}}

\newcommand{\kk}{E}
\newcommand{\tr}{\operatorname{Tr}}

\newcommand{\norm}{\operatorname{N}}
\newcommand{\Gal}{\operatorname{Gal}}
\newcommand{\tth}{\textsuperscript{th }}
\newcommand{\dd}{\mathrm{d}}

\newcommand{\Kd}{K_d}
\newcommand{\fff}{\operatorname{if }}
\newcommand{\Cl}{\operatorname{Cl}}
\newcommand{\Ind}{\operatorname{Ind}}

\begin{document}

\title[On a Conejcture of Yui and Zagier]{On a Conjecture of Yui and Zagier}

\author[Y.~ Li]{Yingkun Li}
\author[T.~Yang]{Tonghai Yang}

\address{Fachbereich Mathematik,
Technische Universit\"at Darmstadt, Schlossgartenstrasse 7, D--64289
Darmstadt, Germany}
\email{li@mathematik.tu-darmstadt.de}

\address{Department of Mathematics, University of Wisconsin Madison, Van Vleck Hall, Madison, WI 53706, USA}
\email{thyang@math.wisc.edu} \subjclass[2000]{11G15, 11F41, 14K22}

\thanks{The first author is supported by the LOEWE research unit USAG.
The second author is partially supported by NSF grant DMS-1762289.}

\subjclass[2000]{14G40, 11F67, 11G18}

\date{\today}

\begin{abstract}
In this paper, we prove the conjecture of Yui and Zagier concerning the factorization of the resultants of minimal polynomials of Weber class invariants. The novelty of our approach is to systematically express differences of certain Weber functions as products of Borcherds products.
\end{abstract}

\maketitle

\section{Introduction}

In  his book \cite{Weber}, Weber proved the following well-known theorem in the theory of complex multiplication.
For a fundamental discriminant $d < 0$, let $\mathcal O_d=\Z[\theta]$ be the ring of integers of an imaginary quadratic field $\Kd=\Qb(\sqrt d)$.
Then the CM value of the famous $j$-invariant $j(\tau)$ at $\tau = \theta$ is an algebraic integer generating the Hilbert class field of $\Kd$.
The number $j(\theta)$ is called \textit{singular moduli} and plays an important role in the arithmetic of CM elliptic curves \cite{GZ}.
In the same book, Weber also considered some special modular functions $h$ of higher levels and observed that some of their CM values $h(\theta)$ still generate the Hilbert class field of $\Kd$ (for some choices of $\theta$), not the larger class fields as expected for general $h$.

These amusing observations were later studied by various authors, see for example \cite{Birch},  \cite{YZ}, and \cite{Gee}. In particular,  Gee gave a systematic proof of these facts using Shimura's reciprocity law. One of them concerns with the CM values of the three classical Weber functions of level $48$, which are defined by the following quotients of $\eta$-functions
\begin{align}\label{eq:Weber}
\mathfrak f(\tau) &:=\zeta_{48}^{-1} \frac{\eta(\frac{\tau+1}2)}{\eta(\tau)}=q^{-\frac{1}{48}} \prod_{n=1}^\infty (1+q^{n-\frac{1}2}), \notag
\\
\mathfrak f_1(\tau) &:=\frac{\eta(\frac{\tau}2)}{\eta(\tau)} = q^{-\frac{1}{48}}\prod_{n=1}^\infty (1-q^{n-\frac{1}2}),
\\
\mathfrak f_2(\tau) &:= \sqrt 2 \frac{\eta(2\tau)}{\eta(\tau)}  = \sqrt 2  q^{\frac{1}{24}} \prod_{n=1}^\infty (1+q^n). \notag
\end{align}
Together, they form a 3-dimensional, vector-valued modular function for $\SL_2(\Z)$ (see \ref{eq:Fd}).
In fact, the same holds for integral powers of these modular functions (see \cite[pg.\ 50]{Mi07}).
Furthermore, $\mathfrak f_2$ is a modular function for $\Gamma_0(2)$ with character $\chi$ of order $24$:
\begin{equation} \label{eq:chi}
\mathfrak f_2 (\gamma \tau) = \chi(\gamma) \mathfrak f_2(\tau),  \quad \gamma \in  \Gamma_0(2).
\end{equation}
The kernel of $\chi$, denoted by $\Gamma_\chi\subset \Gamma_0(2)$, is a congruence subgroup containing $\Gamma(48)$ (see \eqref{eq:Gamma24}).
In \cite{YZ}, Yui and Zagier studied the CM values of these modular functions.
The starting point of their work is the following result.

\begin{proposition}[\cite{YZ} Proposition]
\label{prop:YZ}
Let $d<0$ be a discriminant satisfying
\begin{equation} \label{eqd}
d \equiv 1 \bmod 8,  \hbox{  and }  3\nmid d.
\end{equation}
Denote $\varepsilon_d := (-1)^{(d-1)/8}$.
For each ideal $\mathfrak a=[a, \frac{-b+\sqrt d}2]$ of the order $\mathcal O_d := \Zb[\tfrac{1 + \sqrt{d}}{2} ]$ with $a>0$, let $\tau_{\mathfrak a} =\frac{-b+\sqrt{d}}{2a}$
be the associated CM point and
\begin{equation}
  \label{eq:classinv}
f(\mathfrak a) =\begin{cases}
 \zeta_{48}^{b(a-c-ac^2)}  \mathfrak f(\tau_{\mathfrak a} ), &\fff 2|(a, c),
 \\
  \varepsilon_d \zeta_{48}^{b(a-c-ac^2) } \mathfrak f_1(\tau_{\mathfrak a} ), &\fff 2|a, 2\nmid c,
  \\
  \varepsilon_d \zeta_{48}^{b(a-c+ac^2) } \mathfrak f_2(\tau_{\mathfrak a} ), &\fff 2\nmid a, 2|c.
  \end{cases}
\end{equation}
Then $f(\mathfrak a)$ is an algebraic integer depending only on the class of $\mathfrak a$ in the class group $\mathrm{Cl}(d)$ of $\mathcal{O}_d$, i.e.\ it is a class invariant. Moreover, $H_d := \Kd(f(\mathfrak a)) = \Kd(j(\tau_\mathfrak{a}))$ is the ring class field of $K_d$ corresponding $\mathcal{O}_d$.
\end{proposition}

\begin{remark}
  \label{rmk:BQF}
The class invariant in \cite{YZ} was defined using binary quadratic forms. It is a standard procedure to go between these and ideals in quadratic fields (see e.g.\ \cite{Cox}).
\end{remark}

\begin{remark}
  \label{rmk:Gee}
The sign $\varepsilon_d$ in the definition of $f(\af)$ ensures that the class invariants behave nicely under the action of the Galois group. In particular when $d< 0$ is fundamental,
\begin{equation}
  \label{eq:Galois}
\sigma_{\af_2}(  f(\af_1)) = f(\af_1 \af_2^{-1})
\end{equation}
for any $\mathcal{O}_d$-ideals $\af_1, \af_2$, where $\sigma_\af \in \Gal(H_d/K_d)$ is associated to the ideal class $[\af] \in \Cl(d)$ by Artin's map.
This was conjectured in \cite{YZ} and proved in \cite[Prop.\ 22]{Gee}.
\end{remark}
This class invariant is much better than the singular moduli in the sense that its minimal polynomial (class polynomial) has much smaller coefficients.
This gives a generator of the Hilbert class field with small height, which is crucial in the speed of elliptic curve primality test \cite{AM93}.
For example, according to \cite{YZ},  the minimal polynomial of $j(\frac{1+\sqrt{-55}}2)$ is
$$
x^4 + 3^3 5^3 29 \cdot 134219 x^3 - 3^7 5^3 23\cdot 101\cdot 32987x^2+ 3^9 5^7 11^2 83 \cdot 101 \cdot 110641 x- 3^{12} 5^6 11^3 29^3 41^3,
$$
while the minimal polynomial of $f(\mathcal O_{-55})$ is simply
$$
x^4 +x^3 -2x -1.
$$
In \cite{YZ}, Yui and Zagier made conjectures about the prime factorizations of the discriminants and resultants of such polynomials. The goal of this paper is to prove the conjecture about the factorizations of the resultants, which also clears the path to prove the conjecture about the discriminant (see Remark \ref{rmk:smallCM}).

For two co-prime, fundamental discriminants $d_1$ and $d_2$, Gross and Zagier proved in \cite{GZ} a beautiful factorization formula for the resultant of the class polynomials of $j(\frac{d_1+\sqrt{d_1}}2)$ and $j(\frac{d_2+\sqrt{d_2}}2)$, which is the norm of the difference $j(\frac{d_1+\sqrt{d_1}}2) -j(\frac{d_2+\sqrt{d_2}}2)$.
When $\left(\frac{d_1d_2}p\right) \ne -1$, set
$$
\epsilon(p) =\begin{cases}
   \left(\frac{d_1}p\right) &\fff p \nmid d_1,
   \\
    \left(\frac{d_2}p\right) &\fff p \nmid d_2.
    \end{cases}
$$
Define in general $\epsilon(n)=\prod_{p|n} \epsilon(p)^{\ord_p(n)}$ where $\ord_p(n)$ is the power of $p$ dividing $n$. For a positive integer $m$, if $\epsilon(m)=-1$,
define
\begin{equation}
  \label{eq:Ff}
  \mathfrak{F}(m) = \prod_{\substack{n n' = m\\ n, n' > 0}} n^{\epsilon(n')} \in \Nb,
\end{equation}
which is always a prime power. If $\epsilon(m) =1$ or is not defined, define $\mathfrak{F}(m)=1$.
The result of Gross and Zagier can be stated as follows.

\begin{theorem}[Theorem 1.3 in \cite{GZ}]
  \label{thm:GZ}
Let $d_1, d_2 < 0$ be co-prime, fundamental discriminants, and $w_j = |\mathcal{O}_{d_j}^\times|$. In the notations above, we have
\begin{equation}
  \label{eq:GZ}
J(d_1, d_2)^2 := \prod_{[\mathfrak a_j] \in \Cl(d_j),~ j = 1, 2}  |j(\tau_{\mathfrak a_1}) - j(\tau_{\mathfrak a_2})|^{8/(w_1 w_2)} =
\prod_{\substack{m \in \Nb,~ a\in \Zb\\ a^2 + 4m = d_1d_2}} \mathfrak{F}(m).
\end{equation}
\end{theorem}

Inspired by this beautiful formula, Yui and Zagier gave a conjectural formula of the resultant of the minimal polynomials of the Weber class invariants defined above and provided numerical evidences in \cite{YZ}.
This conjecture was originally given using two tables with totally 48 entries (see pg.\ 1653 of \cite{YZ}), but can be simplified and formulated in the following elegant way (see e.g.\ ($14_?$) in \cite{YZ} for $d_1 \equiv d_2 \equiv 1 \bmod{24}$).

\begin{conjecture}[$(14_?)$ in \cite{YZ}]
  \label{conj:YZ}
Let $d_1, d_2$ be co-prime, fundamental discriminants satisfying \eqref{eqd} and $s \mid 24$.
Define the constant
\begin{equation}
  \label{eq:kappa3}
 \kappa_3(s) :=
 \begin{cases}
  \tfrac{1}{2},& \text{ if } (\tfrac{d_1}{3})  = (\tfrac{d_2}{3}) = -1 \text{ and } 3 \mid s, \\
1, & \text{ otherwise,}
 \end{cases}
\end{equation}
which only depends on $d_1, d_2$ and $s$.
 Then
\begin{equation}
  \label{eq:YZconj}
f_s(d_1, d_2) :=   \prod_{[\mathfrak a_j] \in \Cl(d_j),~ j = 1, 2}  |f(\mathfrak a_1)^{24/s} - f(\mathfrak a_2)^{24/s}| =  \prod_{\substack{m, a \in \Nb,~ r \mid s\\ a^2 + 16mr^2 = d_1d_2\\ m \equiv 19( d_1 + d_2 - 1) \bmod{s/r}}} \mathfrak{F}(m)^{\kappa_3(s)}.
\end{equation}
\end{conjecture}

\begin{remark}
  \label{rmk:divisibility}
Because of the relation $j(\tau) = (\mathfrak{f}_2^{24}(\tau) - 16)^3/\mathfrak{f}^{24}_2(\tau)$, we know that $f_s(d_1, d_2) \mid J(d_1, d_2)$ for any co-prime, fundamental discriminants $d_1, d_2$ satisfying \eqref{eqd}.
Since the invariants are algebraic integers, it is also clear that $f_{s'}(d_1, d_2) \mid f_{s}(d_1, d_2)$ for any $s \mid s' \mid 24$.
The conjecture above also reflects such divisibilities since $\mathfrak{F}(m/r^2) \mid \mathfrak{F}(m)$ for all $m, r \in \Nb$ (see e.g.\ the explicit formula of $\mathfrak{F}(m)$ on pg.\ 1651 of \cite{YZ}).
\end{remark}

When $s = 1$, it was suggested in \cite{YZ} that one can try to prove this conjecture by adapting the analytic approach in \cite{GZ} with $\SL_2(\Zb)$ replaced by $\Gamma_0(2)$.
This was later carried out in \cite{Ros03}.
More recently, Yang and Yin gave another analytic proof of the conjecture for $s = 1$ in \cite{YY19}, where the new ingredients are Borcherds' regularized theta lift in \cite{Borcherds98} and the big CM formula in \cite{BKY}.
Although the spirits of the approaches are the same, the one in \cite{YY19} is conceptually easier to understand and opens the door to attack the conjecture for $s > 1$.
In this paper, we complete the proof of the conjecture for all $s \mid 24$.

\begin{theorem}
  \label{thm:intro}
Conjecture \ref{conj:YZ} is true for every $s \mid 24$.
\end{theorem}

For $s = 1$, the proof of Theorem \ref{thm:intro} in \cite{YY19} consists of three steps:
\begin{enumerate}
\item Relate $\mathfrak f_2(z_1)^{24}-  \mathfrak f_2(z_2))^{24}$ to a Borcherds product on the Shimura variety associated to the rational quadratic space $(M_2(\Q),  \det)$.   \footnote{The Shimura variety is just the product of two modular curves in this case.}
\item View a pair of CM points $(\tau_1, \tau_2)$ as a big CM point on this Shimura variety in the sense of \cite{BKY}.  Apply the big CM value formula \cite[Theorem 5.2]{BKY} and express the CM value in terms of Fourier coefficients of incoherent Eisenstein series.
\item Compute the Fourier coefficients in Step (2) and obtain the formula. This is a local calculation.
\end{enumerate}
In the first step for $s = 1$, one can find a vector-valued modular function $\tilde{F}_1$ and identify $\mathfrak{f}_2(z_1)^{24} - \mathfrak{f}_2(z_2)^{24}$ with the Borcherds product $\Psi(z_1, z_2, \tilde{F}_1)$ associated to $\tilde{F}_1$.
Note $\ff_2(z)^{24} =2^{12} \frac{\Delta(2z)}{\Delta(z)}$ is a Hauptmodul of $\Gamma_0(2)$, and the Borcherds product $\Psi(z_1, z_2, \tilde{F}_1)$ is well-known in the literatures on VOA and moonshine (see e.g.\ \cite{Borcherds92}, \cite{Sch08}).
In the second step, one suitably identifies the Galois orbit of CM points with the toric orbit of big CM points, and apply Theorem 5.2 in \cite{BKY}.
This reduces the proof to the third step, where the local calculations have been completed in many special cases (see \cite{Yang05}, \cite[Section 4.6]{HYbook} and \cite{KY10}) and the most general result can be found in Appendix A of \cite{YYY}.

To execute this strategy for $s > 1$, we first need to relate $\mathfrak{f}_2(z_1)^{24/s} - \mathfrak{f}_2(z_2)^{24/s}$ to Borcherds product.
Since the function $\mathfrak{f}_2(z)^{24/s}$ is invariant with respect to $\Gamma_{\chi, s} := \langle \Gamma_\chi, T^s\rangle \supset \Gamma(2s)$, one would hope to find the analog of $\tilde{F}_1$ in $M^!(\omega_s)$, with $\omega_s$ the Weil representation of $\SL_2(\Zb)$ on the finite quadratic module associated to the lattice $L_s$ (see \eqref{eq:Ld}), which is the same as the lattice used in \cite{YY19} to produce $\Psi(z_1, z_2, \tilde{F}_1)$, but with the quadratic form scaled by $s$.
We have computationally decomposed the representation $\omega_s$ and analyzed the space of vector-valued modular functions. To our surprise, there is NO modular function whose Borcherds product equals to $(\mathfrak{f}_2(z_1)^{24/s} - \mathfrak{f}_2(z_2)^{24/s})^s$! Our new idea then is to express $(\mathfrak{f}_2(z_1)^{24/s} - \mathfrak{f}_2(z_2)^{24/s})^s$ as \textit{a product of Borcherds products}, which works out beautifully.

   \begin{theorem}[Theorems \ref{theo:BorcherdsLifting} and  \ref{thm:Blift}] \label{theo:factorization}
For every $\dd \mid 24$, there is a vector-valued modular function $\tilde F_\dd \in M^!(\omega_\dd)$ with associated Borcherds product $\Psi_\dd(z_1, z_2) := \Psi(z_1, z_2, \tilde F_\dd)$ such that
\begin{equation}
  \label{eq:product2}
  (\mathfrak f_2(z_1)^{24/s} - (\varepsilon \mathfrak f_2(z_2))^{24/s})^s = \prod_{\dd \mid s} \Psi_\dd(z_1, z_2)^{\varepsilon^{\frac{24}\dd}},
\end{equation}
for every $s \mid 24$ and any $\varepsilon =\pm 1$.
   \end{theorem}

   \begin{remark}
The index $r \mid s$ in the summation on the right hand side of \eqref{eq:YZconj} is NOT directly related to the index $\dd \mid s$ in the product above! Instead, it comes out of local calculation in Section \ref{sec:local}.
   \end{remark}

   \begin{remark}
     Each Borcherds product $\Psi_\dd(z_1, z_2)$ comes from a different quadratic space depending on $\dd$, and is a meromorphic function on the Shimura variety $X_\dd^2$, which admits a natural covering map from $X_s^2$ when $\dd \mid s$ (see Section \ref{sect:Borcherds}). One can then pull back $\Psi_\dd$ to a function on $X_s^2$.
Notice that this decomposes the divisor of the LHS, which is a Heegner divisor on $X_s^2$, into a sum of pullbacks of Heegner divisors on $X_\dd^2$ with $\dd \mid s$.
When $s > 1$, the product $\prod_{\dd|s} \Psi_\dd(z_1, z_2)$ is itself \textit{not} a single Borcherds product on $X_s^2$.

\end{remark}
\begin{remark}
Theorem \ref{theo:factorization} naturally leads one to speculate a generalization of the converse theorem in \cite{Br14}, i.e.\ every Heegner divisor on an orthogonal Shimura variety associated to a lattice of signature $(n, 2)$ with Witt rank greater than or equal to 2 should be the divisor of a product of Borcherds products.
   \end{remark}

To arrive at this idea, we took $s = 2$ and started from the simple observation that
\begin{equation}
  \label{eq:observation}
(\ff_2(z_1)^{12} - \ff_2(z_2)^{12})^2 = (\ff_2(z_1)^{24} - \ff_2(z_2)^{24}) \cdot \frac{\ff_2(z_1)^{12} - \ff_2(z_2)^{12}}{\ff_2(z_1)^{12} + \ff_2(z_2)^{12}}.\end{equation}
We already know that the first factor on the RHS is a Borcherds product. If we can realize the second factor as a Borcherds product, then the LHS would be a product of Borcherds products (with different quadratic forms).
To do that, we can read off the divisor of the second factor, and deduce the principal part of the input to Borcherds' lift.
In this case, it is of the form $q^{-1/2} \uf_2$ for a suitable vector $\uf_2$ in a finite dimensional vector space $\Cb[\Ac_2]$, where $\SL_2(\Zb)$ acts via the Weil representation $\omega_2$ (see Section \ref{subsec:Lattice} for details). Then we find the irreducible representation in $\omega_2$ containing $\uf_2$, which is 3-dimensional, and hope to find the suitable vector-valued modular function $\tilde{F}_2$ with this principal part. Miraculously, this function exists and its three components are the $(-24/2)$\tth power of the three Weber functions $\frac{\ff}{\sqrt{2}}, \frac{\ff_1}{\sqrt{2}}, \frac{\ff_2}{\sqrt{2}}$.

The observation \eqref{eq:observation} generalizes to any $s \mid 24$ by substituting $X = \lp \varepsilon \frac{\ff_2(z_2)}{\ff_2(z_1)} \rp^{24/s}$ into the following simple identity in $\Q(X)$
\begin{equation}
  \label{eq:ratid}
  (1 - X)^s = \prod_{\dd \mid s} \prod_{b \mid \dd} (1 - X^{s/b})^{b \cdot \mu(\dd/b)},
\end{equation}
where $\mu$ is the M\"obius function, and multiplying by $\ff_2(z_1)^{24}$ on both sides.
Note that the identity in \eqref{eq:ratid} holds for any $s \in \N$ (see Lemma \ref{lemma:pol}).
Then the miracle continues to happen, and we find a family of vectors $\{\uf_\dd: \dd \mid 24\}$ (see \eqref{eq:ufd}) and vector-valued modular functions $\tilde{F}_\dd = q^{-1/\dd} \uf_\dd + O(q^{1/(2\dd)})$ producing the Borcherds lifts $\Psi_\dd$ (see \eqref{eq:Fd} and Remark \ref{rmk:Fd}).
For $\dd > 1$, the vector $\uf_\dd$ satisfies nice invariance properties (see Prop.\ \ref{prop:udinv}) and is of independent interest, whereas the components of $\tilde{F}_\dd$ are simply the $(-24/\dd)$\tth power of the three Weber functions $\frac{\ff}{\sqrt{2}}, \frac{\ff_1}{\sqrt{2}}, \frac{\ff_2}{\sqrt{2}}$!

\begin{remark} The $\varepsilon =\pm 1$ in Theorem \ref{theo:factorization} is there for a good reason.  To prove the Yui-Zagier conjecture, we need to choose $\varepsilon = \varepsilon_{d_1}\varepsilon_{d_2} = (-1)^{\frac{d_1 +d_2 -2}8}$ (see Proposition \ref{prop:CMFormula} and its proof).
It is also amusing to see that the same $\varepsilon$ appears when we calculate the Fourier coefficients of derivatives of certain Einsenstein series (see Theorem \ref{thm:main1}).

\end{remark}

To complete the proof, we can now apply the second step to each Borcherds product, obtain a big CM value formula, and add them together.
Note that the identification of the Galois orbit of $(\tau_{\af_1}, \tau_{\af_2})$ used in defining $f_s(d_1, d_2)$ with the big CM cycle in \cite{BKY} depends on the input $\tilde{F}_\dd$ in step one. Therefore, it is not a priori clear that this will work out. We prove this in Proposition \ref{prop:CMFormula}, which crucially depends on Lemma \ref{lemma:compact}.
This unexpected result was first observed with some computer calculations, and has been reduced to a computation with finite groups in $\GL_2(\Z/3\Z)$ and $\GL_2(\Z/16\Z)$.
Finally, we apply the local calculations in \cite{YYY} to finish off Step (3).

\begin{remark}
  \label{rmk:smallCM}
With Theorem \ref{theo:factorization}, one can now replace the big CM value formula in \cite{BKY} with the small CM value formula in \cite{Schofer} to prove the conjectural factorization of the discriminant of the minimal polynomials of the Weber invariants in \cite{YZ}.
In fact, Yui and Zagier wrote the discriminant as the norm of an algebraic integer in the Hilbert class field and gave a conjectural prime factorization of the ideal generate by it.
This refined conjecture can be studied by analyzing the arithmetic intersection calculations in \cite{AGHMP17}, which are generalizations of the algebraic approach in \cite{GZ}. Theorem \ref{theo:factorization} tells us how to decompose the divisor of modular function $(\ff_2(z_1)^{24/s} - (\varepsilon \ff_2(z_2))^{24/s})^s$ into sums of Heegner divisors, which will be vital in attacking this refined conjecture in \cite{YZ}. We plan to carry these out as a sequel to this work.
\end{remark}

   This paper is organized as follows. After setting up notation and defining basic terms in Section \ref{sect2}, we study in Section \ref{sect3} the action of certain subgroup $H_\dd' \subset \SO(L_\dd)/\Gamma_{L_\dd}$ on the finite quadratic module $\Ac_\dd := L^\vee_\dd/L_\dd$ and use it to decompose the Weil representation $\omega_\dd$ of $\SL_2(\Z)$ on $\C[\Ac_\dd]^{H_\dd'}$. The goal and main result is to construct certain element $\uf_\dd \in \C[\mathcal A_\dd]^{H_\dd'}$ satisfying (\ref{eq:ufdinv}).
This vector generates a 3-dimensional, $H'_\dd$-invariant subrepresentation of $\omega_\dd$, and will be crucial in finding the input $\tilde{F}_\dd$ that produces the Borcherds product $\Psi_\dd$.
In Section \ref{sect:Borcherds}, we view product of two modular curves as a Shimura variety of orthogonal type $(2, 2)$ associated to $L_\dd$, construct the Borcherds product $\Psi_\dd$, and prove Theorem \ref{theo:factorization}.  In Section \ref{sect:BigCM}, we view the pair $(\tau_{\mathfrak a_1}, \tau_{\mathfrak a_2})$ as a big CM point on the product of two modular curves and study its Galois orbit. The upshot is Prop.\ \ref{prop:CMFormula}, which relates the left hand side of Conjecture \ref{conj:YZ} to the big CM value of Borcherds products. By the second step of strategy, Conjecture \ref{conj:YZ} is reduced to local calculation of certain Eisenstein series and its derivative, which we carry out in Section \ref{subsec:local} using the results in the appendix of \cite{YYY}.
Finally in the appendix, we explicitly write down the cosets in the finite quadratic module used in constructing the Borcherds products, and include a numerical example for $d_1 = -31$ and $d_2 = -127$.
\vspace{.2in}

\textit{Acknowledgement}: We thank Hongbo Yin for a careful reading and helpful comments of an earlier version of this paper.

\section{Preliminaries} \label{sect2}
\subsection{Weil Representation.}
Let $(L, Q)$ be an even integral lattice of signature $(2, 2)$ and $V := L \otimes \Qb$ the rational quadratic space.
Denote $L'$ the dual lattice and $\Ac_L := L'/L$ the finite quadratic module.
The group $\SL_2(\Zb)$ acts on $U_L:= \Cb[\Ac_L]$ via the Weil representation $\omega_L$ given by
\begin{equation}
  \label{eq:omegaL}
  \omega_L(T)\ef_h = \ebf(-Q(h)) \ef_h,~ \omega_L(S) \ef_h = \frac{1}{\sqrt{|\Ac_L|}} \sum_{\mu \in \Ac_L} \ebf((\mu, h)) \ef_\mu,
\end{equation}
where $\{\ef_\mu: \mu \in \Ac_L\}$ is the standard basis of $U_L$ and
\begin{equation}
  \label{eq:TS}
    T:= \pmat{1}{1}{0}{1},\quad  S:= \pmat{0}{-1}{1}{0}.
\end{equation}
Note that this differs from the convention of Borcherds by complex conjugation.

Let $S(L)=\oplus_{\mu \in L'/L} \phi_\mu \subset S(V \otimes \A_f)$ with $\hat{L} = L \otimes \hat{\Zb}$ and
$$
\phi_\mu=\cha(\mu+\hat L).
$$
Under the isomorphism $U_L \to S(L)$ that maps $\ef_\mu$ to $\phi_\mu$, the representation $\omega_L$ becomes the restriction of the Weil representation $\omega = \omega_{V, \psi}$ (with the usual idelic character $\psi$ of $\Qb$) from $\SL_2(\Ab)$ to (the diagonally embedded) $\SL_2(\Zb)$.
We will sometimes switch the representation spaces between $U_L$ and $S(L)$.

\subsection{Weber Functions.}
For any finite dimensional, $\C$-representation $\rho: \Gamma \to V$ of a finite index subgroup $\Gamma \subset \SL_2(\Z)$, denote $M^!(\rho, \Gamma)$ the space of weakly holomorphic, vector-valued modular function with respect to $\rho$. We drop $\rho$, resp.\ $\Gamma$, from the notation if $\rho$ is trivial, resp. $\Gamma=\SL_2(\Z)$. For example, the three Weber functions defined by (\ref{eq:Weber}) form a vector-valued modular function
$\begin{pmatrix}
  \ff_2 \\ \ff_1 \\ \ff
\end{pmatrix} \in M^!(\overline{\varrho_{24}})$. Here, for a positive integer $\dd$ and $j \in (\Zb/2\dd\Zb)^\times$, the representation $\varrho_{\dd, j}:\SL_2(\Zb) \to \GL_3(\Cb)$ is defined by
\begin{equation}
  \label{eq:varrhod}
  \varrho_{\dd, j}(T) :=
  \begin{pmatrix}
    \zeta_{\dd}^{-j} & 0 & 0 \\
0 & 0 & \zeta_{2\dd}^{j}\\
0 & \zeta_{2\dd}^{j}  & 0
  \end{pmatrix},~
  \varrho_{\dd, j}(S) :=
  \begin{pmatrix}
0 & 1 & 0 \\
1 & 0 & 0 \\
0 & 0 & 1
  \end{pmatrix}.
\end{equation}
We simply write $\varrho_\dd$ for $\varrho_{\dd, 1}$. Finally, $\bar{\rho}(g)= \overline{\rho(g)}$. Later, the modular function
\begin{equation}
  \label{eq:Fd}
  F_\dd(\tau) :=
  \sqrt{2}^{24/\dd}
\begin{pmatrix}
  \ff_2^{-24/\dd}(\tau) \\ \ff_1^{-24/\dd}(\tau) \\ \ff^{-24/\dd}(\tau)
\end{pmatrix}
\in M^!(\varrho_\dd),~ \dd \mid 24
\end{equation}
will play an important role for us as the representation $\varrho_\dd$ defined above is a subrepresentation of certain Weil representation that we will consider.
\begin{remark}
  \label{rmk:FC}
For convenience later, we will denote
\begin{equation}
  \label{eq:adn}
  \sqrt{2}^{24/\dd} \ff_2^{-24/\dd}(\tau) =  \lp \frac{\eta(\tau)}{\eta(2\tau)}\rp^{24/\dd}  = \sum_{l \ge -1,~ l \equiv -1 \bmod{\dd}} c_\dd(l) q^{l/\dd}\in q^{-1/\dd}\Zb\llbracket q \rrbracket.
\end{equation}
Clearly $c_\dd(-1) = 1$ for every $\dd \mid 24$.
We will also denote $c_{-1}(l)$ the $l$\tth Fourier coefficient of $2^{12}\ff_1^{-24}(2\tau) = \ff_2^{24}(\tau) = 2^{12} \frac{\Delta(2\tau)}{\Delta(\tau)}$.
\end{remark}

Let $\chi: \Gamma_0(2) \to \Cb^\times$ be the character defined in \eqref{eq:chi}.
On the generators $T, S^2$ and $TB$ of $\Gamma_0(2)$, where
\begin{equation}
  \label{eq:B}
B:=ST^2S^{-1} = \pmat{1}{0}{-2}{1}, 
\end{equation}
the character $\chi$ is explicitly given by
\begin{equation}
  \label{eq:chi0}
  \chi(T) = \zeta_{24},  ~ \chi(S^2)=1, \, \hbox{ and }~  \chi(TB) = 1.
\end{equation}

The kernel of $\chi$ is a normal subgroup of $\Gamma_0(2)$ defined by
\begin{equation}
  \label{eq:Gamma24}
  \Gamma_\chi := \langle \Gamma_0(2)^{\mathrm{der}}, T^{24}, S^2,    TB \rangle \subset \Gamma_0(2),
\end{equation}
where $\Gamma_0(2)^{\mathrm{der}}$ is the derived subgroup of $\Gamma_0(2)$.  We remark that $\Gamma_\chi$ is the group $\Phi_0^0(24)$ in \cite{YY16}.
Furthermore, it contains the congruence subgroup $\Gamma_0(48) \cap \Gamma(24) $ and $\Gamma_0(2) / \Gamma_\chi \cong \Z/24$.
More generally, for any divisor $\dd \mid 24$, denote the kernel of $\chi^{24/\dd}$ by
\begin{equation}
  \label{eq:Gammachid}
  \Gamma_{\chi, \dd} := \langle \Gamma_{\chi}, T^{\dd} \rangle  \subset \Gamma_0(2).
\end{equation}
It has index $\dd$ in $\Gamma_0(2)$ and contains $\Gamma_{\chi} = \Gamma_{\chi, 24}$, as well as the congruence subgroup
\begin{equation}
  \label{eq:Gammad}
  \Gamma_\dd := \Gamma_1(2\dd) \cap \Gamma(\dd).
\end{equation}
In particular, $\Gamma_0(2) = \Gamma_{\chi, 1}$.
More generally for $\dd \mid \dd' \mid 24$, we have $\Gamma_{\dd} \supset \Gamma_{\dd'}$.
For future convenience, we also write $\dd_p$ for the $p$-part of $\dd$.
Then clearly $\dd = \dd_2 \dd_3$.

\section{Decomposition of Weil Representations }  \label{sect3}
\subsection{Lattice}
\label{subsec:Lattice}
For a divisor $\dd \mid 24$, consider the quadratic lattice
\begin{equation}
  \label{eq:Ld}
L_\dd = \left\{
\lambda = \pmat{\lambda_{00}}{\lambda_{01}}{2\lambda_{10}}{\lambda_{11}}:  \lambda_{ij} \in \Zb
\right\},~ Q_d(\lambda) := \dd \det(\lambda).
\end{equation}
The dual lattice is given by
\begin{equation}
  \label{eq:Ld'}
L_\dd' = \left\{ \lambda = \frac{1}{\dd} \pmat{\lambda_{00}}{\lambda_{01}/2}{\lambda_{10}}{\lambda_{11}}:  \lambda_{ij} \in \Zb   \right\}.
\end{equation}
The finite quadratic module $L_\dd' /L_\dd$ is then isomorphic to
\begin{equation}
  \label{eq:Acd}
  \Ac_\dd:= \left\{ h = [h_{0},h_{1},h_{2},h_{3}]:
        h_{0}, h_3 \in \Zb/\dd \Zb, h_1, h_2 \in \Zb/(2\dd \Zb) \right\},
    \end{equation}
    where the isomorphism is fixed throughout and given by
    \begin{equation}
      \label{eq:Acdiso}
      \begin{split}
    L_\dd' /L_\dd       &\cong         \Ac_\dd\\
      \frac{1}{\dd} \pmat{\lambda_{00}}{\lambda_{01}/2}{\lambda_{10}}{\lambda_{11}} + L_\dd
      &\mapsto \left[ \lambda_{00}, \lambda_{01}, \lambda_{10}, \lambda_{11}\right].
  \end{split}
    \end{equation}
Via this isomorphism, the quadratic form $Q_\dd$ on $\Ac_\dd$ becomes
\begin{equation}
  \label{eq:Qd}
  Q_\dd(h) := \frac{2h_0 h_3 - h_1h_2}{2\dd} \in \frac{1}{2\dd}\Zb/\Zb
\end{equation}
for $h = [h_0, h_1, h_2, h_3] \in \Ac_\dd$.
We can now map $L_\dd$ into $L_\dd'/L_\dd \cong \Ac_\dd$ via
\begin{equation}
  \label{eq:kappad}
  \begin{split}
    \kappa_\dd : L_\dd &\to L_\dd'/L_\dd \\
    \lambda &\mapsto \frac{1}{\dd} \lambda + L_\dd
  \end{split}
\end{equation}
which is compatible with the left and right action of $\Gamma_0(2)$, i.e.
\begin{equation}
  \label{eq:kappacompatible}
g_1 \cdot  \kappa_\dd(\lambda \cdot g_2) = \kappa_\dd(g_1 \cdot \lambda \cdot g_2)= \kappa_\dd(g_1 \cdot \lambda) \cdot g_2
\end{equation}
for all $g_1, g_2 \in \Gamma_0(2)$ and $\lambda \in L_\dd$.
By viewing $\Gamma_{\chi, 1} = \Gamma_0(2)$ as a subset of $L_\dd$, we can send it to a subset in $\Ac_\dd$.
If we denote
\begin{equation}
  \label{eq:Ac0}
  \Ac_{\dd}^0 := \{[h_0, h_1, h_2, h_3] \in \Ac_\dd: h_2 = 0 \in \Zb/(2\dd\Zb)\},
\end{equation}
 it will be helpful to know the parts of $\Gamma_{0}(2)$ that land in $\Ac_\dd^0$ under $\kappa_\dd$ when we simplify the expression of Borcherds products.
For this we need the following lemma, whose proof will follow from combining the corresponding local results in Lemma \ref{lemma:j3} and \ref{lemma:j2}.
\begin{lemma}
  \label{lemma:j}
  For any $j \in \Zb/\dd\Zb$, we have (viewing $\Gamma_0(2) \subset L_\dd$)
  $$
\kappa_\dd (T^j \Gamma_{\chi, \dd} ) \cap \Ac^0_\dd = \{\dd_3^{-1}[r, r(2j + (r^2 - 1)), 0, r]: r \in (\Zb/\dd\Zb)^\times\}.
  $$
\end{lemma}
\begin{remark}
  Note that $r^3 - r \bmod {2\dd}$ is well-defined for $r \in (\Zb/\dd\Zb)^\times$ when $\dd \mid 24$. Furthermore
  $$
  r^3 - r \equiv
  \begin{cases}
 \dd \bmod{2\dd},& \text{if } 8 \mid \dd \text{ and } r \equiv \pm 3 \bmod{8}, \\
  0 \bmod{2\dd},& \text{otherwise.}
  \end{cases}
  $$
\end{remark}

We want to analyze the action of the finite orthogonal group $H_\dd:= \SO(L_\dd)/\Gamma_{L_\dd}$ on $\Ac_\dd$, which intertwines the Weil action of $(G =\SL_2(\Z), \omega_\dd = \omega_{L_\dd})$ on
\begin{equation}
  \label{eq:Ud}
U_\dd:= \C[\Ac_\dd],
\end{equation}
 and will be useful for us to decompose the Weil representation $\omega_\dd$. In particular, we want to consider a subrepresentation of $\omega_\dd$ on the subspace $U_\dd^{H'_\dd} \subset U_\dd$ fixed by a subgroup $H'_\dd \subset H_\dd$.

 Let $\GL_2(\Q) \times \GL_2(\Q)$ acts on  $V_\dd =L_\dd \otimes \Q =M_2(\Q)$ via
 $$
 (g_1, g_2)\cdot X = g_1  X g_2^{-1}.
 $$
 This action gives an identification of $\Gspin(V)$ with $H=\{ (g_1, g_2):\,  \det g_1 =\det g_2\}$, and a commutative diagram of exact sequences:
 $$
\xymatrix{
 1 \ar[r]  &\{\pm 1\} \ar[r] \ar[d] &\SL_2 \times \SL_2 \ar[r] \ar[d] &\SO(V) \ar[r] \ar[d] &1
\cr
 1 \ar[r]  &\G_m \ar[r] &H \ar[r] &\SO(V) \ar[r] &1.
 \cr
}
 $$

For the particular lattice $L_\dd$, we have
\begin{equation}
  \label{eq:disck}
  \SO(L_\dd) = \overline{\Gamma_0(2) \times \Gamma_0(2)}= \Gamma_0(2) \times \Gamma_0(2)/\{ \pm (I_2, I_2)\},~
\end{equation}
As  $\{\pm (I_2, I_2)\}$ does not matter in this paper, we will simply identify $\SO(L_\dd)$ with $\Gamma_0(2) \times \Gamma_0(2)$ and  drop the over line.  Under this identification, we have $ \Gamma_{L_\dd} = {\Gamma_{\dd} \times \Gamma_{\dd}}$, where $\Gamma_\dd$ is defined in \eqref{eq:Gammad}.
In particular, we are interested in the action of the subgroup of $\SO(L_\dd)$ generated by the images of ${(T, T)}$  and ${\Gamma_{\chi, \dd} \times \Gamma_{\chi, \dd}}$.
We let $H_\dd'$ be its image in   $H_\dd =N_\dd \times N_\dd$, where
$N_\dd := \Gamma_0(2)/\Gamma_{\dd}$. Let $ N'_\dd := \Gamma_{\chi, \dd}/\Gamma_{\dd},
$
then we have
\begin{equation}
  \label{eq:contain}
{N_\dd' \times N_\dd'}
\subset H'_\dd \subset H_\dd =
{N_\dd \times N_\dd}.
\end{equation}
Since $\Gamma_0(2)/\Gamma_{\chi, \dd} \cong \Zb/\dd\Zb$, the quotient group $H_\dd/H'_\dd$ is isomorphic to $\Z/\dd\Z$.
For prime $p$ and an abelian group $K $, we denote
$
K_p := K \otimes_\Zb \Zb_p.
$
By carrying this through product and quotient, we can also make sense of $K_p$ for $K \in \{N_{\dd}, N_\dd', H_\dd, H_\dd'\}$.
Since $\dd$ is only divisible by $2$ and $3$ in our case, the Chinese remainder theorem implies
\begin{align*}
H_\dd &\cong H_{\dd, 2} \times H_{\dd, 3},~ H'_\dd \cong H'_{\dd, 2} \times H'_{\dd, 3},
\\
 {N'_{\dd, p} \times N'_{\dd, p}}&\subset  H'_{\dd, p}\subset H_{\dd, p} = {N_{\dd, p} \times N_{\dd, p}}   .
\end{align*}
For the same reason, we have the decomposition
\begin{equation}
  \label{eq:Aciso}
\Ac_\dd \cong \Ac_{\dd, 2} \times \Ac_{\dd, 3},~ \Ac_{\dd, p} := \Ac_\dd \otimes_\Z \Z_p.
\end{equation}
Using this isomorphism, we can write $\omega_{\dd} \cong \omega_{\dd, 2} \otimes \omega_{\dd, 3}$ and $U_\dd \cong U_{\dd, 2} \otimes U_{\dd,3}$, where $\omega_{\dd, p}$ is the Weil representation of $\SL_2(\Z_p)$ acting on $U_{\dd, p}$ associated to $\Ac_{\dd, p}$.

Now, we introduce the vector $\uf_\dd \in U_\dd$.
\begin{equation}
  \label{eq:ufd}
  \begin{split}
  \uf_\dd &:= \uf_{\dd, 2} \otimes \uf_{\dd, 3} =  \sum_{j \in \Zb/\dd\Zb} a_\dd(j) \lp \sum_{h \in \kappa_{\dd}(T^j \Gamma_{\chi, \dd})} \ef_h \rp,\\
a_\dd(j) &:= \lp \sum_{s \in (\Zb/\dd\Zb)^\times} \zeta_{\dd}^{sj} \rp = \mu\lp \frac{\dd}{(\dd, j)} \rp \frac{\varphi(\dd)}{\varphi(\dd/(\dd, j))} \in \Zb,
  \end{split}
\end{equation}
where $\uf_{\dd, p} = \uf_{\dd, p}(1, \dots, 1) \in U_{\dd, p}' \subset U_{\dd, p}$ is the vector defined in \eqref{eq:ud3c} and \eqref{eq:ud2c}, $\mu$ and $\varphi$ are the M\"obius and Euler $\varphi$-function respectively.
Note that $a_\dd(j)$ is defined for any $\dd \in \Nb$ and $j \in \Zb/\dd \Zb$.

\begin{remark} A natural question is where the element $\uf_{\dd, p}$ comes from and what it is good for?
  In the next two subsections, we will give some ideas where they come from.
  For now, we are satisfied to give its nice properties as below. See Prop.\ \ref{prop:repembed} below.
\end{remark}
\begin{lemma}
  \label{lemma:ufinv}
  For any $\dd, r \mid 24$, the vector $\uf_\dd  \in \Zb[\Ac_\dd]$ is invariant with respect to $\Gamma_{\chi, r} \times \Gamma_{\chi, r}$ if and only if $\dd \mid r$.
\end{lemma}

\begin{proof}
If $\dd \mid r$, then $\Gamma_{\chi, r} \subset \Gamma_{\chi, \dd}$ and we just need to prove the case when $r = \dd$.
Let ${(g_1, g_2)} \in \SO(L_\dd)$ with $g_j \in \Gamma_{\chi, \dd}$. Then
  \begin{align*}
{    (g_1, g_2)} \cdot \uf_\dd &= \sum_{j \in \Zb/\dd\Zb} a_\dd(j) \lp \sum_{h \in \kappa_{\dd}(g_1 T^j \Gamma_{\chi, \dd} g_2^{-1})} \ef_h \rp = \uf_\dd,
  \end{align*}
  where we have used the fact that $\Gamma_{\chi, \dd}$ is normal in $\Gamma_{0}(2)$ with coset representatives $\{T^j: j \in \Zb/\dd\Zb\}$. Similarly, ${(T, T)}\cdot \uf_\dd = \uf_\dd$. Therefore $\uf_\dd$ is $\Gamma_{\chi, \dd} \times \Gamma_{\chi, \dd}$-invariant.
If $\dd \nmid r$, then $(T^r, 1) \in (\Gamma_{\chi, r}  \times \Gamma_{\chi, r} ) \backslash (\Gamma_{\chi, \dd} \times \Gamma_{\chi, \dd} )$.
It is easy to see that
  $$
{(T^r, 1)} \cdot \uf_\dd =  \sum_{j \in \Zb/\dd\Zb} a_\dd(j-r)\lp \sum_{h \in \kappa_{\dd}(T^j N'_{\dd})} \ef_h \rp \neq \uf_d
$$
since $\frac{a_\dd(-r)}{\varphi(\dd)} =  \frac{\mu(\dd/(\dd, r))}{\varphi(\dd/(\dd, r))} \neq 1 = \frac{a_\dd(0)}{\varphi(\dd)}$ when $\dd \nmid r$.
\end{proof}

\begin{proposition}
\label{prop:udinv}
  For any $\dd \mid 24$, we have
  \begin{equation}
    \label{eq:ufdinv}
   \omega_\dd(g)\uf_\dd = \chi(g)^{-24/\dd} \uf_\dd
 \end{equation}
 for all $g \in \Gamma_0(2)$.
\end{proposition}
\begin{proof}
  This follows directly from the local results \ref{prop:udinv3} and \ref{prop:decomp2} as
$$
\omega_\dd(g) \uf_\dd = 
(\omega_{\dd, 2}(g)\uf_{\dd, 2}) \otimes (\omega_{\dd, 3}(g)\uf_{\dd, 3}) =
\chi(g)^{-24(\dd_2/\dd_3 + \dd_3/\dd_2)}\uf_{\dd, 2} \otimes \uf_{\dd, 3} =
\chi(g)^{-24/\dd}\uf_{\dd}.
$$
for all $g \in \Gamma_0(2)$. Here we have used $\frac{1}{\dd} - (\frac{\dd_2}{\dd_3} + \frac{\dd_3}{\dd_2}) \in \Zb$ when $\dd \mid 24$.
\end{proof}
Now, define two further vectors
\begin{equation}
  \label{eq:wv}
  \vf_\dd := \omega_\dd(S) \uf_\dd,~ \wf_\dd := \zeta_{2\dd}^{-1} \omega_\dd(T)\vf_\dd.
\end{equation}
Note that $\uf_\dd, \vf_\dd$ and $\wf_\dd$ are linearly independent for all $\dd \mid 24$.
The key to the input of Borcherds lifting is then constructed using these vectors in the following result.

\begin{proposition}
  \label{prop:repembed}
  The representations $\varrho_\dd$ defined in \eqref{eq:varrhod} is a subrepresentation of the Weil representation $\omega_\dd$ via the map
  \begin{equation}
    \label{eq:iotad}
    \begin{split}
          \iota_\dd: \Cb^3 &\to U_\dd^{H_\dd'} \subset U_\dd\\
    \begin{pmatrix}
  a \\ b \\ c
\end{pmatrix} &\mapsto a\uf_\dd + b \vf_\dd + c\wf_\dd.
    \end{split}
  \end{equation}
  Let $F_\dd$ be the modular function defined in \eqref{eq:Fd}.
  The function $\iota_\dd \circ F_\dd$ is then in $M^!(\omega_\dd)$ and invariant with respect to the orthogonal group $H_\dd' \subset \SO(L_\dd)/\Gamma_{L_\dd}$.
  Furthermore, it has the principal part
  \begin{equation}
    \label{eq:Fdprin}
    \iota_d \circ F_\dd(\tau) = q^{-1/\dd} \uf_\dd +
    \begin{cases}
    O(q^{1/2}),& \dd > 1,\\
    O(1),& \dd = 1,\\
    \end{cases}
  \end{equation}
\end{proposition}

\begin{remark}
  \label{rmk:Fd}
  When $\dd = 1$, the function $\iota_\dd \circ F_\dd$ differs from the input in \cite{YY19} by a constant vector. 
  To simplify the notation, we will write
  \begin{equation}
    \label{eq:Fdt}
    \tilde{F}_\dd := \iota_\dd \circ F_\dd +
    \begin{cases}
      24(\ef_{(0, 0)} + \ef_{(1/2, 0)}),& \dd = 1,\\
      0, & \dd > 1,
    \end{cases}
  \end{equation}
  which is an element in $M^!(\omega_\dd)$ invariant with respect to $H_\dd'$.
\end{remark}
\begin{proof}
It suffices to check on the generators $T, S$ of $\SL_2(\Zb)$. From the definition and Proposition \ref{prop:udinv}, it is clear that
\begin{align*}
  \omega_\dd(T) \uf_\dd &= \zeta^{-1}_{\dd} \uf_\dd,~
\omega_\dd(T) \vf_\dd = \zeta_{2\dd} \wf_\dd,\\
\omega_\dd(T) \wf_\dd &= \zeta^{-1}_{2\dd} \omega_\dd(T^2S)\uf_\dd =
\zeta^{-1}_{2\dd} \omega_\dd(SB)\uf_\dd = \zeta_{2\dd} \omega_\dd(S)\uf_\dd = \zeta_{2\dd} \vf_\dd,\\
\omega_\dd(S)\uf_\dd &= \vf_\dd,~
\omega_\dd(S)\vf_\dd = \uf_\dd,\\
\omega_\dd(S)\wf_\dd &= \zeta_{2\dd}^{-1} \omega_\dd(STS)\uf_\dd =
\zeta_{2\dd}^{-1} \omega_\dd(STSTBS^2)\uf_\dd = \zeta_{2\dd}^{-1} \omega_\dd(STSTSTS)\uf_\dd = \zeta_{2\dd}^{-1} \omega_\dd(S)\uf_\dd = \wf_\dd.
\end{align*}
\end{proof}
In the following two subsections, we work at the $2$-part and $3$-part separately and construct $\uf_{\dd, p}$ for $p=2, 3$. This will shed some light on where $\uf_\dd$ comes from.

\subsection{The case $p = 3$.}
There are two possibilities for $\Ac_{\dd, 3}$. If $3 \nmid \dd$, then $\Ac_{\dd, 3}$ is trivial.
If $3 \mid \dd$,  we can identify the groups $\Ac_{\dd, 3}$ and $\Ac := M_2(\F_3)$ via
\begin{equation}
  \label{eq:Ac3}
  \begin{split}
    \kappa_{\dd, 3}:M_2(\F_3) &\cong \Ac_{\dd, 3}  \\
 \pmat{h_0}{-h_1}{h_2}{h_3} \bmod{3} &\mapsto h =[{h_0},{h_1},{h_2},{h_3}] \otimes \Zb_3,
  \end{split}
\end{equation}
which is just the map $\kappa_\dd$ in \eqref{eq:kappad} tensored with $\Zb_3$.
This is an isomorphism of finite quadratic modules if we equip $M_2(\F_3)$ with the quadratic form $Q_{\dd, 3}:=(3\dd_2)^{-1}\det $, which has value in $\frac{1}{3}\Zb/\Zb$.
Then
$
 H_3 \cong {\SL_2(\F_3) \times \SL_2(\F_3)},~ H'_3 = \langle {N'_3 \times N'_3}, {(T, T)} \rangle,
$
where
\begin{equation}
  \label{eq:N3'}
  \begin{split}
N_{\dd, 3}' &:=
\left\{
\begin{smat}
{1}{0}{0}{1}
\end{smat},
\begin{smat}
  { -1 }{ -1 }{ -1 }{ 1 }
\end{smat},
\begin{smat}
  { 0 }{ 1 }{ -1 }{ 0 }
\end{smat},
\begin{smat}
  { -1 }{ 1 }{ 1 }{ 1 }
\end{smat},
\begin{smat}
  { -1 }{ 0 }{ 0 }{ -1 }
\end{smat},
\begin{smat}
  { 1 }{ 1 }{ 1 }{ -1 }
\end{smat},
\begin{smat}
  { 0 }{ -1 }{ 1 }{ 0 }
\end{smat},
\smat{ 1 }{ -1 }{ -1 }{ -1 }
\right\} \\
&= \langle \smat{0}{1}{-1}{0}, \smat{1}{1}{1}{-1}\rangle \subset \SL_2(\F_3) \subset M_2(\F_3)
  \end{split}
\end{equation}
is isomorphic to the group of quaternions.
Another way to characterize $N_{\dd, 3}'$ is
\begin{equation}
  \label{eq:N3'p}
N_{\dd, 3}' =
\left\{
\pm
\begin{smat}
  { 1 }{ 0 }{ 0 }{ 1 }
\end{smat}
\right\}
\cup
\left\{
g \in \SL_2(\F_3): \tr(g) = 0
\right\}.
\end{equation}
From this, it is easy to check the following local analog of Lemma \ref{lemma:j} at 3.
\begin{lemma}
  \label{lemma:j3}
  For any $j \in \Zb/\dd_3\Zb$, we have
  $$
\kappa_{\dd, 3}(T^j N_{\dd, 3}') \cap \Ac_{\dd, 3}^0 = \{\pm [1, -j, 0, 1]\}.
  $$
\end{lemma}

Denote $\oo_3 \in \Ac$  the zero matrix.
Then $H_{\dd, 3}$ acts on the set $\Ac \backslash \oo_3$, and decomposes it into 3 orbits according to the norm of the elements.
The subgroup $H_{\dd, 3}' \subset H_{\dd, 3}$ acts on $\Ac \backslash \oo_3$ similarly and decomposes the three orbits into 5 orbits. We  denote  the sum of elements in each orbit by $\wf_i$ for $i = 0, 1, 2, 3, 4$.
They are explicitly given as follows: 
\begin{equation}
  \label{eq:wf}
  \wf_i :=
  \begin{cases}
    \sum_{h \in \kappa_{\dd, 3}( T^{i}N_{\dd, 3}') } \ef_h,& i= 0, 1, 2. \\
        \sum_{h \in \Ac \backslash \oo_3 ,~ \det(h) \equiv -i \bmod{3} } \ef_h,& i= 3, 4. \\
  \end{cases}
\end{equation}
This gives $U_{\dd, 3}^{H_{\dd, 3}'} \cong \Cb \ef_{\oo_3} + \sum_{j = 0}^4 \Cb \wf_j \subset \Cb[\Ac]$. Moreover, $ U^{H'_{\dd, 3}}_{\dd, 3}$ contains an $\SL_2(\Zb)$-invariant vector $4 \ef_{\oo_3} + \wf_3$, which is also in $U^{H_{\dd, 3}}_{\dd, 3}$.
Its orthogonal complement in $U^{H'_{\dd, 3}}_{\dd, 3}$ is 5 dimensional and decomposes into $ \chi_3^{-\dd_2} \oplus \chi_3^{-\dd_2} \oplus \varrho^{-\dd_2}$, where $\chi_3$ and $\varrho_{}$ are irreducible representations of $\SL_2(\Zb)$ given by
\begin{equation}
  \label{eq:reps3}
  \begin{split}
    \chi_3(T) &= \zeta^{}_3,~ \chi_3(S) = 1, \\
    \varrho(T) &=
\begin{pmatrix}
1 & & \\
& \zeta_3 & \\
& & \zeta_3^2
\end{pmatrix},~
\varrho(S) =
\frac{1}{3} \begin{pmatrix}
  -1 & 2 & 2 \\
2 & -1 & 2 \\
2 & 2 & -1
\end{pmatrix},
  \end{split}
\end{equation}
with respect to the basis $\{\wf_0 - \wf_1, \wf_0 - \wf_2, 8\ef_{\oo_3} - \wf_3, \wf_0 + \wf_1 + \wf_2, 2\wf_4 \}$.
For any $m \in \Zb$, we use $\varrho^{[m]}$ and $\chi^{[m]}_3$ to denote the representations of $\SL_2(\Zb)$ defined by
$$
\varrho^{[m]}(g) := \varrho(g)^m,~
\chi_3^{[m]}(g) := \chi_3(g)^m.
$$
Note that $\varrho^{[m]}$ and $\chi_3^{[m]}$ are well-defined and only depends on $m \bmod 3$.
We remark $\chi_3|_{\Gamma_0(2)}=\chi^{8}$.
In summary, we have
\begin{lemma}
  \begin{enumerate}
  \item
    The subrepresentation $\omega_{\dd, 3}^{H_{\dd, 3}} \subset \omega_{\dd, 3}$ fixed by $H_{\dd, 3}$ decomposes as
$$
\omega_{\dd, 3}^{H_{\dd, 3}} \cong \mathds{1} \oplus \varrho^{[-\dd_2]}
$$
with respect to the basis $\{  4 \ef_{\oo_3} + \wf_3, 8\wf_0 - \wf_3, \wf_0 + \wf_1 + \wf_2, 2\wf_4 \}$.
\item  Denote $U_{\dd, 3}'$ the orthogonal complement of $U_{\dd, 3}^{H_{\dd, 3}}$ in $U_{\dd, 3}^{H_{\dd, 3}'}$ and $\omega_{\dd, 3}'$ the restriction of $\omega_{\dd, 3}$ to $U_{\dd, 3}'$. Then
  $$
U_{\dd, 3}' = \left\{ \sum_{j = 0}^2 a_j \wf_j: a_j \in \Cb, \sum_j a_j = 0\right\}
$$
and $\omega_{\dd, 3}' \cong (\chi_3^{[-\dd_2]})^{\oplus 2}$.
\item
  Under this identification, $  M^{!}(\omega_{\dd, 3})^{H_{\dd, 3}} \cong  M^! \oplus M^!(\varrho^{[-\dd_2]})   $ and
  $$
  M^{!}(\omega_{\dd, 3})^{H_{\dd, 3}'} \cong M^{!}(\omega_{\dd, 3})^{H_{\dd, 3}} \oplus
M^!(\chi_3^{[-\dd_2]})^{ \oplus 2}.
  $$
\end{enumerate}
\end{lemma}

The analog of $\uf_\dd$ satisfying Lemma \ref{lemma:ufinv} and Proposition \ref{prop:udinv} is in the subspace $U_{\dd, 3}' = \{\uf_{\dd, 3}(\vec{c}): \vec{c} = (c_s) \in \Cb^{\varphi(\dd_3)}\}$, where
\begin{equation}
  \label{eq:ud3c}
  \uf_{\dd, 3}(\vec{c}) :=
  \sum_{j \in \Zb/\dd_3\Zb}
\lp  \sum_{s \in (\Zb/\dd_3\Zb)^\times } c_s \zeta_{\dd_3}^{sj} \rp
    \lp \sum_{h \in \kappa_{\dd, 3}(T^j N_{ \dd, 3}')} \ef_h \rp.
\end{equation}
As a consequence of Lemma \ref{lemma:ufinv}, we have the following local analog of Prop.\ \ref{prop:udinv} at $p = 3$.
\begin{proposition}
  \label{prop:udinv3}
  For any $\dd \mid 24$ and $\vec{c} \in \Cb^{\varphi(\dd_3)}$, we have
  \begin{equation}
    \label{eq:udinv3}
    \omega_{\dd, 3}(g) \uf_{\dd, 3}(\vec{c}) = \chi(g)^{-24\dd_2 /\dd_3} \uf_{\dd, 3}(\vec{c})
  \end{equation}
  for all $g \in \Gamma_0(2)$.
\end{proposition}
\begin{proof}
  If $\dd_3 = 1$, this is clear. Otherwise,
  $$
\omega_{\dd, 3}(T) \uf_{\dd, 3}(\vec{c}) = \chi_3^{-\dd_2}(g) \uf_{\dd, 3}(\vec{c}) = \chi^{-8\dd_2}(g) \uf_{\dd, 3}(\vec{c}).
$$
This finishes the proof.
\end{proof}
If $\vec{c} = (1, \dots, 1) \in \Cb^{\varphi(\dd_3)}$, then we simply denote $\uf_{\dd, 3}(\vec{c})$ by $\uf_{\dd, 3}$, which is explicitly given by
\begin{equation}
  \label{eq:uf3}
  \uf_{\dd, 3} =
    \begin{cases}
    2\wf_0 - \wf_1 - \wf_2,& \dd_3 = 3,\\
 \ef_{\oo_3},& \dd_3 = 1.
  \end{cases}
\end{equation}

\subsection{The case $p = 2$} In this case, the finite quadratic module $\Ac_{\dd, 2} = \Zb/\dd_2\Zb \times \Zb/(2\dd_2)\Zb \times \Zb/(2\dd_2)\Zb \times \Zb/\dd_2\Zb $ has the quadratic form
\begin{equation}
  \label{eq:Q2}
Q_{\dd, 2}([h_0, h_1, h_2, h_3]):= \frac{\dd_3^{-1}}{2\dd_2} (2 h_0 h_3  - h_1 h_2) \in \frac{1}{2\dd_2}\Zb/\Zb.
\end{equation}
Even though the size of $\Ac_{\dd, 2}$ can be large, the number of orbits under the suitable orthogonal group $H'_{\dd, 2}$ is much smaller.
More precisely, we have $H_{\dd, 2} = {N_{\dd, 2} \times N_{\dd, 2}}$ and $H'_{\dd, 2} \supset {N'_{\dd, 2} \times N'_{\dd, 2}}$, where
\begin{equation}
  \label{eq:H2}
  \begin{split}
    N_{\dd, 2} &:= \left\{\pmat{a}{b}{2c}{d} \in \SL_2(\Zb/(2\dd_2 \Zb))\right\}/\langle T^{\dd_2}, C^{ \dd_2/(2, \dd_2)}\rangle, \\
    N'_{\dd, 2} &:= \langle A, C, D\rangle \cong (\Zb/(\dd_2(2, \dd_2)/(4, \dd_2))\Zb   \times \Zb/(\dd_2/(2, \dd_2))\Zb) \rtimes \Zb/(4, \dd_2)\Zb.
  \end{split}
\end{equation}
Here $ A:= \smat{3}{2}{4}{3}, C:= \smat{5}{4}{16}{13}, D := \smat{-1}{1}{-2}{1}$ are elements in $\SL_2(\Zb)$ projected into $N_{\dd, 2}$.
The commutation relation is given by $DAD^{-1} = A^3$.
In particular $N'_{\dd, 2}$ has size $\dd_2^2$ and is abelian for $\dd_2 = 1, 2, 4$.

The group $N_{\dd, 2}$ acts on the left on $\Ac_{\dd, 2}$ via (simply coming from  matrix multiplication)
  \begin{equation}
    \label{eq:N2act}
 \pmat{a}{b}{2c}{d} \cdot [h_0, h_1, h_2, h_3] :=
[ah_0 + bh_2, ah_1 + 2(bh_3), 2(ch_0) + dh_2, ch_1 + dh_3]
  \end{equation}
for $\pmat{a}{b}{2c}{d} \in N_{\dd, 2}$ and $[{h_0}, {h_1}, {h_2}, {h_3}] \in \Ac_{\dd, 2}$. The same holds for the right action.
%
We can embed $N_{\dd, 2}$ into $\Ac_{\dd, 2}$ using the map $\kappa_{\dd, 2}: N_{\dd, 2} \to \Ac_{\dd, 2}$ defined by
\begin{equation}
  \label{eq:kappa2}
  \kappa_{\dd, 2}\lp \pmat{a}{b}{2c}{d} \rp := \dd_3^{-1} [a\bmod{\dd_2}, 2b, 2c, d \bmod{\dd_2}].
\end{equation}
It is then easy to check that
\begin{equation}
  \label{eq:easy}
\begin{split}
\kappa_{\dd, 2}(g_1 g_2) &= g_1 \cdot \kappa_{\dd, 2}(g_2) = \kappa_{\dd, 2}(g_1) \cdot g_2,
\\
  Q_{\dd, 2}(\kappa_{\dd, 2}(g)) &= \frac{\dd_3^{-1} \det(g) \bmod{\dd_2}}{\dd_2} \in \frac{2}{2\dd_2}\Zb/\Zb
\end{split}
\end{equation}
for all $g, g_1, g_2 \in N_{\dd, 2}$.
From this, it is easy to check that when $2 \mid \dd$, $\kappa_{\dd, 2}$ is a 2-1 map since $(\dd_2 + 1) \smat{1}{}{}{1} \in N_{\dd, 2}'$ and
$$
\kappa_{\dd, 2} \lp (\dd_2 + 1) \smat{1}{}{}{1} \rp =
\kappa_{\dd, 2} \lp \smat{1}{}{}{1} \rp.
$$

To better describe $\kappa_{\dd, 2}(N_{\dd, 2})$, it is useful to know the smallest additive subgroup of $\Ac_{\dd, 2}$ containing it. We describe it in the following lemma.

\begin{lemma}
  \label{lemma:additive}
Let $\Ac_{\dd, 2}' \subset \Ac_{\dd, 2}$ be the smallest (additive) subgroup containing $\kappa_{\dd, 2}(N'_{\dd, 2})$.
\begin{enumerate}
\item
When  $8 \mid \dd$, $\Ac_{\dd, 2}' \cong (\Zb/2\Zb)^2 \times (\Zb/8\Zb)^2$ is the orthogonal complement of the subgroup generated by $[6, 4, 0, 2], [0, 8, 0, 0], [0, 2, 2, 0] \in \Ac_{\dd, 2}$, and
$$
\kappa_{\dd, 2}^{-1}(\Ac_{\dd, 2}') = N'_{\dd, 2} \sqcup T^4N'_{\dd, 2}.
$$
Furthermore, we can distinguish the elements in $N_{\dd, 2}'$ and $T^4 N_{\dd, 2}'$ via
\begin{equation}
  \label{eq:distinguish}
  \begin{split}
    N_{\dd, 2}' &= \kappa_{\dd, 2}^{-1} \lp \{[h_0, h_1, h_2, h_3] \in \Ac'_{\dd, 2}: h_0^2 - \dd_3^2 \equiv h_1 + h_2 \bmod{16} \} \rp,\\
    T^4 N_{\dd, 2}' &= \kappa_{\dd, 2}^{-1} \lp \{[h_0, h_1, h_2, h_3] \in \Ac'_{\dd, 2}: h_0^2 - \dd_3^2 \equiv h_1 + h_2 + 8 \bmod{16} \} \rp.
  \end{split}
\end{equation}
\item
When $8 \nmid \dd$, $\Ac_{\dd, 2}' \cong (\Zb/\dd_2\Zb)^2$ is generated by $\kappa_{\dd, 2}(\smat{1}{}{}{1}), \kappa_{\dd, 2}( \smat{-1}{1}{-2}{1})$ and
$$
\kappa_{\dd, 2}^{-1}(\Ac_{\dd, 2}') = N'_{\dd, 2}.
$$
\end{enumerate}
\end{lemma}

\begin{proof}
  This can be verified using the appendix and some computer calculation.
\end{proof}

In addition, we record the following local analog of Lemma \ref{lemma:j} at the prime 2.
\begin{lemma}
  \label{lemma:j2}
  For any $j \in \Zb/\dd_2 \Zb$, we have
  $$
  \kappa_{\dd, 2}(T^j N_{\dd, 2}')\cap  \Ac_{\dd, 2}^0 = \{[r, r(2j + (\dd_3r)^2 - 1), 0, r]: r \in (\Zb/\dd_2\Zb)^\times\}.
  $$
\end{lemma}

\begin{proof}
  For $j = 0$, this follows directly from Lemma \ref{lemma:additive}. In general, it is easy to check that
  $$
  T^j (\kappa_{\dd, 2}(N_{\dd, 2}') \cap \Ac_{\dd, 2}^0) = \kappa_{\dd, 2}(T^j N_{\dd, 2}') \cap \Ac_{\dd, 2}^0
  $$
  for any $j$ since the action of $T$ preserves $\Ac_{\dd}^0$.
\end{proof}

Since $T^{\dd_2} = 0 \in N'_{\dd, 2}$ and $H_{\dd, 2}'$ is generated by ${N_{\dd, 2}' \times N_{\dd, 2}'}$ and $(T, T)$, the index of $H'_{\dd, 2}$ in $H_{\dd, 2}$ is $\dd_2$ and the sizes of $H_{\dd, 2}$ and $H'_{\dd, 2}$ are $\dd_2^6/(2, \dd_2)$ and $\dd_2^{5}/(2, \dd_2)$ respectively.
The dimension of $U_{\dd, 2}^{H_{\dd, 2}'}$ is the number of orbits in $\Ac_{\dd, 2}$ under the action of $H_{\dd, 2}'$.
Since the finite group $H_{\dd, 2}'$ is explicitly given in \eqref{eq:H2}, it is straightforward to calculate these orbits on a computer in practice.
We did this in SAGE \cite{SAGE} and received the following results
\begin{equation}
  \label{eq:rd2}
\dim U_{\dd, 2}^{H_{\dd, 2}'} =
  \begin{cases}
    4& \dd_2 = 1,\\
    16& \dd_2 = 2,\\
    46& \dd_2 = 4,\\
    118& \dd_2 = 8.\\
  \end{cases}
\end{equation}
With these calculations, one can already explicitly decompose the representation $\omega_{\dd, 2}^{H_{\dd, 2}'}$ on $U_{\dd, 2}^{H_{\dd, 2}'} $.
To find the desired vectors, we need to consider the following subspace of $U_{\dd, 2}^{H_{\dd, 2}'}$.

For $\dd_2 = 1$, the vector $\ef_{(1/2, 0)} - \ef_{(0, 1/2)}$ generates a 1-dimensional $\SL_2(\Zb)$-invariant subspace.
Denote $U_{\dd, 2}' \subset U_{\dd, 2}^{H_{\dd, 2}'}$ its orthogonal complement.
For $\dd_2 \ge 2$, the subgroup $H_{\dd, 2}'$ has index 2 in $H_{\dd/2, 2}' = \langle T^{\dd_2/2}, H_{\dd, 2}'\rangle$.
Even though $U_{\dd, 2}^{H_{\dd, 2}'}$ is large, it is much smaller compare to $U_{\dd, 2}$ and is amenable to calculations on a computer.
Denote $U_{\dd, 2}' \subset U_{\dd, 2}^{H_{\dd, 2}'}$ the orthogonal complement of $U_{\dd, 2}^{H_{\dd/2, 2}'} \subset U_{\dd, 2}^{H_{\dd, 2}'}$.
Then it is clear that
$$
\dim U_{\dd, 2}' = \frac{1}{2} \lp
\text{numbers of } H_{\dd, 2}'\text{-orbits of } \Ac_{\dd, 2}
- \text{numbers of } H_{\dd/2, 2}'\text{-orbits of } \Ac_{\dd, 2}
\rp.
$$
The following result comes out of the computer calculations.
\begin{lemma}
  \label{lemma:computer}
For any $\dd \mid 24$, the dimension of $U_{\dd, 2}'$ is $3\varphi(\dd_2)$. Furthermore, the support of any elements in $U_{\dd, 2}'$ is contained in the union of $\{h \in \Ac_{\dd, 2}: 2Q_\dd(h) = -\frac{\dd_3^{-1}}{\dd_2}\}$ and $\cup_{j \in \Zb/\dd_2 \Zb} \kappa_{\dd, 2}(T^j N_{\dd, 2}')$.
\end{lemma}

Now for $\dd \mid 24$, define the following vectors
\begin{equation}
  \label{eq:ud2c}
  \begin{split}
  \uf_{\dd, 2}(\vec{c}) &:= \sum_{j \in \Zb/\dd_2 \Zb}
  \lp \sum_{s \in (\Zb/\dd_2\Zb)^\times} c_s \zeta_{\dd_2}^{js}\rp
  \lp \sum_{h \in \kappa_{\dd, 2}(T^j N_{\dd, 2}')} \ef_h \rp, \\
  \vf_{\dd, 2}(\vec{c}) &:= \omega_{\dd, 2}(S) \uf_{\dd, 2}(\vec{c}),~
    \wf_{\dd, 2}(\vec{c}) := \zeta^{-\dd_3}_{\dd_2} \omega_{\dd, 2}(T) \vf_{\dd, 2}(\vec{c})
  \end{split}
\end{equation}
for all $\vec{c} = (c_s) \in \Cb^{\varphi(\dd_2)}$.
From Lemma \ref{lemma:computer}, we can show that these vectors give a basis of $U_{\dd, 2}'$.

\begin{lemma}
  \label{lemma:support}
For any $\dd \mid 24$ with $2 \mid \dd$ and $\vec{c} \in \Cb^{\varphi(\dd_2)}$, the vectors $\vf_{\dd, 2}(\vec{c}), \wf_{\dd, 2}(\vec{c})$ have the same support, which is disjoint from that of $\uf_{\dd, 2}(\vec{c_1})$ for any $\vec{c_1} \in \Cb^{\varphi(\dd_2)}$.
\end{lemma}

\begin{proof}
Since the action of $\omega_{\dd, 2}(T)$ does not change the support, we know that $\vf_{\dd, 2}(\vec{c})$ and $\wf_{\dd, 2}(\vec{c})$ have the same support. Now we have by definition
$$
\vf_{\dd, 2}(\vec{c}) =
(2\dd_2^2)^{-1}
 \sum_{s \in (\Zb/\dd_2\Zb)^\times} c_s
\sum_{\mu \in \Ac_{\dd, 2}}
\sum_{j \in \Zb/\dd_2\Zb}
 \zeta_{\dd_2}^{js}
  \lp \sum_{h \in \kappa_{\dd, 2}(T^j N_{\dd, 2}')} \ebf((\mu, h)) \ef_\mu \rp.
$$
We want to show that the coefficient of $\ef_\mu$ is zero if $\mu = \kappa_{\dd, 2}(T^{j'}g')$ with $j' \in \Zb/\dd_2\Zb$ and $g' \in N_{d, 2}'$.
Now if $h = \kappa_{\dd, 2}(g)$ with $g \in T^j N_{d, 2}'$, then $(\mu, h) = \frac{\dd_3^{-1}\tr(g (T^{j'}g')^{-1})}{\dd_2}$ by \eqref{eq:easy}. Since $N_{\dd, 2}'$ is normal in $N_{\dd, 2}$, which contains $T$, we have $T^j N_{\dd, 2}' T^{-j'} = T^{j-j'} N_{\dd, 2}'$.
Therefore, it suffices to show that the sum below vanishes
\begin{align*}
  \sum_{j \in \Zb/\dd_2\Zb}
& \zeta_{\dd_2}^{js}
   \sum_{h \in \kappa_{\dd, 2}(T^j N_{\dd, 2}')} \ebf((\mu, h))
=
\zeta_{\dd_2}^{j's}  \sum_{j'' \in \Zb/\dd_2\Zb}
 \zeta_{\dd_2}^{j''s}
   \sum_{h \in \kappa_{\dd, 2}(T^{j''} N_{\dd, 2}')} \zeta_{\dd_2}^{\dd_3^{-1} \tr(h)}
\end{align*}
with $j'' := j - j'$.
Also for $h = [h_0, h_1, h_2, h_3] \in \kappa_{\dd, 2}(N_{\dd, 2}')$, we have $\tr(T^j \cdot h) = \tr(h) + j\cdot h_1$.
Using this, we can rewrite
\begin{align*}
  \sum_{j'' \in \Zb/\dd_2\Zb}
& \zeta_{\dd_2}^{j''s}
   \sum_{h \in \kappa_{\dd, 2}(T^{j''} N_{\dd, 2}')} \zeta_{\dd_2}^{\dd_3^{-1} \tr(h)}
=
   \sum_{h \in \kappa_{\dd, 2}(N_{\dd, 2}')}
\zeta_{\dd_2}^{\dd_3^{-1} \tr(h)}
\sum_{j'' \in \Zb/\dd_2\Zb}
 \zeta_{\dd_2}^{j''(s + \dd_3^{-1} h_1)}
\end{align*}
By Lemma \ref{lemma:additive} (or inspecting the appendix), we know that $h_1 \in 2\Zb/2\dd_2\Zb$ for all $h \in \kappa_{\dd, 2}(N_{\dd, 2}')$.
So $s + \dd_3^{-1}h_1 \in (\Zb/\dd_2\Zb)^\times$ and the sum above vanishes.
\end{proof}

\begin{lemma}
\label{lemma:basis}
For any $\dd \mid 24$ and any basis $\mathcal{B}$ of $\Cb^{\varphi(\dd_2)}$, the set
\begin{equation}
  \label{eq:cbasis}
 \bigcup_{\vec{c} \in \mathcal{B}}  \{\uf_{\dd, 2}(\vec{c}), \vf_{\dd, 2}(\vec{c}), \wf_{\dd, 2}(\vec{c})\}
\end{equation}
is a basis of $U_{\dd, 2}'$.
\end{lemma}

\begin{proof}
We know that dimension of $U'_{\dd, 2}$ is $3 \varphi(\dd_2)$ from Lemma \ref{lemma:computer}, and need to check linear independence of the vectors in the set above.
Since the vectors $\uf_{\dd, 2}(\vec{c}), \vf_{\dd, 2}(\vec{c}), \wf_{\dd, 2}(\vec{c})$ are defined linearly, it suffices to prove the lemma for $\mathcal{B} = \{\vec{e}(s_0): s_0 \in (\Zb/\dd_2\Zb)^\times \}$ with $\vec{e}(s_0)\in \Cb^{\varphi(\dd_2)}$ the standard basis vector with 0 everywhere except 1 at the $s_0^{\text{th}}$ entry.
It is easily checked from the definition that $\uf_{\dd, 2}(\vec{c}), \vf_{\dd, 2}(\vec{c}), \wf_{\dd, 2}(\vec{c})$ are in $U_{\dd, 2}'$ are eigenvectors of $T$ with eigenvalue $\zeta_{\dd_2}^{-s_0}$ when $\vec{c} = \vec{e}(s_0)$.
Therefore, it suffices to check that the three vectors $\uf_{\dd, 2}(\vec{c}), \vf_{\dd, 2}(\vec{c}), \wf_{\dd, 2}(\vec{c})$ are linearly independent whenever $\vec{c} = \vec{e}(s_0)$.

When $\dd_2 = 1$, this is easily checked by hand.
When $\dd_2 \ge 2$, it suffices to show that $\vf_{\dd, 2}(\vec{e}(s_0))$ and $\wf_{\dd, 2}(\vec{e}(s_0))$ are linearly independent by Lemma \ref{lemma:support}.
Assume otherwise, then the restriction of $\omega_{\dd, 2}$ to $\Cb \uf_{\dd, 2}(\vec{e}(s_0)) + \Cb \vf_{\dd, 2}(\vec{e}(s_0))$ is a 2-dimensional representation of $\SL_2(\Zb)$. In the basis $\{\uf_{\dd, 2}(\vec{e}(s_0)), \vf_{\dd, 2}(\vec{e}(s_0))\}$, it is given by the map
$$
T \mapsto \pmat{\zeta_{\dd_2}^{-1}}{}{}{\pm \zeta_{2\dd_2}},~
S \mapsto \pmat{}{1}{1}{}.
$$
However, $(T\cdot S)^6$ is the identity, whereas
$$
\lp \pmat{\zeta_{\dd_2}^{-1}}{}{}{\pm \zeta_{2\dd_2}}
\pmat{}{1}{1}{} \rp^6
= \lp \pm \pmat{\zeta_{2\dd_2}^{-1}}{}{}{\zeta_{2\dd_2}^{-1}} \rp^3
$$
is not the identity since $2 \mid \dd_2$.
This is a contradiction and finishes the proof.
\end{proof}
\begin{proposition}
  \label{prop:decomp2}
  For $\dd \mid 24$, let $\omega_{\dd, 2}'$ denote the restriction of $\omega_{\dd, 2}$ to $U_{\dd, 2}' \subset U_{\dd, 2}^{H_{\dd, 2}'}$. Then $\uf_{\dd, 2}(\vec{c})$ satisfies
  \begin{equation}
        \label{eq:udinv2}
     \omega_{\dd, 2}'(g) \uf_{\dd, 2}(\vec{c}) = \chi(g)^{-24\dd_3/\dd_2} \uf_{\dd, 2}(\vec{c})  \end{equation}
  for all $g \in     \Gamma_0(2)$ and $\vec{c} \in \Cb^{\varphi(\dd_2)}$.
  Furthermore with respect to the basis in \eqref{eq:cbasis}, we have
  \begin{equation}
    \label{eq:decomp2}
    \omega_{\dd, 2}' \cong
      \varrho_{\dd_2, \dd_3}^{\oplus \varphi(\dd_2)}
  \end{equation}
Here $\varrho_{\dd_2, \dd_3}$ is the 3-dimensional representation defined in \eqref{eq:varrhod}.
\end{proposition}
\begin{remark}
  If $\vec{c} = (1, \dots, 1) \in \Cb^{\varphi(\dd_2)}$, we will simply write $\uf_{\dd, 2}$ for $\uf_{\dd, 2}(\vec{c})$.
  They are explicitly given by
    \begin{equation}
    \label{eq:uf2}
\uf_{\dd, 2} =
    \begin{cases}
\ef_{\oo_2},      & \dd_2 = 1, \\
2^{\dd_2/2} \lp \sum_{\kappa_{\dd, 2}(N_{\dd, 2}')} \ef_h - \sum_{\kappa_{\dd, 2}(T^{\dd_2/2} N_{\dd, 2}' )} \ef_h \rp,      & \dd_2 = 2, 4, 8.
    \end{cases}
  \end{equation}
\end{remark}
\begin{proof}
  For the first claim, it suffices to prove the it when $g = T, S^2, TB$, which are generators of $\Gamma_0(2)$.
    If $g = T$, then $\omega_{\dd, 2}'(T) \ef_h = \ebf(-Q_{\dd, 2}(h)) \ef_h$.
  For $h \in \kappa_{\dd, 2}(T^j N_{\dd, 2}')$, we have $Q_{\dd, 2}(h) = \dd_3^{-1}/\dd_2 = \dd_3/\dd_2 \in \tfrac{1}{\dd_2} \Zb/\Zb$. Equation \eqref{eq:udinv2} therefore holds for $g = T$.
  When $g = S^2$, since $\omega_\dd(S^2) \ef_h = \ef_{-h}$ for all $h \in \Ac_\dd$ and $-\smat{1}{}{}{1} \in N_{\dd, 2}'$, we know that $- \kappa_\dd(T^j N_{\dd, 2}') = \kappa_\dd(T^j N_{\dd, 2}')$ and equation \eqref{eq:udinv2} holds for $g = S^2$.

For $g = TB = TST^2S^{-1}$, it suffices to show that the middle equation below
$$
\omega_{\dd, 2}'(S) \omega_{\dd, 2}(T^{-1}) \uf_{\dd, 2}(\vec{c}) =
\zeta_{\dd_2}^{\dd_3^{-1}} \vf_{\dd, 2}(\vec{c}) =
 \omega_{\dd, 2}(T^{2})\vf_{\dd, 2}(\vec{c}) =
 \omega_{\dd, 2}(T^{2}) \omega_{\dd, 2}'(S) \uf_{\dd, 2}(\vec{c}).
$$
This is easily checked by hand when $\dd_2 = 1$.
If $2 \mid \dd_2$, we know by Lemmas \ref{lemma:computer} and \ref{lemma:support} that the support of $\vf_{\dd, 2}(\vec{c})$ is contained in $\{h \in \Ac_{\dd, 2}: 2Q_\dd(h) = - \frac{\dd_3^{-1}}{\dd_2} \}$. It is therefore an eigenvector of $\omega'_{\dd, 2}(T^2)$ with eigenvalue $\zeta_{\dd_2}^{\dd_3^{-1}}$. This proves the first claim.
As in the proof of Prop.\ \ref{prop:repembed}, the vectors $\{\uf_{\dd, 2}(\vec{c}), \vf_{\dd, 2}(\vec{c}), \wf_{\dd, 2}(\vec{c})\}$ generate a 3-dimensional subrepresentation of $\omega_{\dd, 2}$ isomorphic to $\varrho_{\dd_2, \dd_3}$.
The second claim then follows from Lemma \ref{lemma:basis}.
\end{proof}

\section{Borcherds liftings}
 \label{sect:Borcherds}

\subsection{Brief review of Borcherds liftings} \label{sect:Borcherds}
We first  set up notation and briefly review the Borcherds lifting, following \cite[Section 3]{YY19}. Let $V=V_\dd$ and $H$ be as in Section \ref{sect3}.

 Let
 \begin{equation}
 \mathcal L = \{ w \in V_\C:\,  (w, w) =0,  \quad (w, \bar w)<0\}.
 \end{equation}
 and let $\mathbb D$ be the Hermitian symmetric domain of oriented negative $2$-planes in $V_\R=V\otimes_\Q \R$. Then one has an isomorphism
 $$
 pr: \mathcal L/\C^\times \cong \mathbb D, \quad  w=u + i v \mapsto \R u + \R (-v).
 $$
For the isotropic matrix $\ell= \kzxz{0}{-1}{0}{0} \in L$ and $\ell' =\kzxz {0} {0} {\frac{1}{\dd}} {0} \in V$ with $(\ell, \ell')=1$. We also have the associated tube domain
$$
\mathcal H_{\ell, \ell'} =\{ \kzxz {z_1} {0} {0} {-z_2} : \,  y_1 y_2  >0\} ,\quad y_i=\hbox{Im}(z_i),
$$
together with
$$
w: \mathcal H_{\ell, \ell'} \rightarrow \mathcal L,  \quad w( \kzxz {z_1} {0} {0} {-z_2}) = \kzxz {z_1} {-\dd z_1 z_2} { \frac{1}{\dd}} { -z_2}.
$$
This gives an isomorphism $\mathcal H_{\ell, \ell'}\cong \mathcal L/\C^\times$. We also identity $\H^2 \cup (\H^-)^2$ with $\mathcal H_{\ell, \ell'}$ by
$$\psi_\dd: z=(z_1,z_2)\mapsto\kzxz{\frac{z_1}{\dd}}{0}{0}{\frac{z_2}{\dd}}.$$ Note that we use this identification in order to have the following compatibility property and it is also the identification used in the computation of Borcherds products. The following is a special case of \cite[Proposition 3.1]{YY19}

\begin{proposition} Define
$$
w_{\dd}: \,  \H^2 \cup (\H^-)^2  \rightarrow \mathcal L, \quad w_{\dd}(z_1, z_2) = w\circ\psi_\dd(z_1,z_2)=\kzxz {\frac{z_1}{\dd}} {\frac{-z_1 z_2}{\dd}} {\frac{1}{\dd}} {\frac{-z_2}{\dd}}.
$$
 Then the composition $pr \circ w_{\dd}$ gives an isomorphism between $\H^2 \cup (\H^-)^2$ and $\mathbb D$. Moreover, $w_{\dd}$ is $H(\R)$-equivariant, where $H(\R)$  acts on $\H^2 \cup (\H^-)^2$ via  the usual linear fraction:
$$
(g_1, g_2)(z_1, z_2) =(g_1(z_1), g_2(z_2)),
$$
and acts on $\mathcal L$ and $\mathbb D$ naturally via its action on $V$. Moreover, one has
\begin{equation} \label{eq:linebundle}
(g_1, g_2) w_{\dd}(z_1, z_2) = \frac{j(g_1, z_1) j(g_2, z_2)}{\nu(g_1, g_2)} w_{\dd}(g_1(z_1), g_2(z_2)),
\end{equation}
where $\nu(g_1, g_2) =\det g_1 =\det g_2$ is the spin character of $H\cong\Gspin(V)$, and
$$
j(g_1,g_2, z_1, z_2) = j(g_1, z_1) j(g_2, z_2)=(c_1 z_1+d_1)(c_2 z_2 +d_2)
$$
is the automorphy factor (of weight $(1, 1)$).
\end{proposition}

For a congruence subgroup $\Gamma$ of $\SL_2(\Z)$, let $X_\Gamma$ be the associated open modular curve over $\Q$ such that $X_\Gamma(\C) =\Gamma \backslash \H$.
Assume $\Gamma \supset \Gamma(M)$ for some integer $M \ge 1$.
Let
$$
\nu:  \A^\times \hookrightarrow \GL_2(\A), \quad \nu(d) = \diag(1, d).
$$
Let
$K(\Gamma)$ be the product of $\nu(\hat{\Z}^\times)$ and the preimage of $\Gamma/\Gamma(M)$ in $\GL_2(\hat\Z)$ (under the map $\GL_2(\hat\Z)\rightarrow \GL_2(\Z/M))$. Let $K=(K(\Gamma)\times K(\Gamma))\cap H(\A_f)$. Then    one has by the strong approximation theorem
$$
X_K\cong X_\Gamma \times X_\Gamma.
$$
In this way, we have identified the product of two copies of a modular curve $X_\Gamma$  with a Shimura variety $X_K$.

Suppose that $\Gamma$ acts on $L'/L$ trivially, then  for each $\mu \in  L'/L$ and $m \in Q(\mu)+L$, the associated special divisor $Z_\Gamma(m, \mu)$ is given by
$$
Z_{\Gamma}(m, \mu) =(\Gamma \times \Gamma)\backslash \{ (z_1, z_2):\, w_{\dd}(z_1, z_2) \perp x \hbox{ for some } x \in \mu+L, Q(x)=m\}.
$$
More generally, assume $\Gamma \supset \Gamma(M)$ preserves $L$, and $\uf  =\sum a_\mu \ef_\mu \in \C[L'/L]$ is $\Gamma \times \Gamma$-invariant, the cycle
$$
Z_{\Gamma(M)} (m, \uf) =\sum  a_\mu  Z_{\Gamma(M)} (m, \mu)
$$
descends to a cycle $Z_\Gamma(m, \uf)$ in $X_{\Gamma} \times X_{\Gamma}$.
For our purpose, we will take
$$
\dd \mid 24,~\Gamma = \Gamma_{\chi, \dd}\supset \Gamma_{\dd} \supset \Gamma(2\dd) \supset \Gamma(48)
$$
from now on and write $X_\dd := X_\Gamma = X_{\Gamma_{\chi, \dd}}$.
Notice that $X_1 = X_0(2)$ has two cusps, $i\infty$ and $0$. Since $\{T^j: 1 \le j \le \dd\}$ are coset representatives of $\Gamma_{\chi, \dd}$ in $\Gamma_{\chi, 1}$, the modular curve $X_\dd$ has the same cusps as $X_1$.

\begin{lemma}\cite{YY19}[Corollary 3.3]
\label{lemma:YYdivisor}
For $\dd \mid \dd' \mid 24$, let $\pi_{}: X_{\Gamma(2\dd')} \rightarrow X_{\dd}$ be the natural projection.
Then
\begin{equation}
  \label{eq:diagonal*}
(\pi_{} \times \pi_{})^*(X_\dd^\Delta) =\sum_{\gamma \in \Gamma_{}/\Gamma(2\dd')}  Z_{\Gamma(2\dd')} \lp \frac{1}{\dd'}, \frac{1}{\dd'}\gamma +L\rp
\end{equation}
and the group $\Gamma(2\dd')$ can be replaced by $\Gamma_{\dd'}$.
\end{lemma}

Since $\Gamma$ is normal in $\Gamma_0(2) = \Gamma_{\chi, 1}$, the action of $T = \smat{1}{1}{0}{1} \in \Gamma_0(2)$ on $\Hb$ factors through $X_\dd$ and defines a map $X_\dd \to X_\dd$, which we also denote by $T$.
Using this, we can define translates of the diagonal
\begin{equation}
  \label{eq:diagonalj}
  X_\dd^\Delta(j) := (T^j \times \mathds{1})^*(X_\dd^\Delta) \subset X_\dd \times X_\dd
\end{equation}
for $j \in \Zb/\dd\Zb$.
Equation \eqref{eq:diagonal*} also generalizes to
\begin{equation}
  \label{eq:diagonalj*}
(\pi_{} \times \pi_{})^*(X_\dd^\Delta(j)) =\sum_{\gamma \in \Gamma_{}/\Gamma(2\dd')}  Z_{\Gamma(2\dd')} \lp \frac{1}{\dd'}, \frac{1}{\dd'}T^j\gamma +L\rp    ,
\end{equation}
where one can replace $\Gamma(2\dd')$ with $\Gamma_{\dd'}$.
From this, we see that the pull back of $X_\dd^\Delta(j)$ along natural projection $X_{\dd'} \times X_{\dd'} \to X_{\dd} \times X_{\dd}$ is $\cup_{l \in \dd \Zb/\dd'\Zb} X_{\dd'}^\Delta(j + l)$.
Before proceeding further to state and prove the main result of this section, we record the following identity for convenience.

\begin{lemma}
  \label{lemma:pol}
  For any $\dd \in \Nb_{}$, we have the following identity in $\Q(X)$
  \begin{equation}
    \label{eq:pol}
    \mathrm{p}_\dd(X) :=     \prod_{j \in \Zb/\dd\Zb} (1 - \zeta_\dd^j X)^{a_\dd(j)}
    = \prod_{b \mid \dd} ( 1- X^{\dd/b})^{b \cdot \mu(\dd/b)},
  \end{equation}
  where $a_\dd(j)$ is the constant defined in \eqref{eq:ufd}.
  Furthermore for any $s \in \Nb$, we have
  \begin{equation}
    \label{eq:pol2}
    \prod_{\dd \mid s} \mathrm{p}_\dd(X^{s/\dd}) = (1 - X)^{s}.
  \end{equation}
\end{lemma}

\begin{proof}
  To prove \eqref{eq:pol}, it suffices to check that both sides have the same roots counting multiplicity, since they agree at $X = 0$.
  The multiplicity of $X = \zeta_\dd^{j}$ on the LHS is $a_\dd(-j) = a_\dd(j)$, whereas it is
  $
\sum_{b \mid (\dd, j)} b\cdot \mu \lp \dd/b \rp
  $ on the RHS.
  The equality is then a consequence of the identity
  $$
\sum_{b \mid n} b\cdot \mu \lp \frac{\dd}{b} \rp = \mu(\dd/n) \frac{\varphi(\dd)}{\varphi(\dd/n)},~ n \mid \dd,
$$
which is a standard exercise that we leave, along with \eqref{eq:pol2}, to the curious readers.
\end{proof}

Now, we can specialize Borcherds' far reaching lifting theorem (\cite[Theorem 13.3]{Borcherds98} (see also \cite[Theorems 2.1 and 2.2]{YY19}) to the modular function $\tilde{F}_\dd$ in \eqref{eq:Fdt} and the result below.

\begin{theorem} \label{theo:BorcherdsLifting}
  For every $\dd \mid 24$, recall the modular function in $M^!(\omega_\dd)^{H_{\dd}'}$
  $$
  \tilde{F}_\dd (\tau) =
  \sqrt{2}^\dd \lp \ff_2^{-24/\dd}(\tau)  \uf_\dd + \ff_1^{-24/\dd}(\tau)  \vf_\dd + \ff^{-24/\dd} (\tau) \wf_\dd \rp +
  \begin{cases}
        24(\ef_{(0, 0)} + \ef_{(1/2, 0)}),& \dd = 1,\\
      0, & \dd > 1,
  \end{cases}
  $$
  defined in \eqref{eq:Fdt} with $\uf_\dd, \vf_\dd, \wf_\dd \in U_\dd$ vectors defined in \eqref{eq:ufd} and \eqref{eq:wv}.
  Let $\Psi_\dd(z)$ be the meromorphic modular function on $X_{\dd} \times X_{\dd}$ (with some characters) associated to $\tilde{F}_\dd$ via Borcherds multiplicative lifting, i.e.\ $-\log \lVert \Psi_\dd(z) \rVert_{\mathrm{Pet}}^2$ is the regularized theta lift of $\tilde{F}_\dd$ with $\lVert \cdot \rVert_{\mathrm{Pet}}$ a suitably normalized Petersson norm (see e.g.\ Theorem 2.1 in \cite{YY19}).
Then $\Psi_\dd(z)$ has the following properties:
  \begin{enumerate}
\item  one has on $X_\dd \times X_\dd$
$$
\mathrm{Div} (\Psi_\dd(z) )=
 \sum_{j \in \Zb/\dd\Zb} a_\dd(j) X_\dd^\Delta(j)
$$
\item
When $\dd = 1$, $\Psi_\dd(z)$ has a product expansion of the form
$$
\Psi_1(z)=
2^{12} (q_1 - q_2)
\prod_{m, n \ge 1}
(1 - q_1^{n}q_2^{m})^{c_1(mn)}
 (1 - q_1^{2n}q_2^{2m})^{c_{-1}(2mn)}
$$
near the cusp $\Q\ell$ of $X_K$, where $q_j := e^{2\pi i z_j}$ and $c_\dd(l)$ are the Fourier coefficients defined in Remark \ref{rmk:FC}.
\item
When $\dd > 1$, $\Psi_\dd(z)$ has a product expansion of the form
$$
\Psi_\dd(z)=
 \prod_{b \mid \dd} (q_1^{1/b} - q_2^{1/b})^{b \cdot \mu(\dd/b)}
\prod_{\substack{m, n \in \Nb \\ mn \equiv -1 \bmod{\dd}}}
 \lp
 \prod_{b \mid \dd} (1 - q_1^{n/b}q_2^{m/b})^{b \cdot \mu(\dd/b)}
\rp^{c_\dd(mn)(-1)^{(n^2 - 1)/\dd_2}}
$$
 near the cusp $\Q\ell$ of $X_K$, where $\mu$ and $\varphi$ are the M\"obius and Euler $\varphi$-function respectively.
\end{enumerate}
\end{theorem}

\begin{proof}
  This is a specialization of Borcherds' result to the input $\tilde{F}_\dd \in M^!(\omega_\dd)$. For this, we need to substitute the suitable parameters into Borcherds' result, which has been specialized to this case in Theorem 2.1 and 2.2 in \cite{YY19}. Using the specialization there, we see that the divisor of $\Psi_\dd$ is
  \begin{align*}
    \sum_{j \in \Zb/\dd\Zb} a_\dd(j)\sum_{\mu \in \kappa(T^jN_\dd')} Z_{\Gamma_\dd}\lp -\frac{1}{\dd}, \mu \rp &=
\sum_{j \in \Zb/\dd\Zb} a_\dd(j) \sum_{\gamma \in \Gamma /\Gamma(2\dd)}  Z_{\Gamma(2\dd)} \lp - \frac{1}{\dd}, \frac{1}{\dd} T^j \gamma + L\rp,
  \end{align*}
which gives us the first claim after applying Lemma \ref{lemma:YYdivisor}.

For the second and third claim, we specialize Theorem 2.2 in \cite{YY19} and use the notations there.
When $\dd = 1$, this is rather classical and can be found in \cite{Sch08} (see also Prop.\ 5.3 in \cite{YY19}\footnote{Note that the Fourier expansion of $f$ loc.\ cit.\ is incorrect.}).
For $\dd > 1$, the Weyl chambers for $\tilde{F}_\dd$ are the same as in the case $\dd = 1$, and we choose the one $W = \Rb \{\smat{a}{}{}{-1}: a > 1\}$. Since $\uf_\dd, \vf_\dd$ and $\wf_\dd$ do not have support on any isotropic vector, the associated form $\tilde{F}_{\dd, P}$ is identically zero, and the Weyl vector $\rho(W, \tilde{F}_\dd)$ is 0.
Since $\tilde{F}_\dd$ does not have any constant term, the constant $C$ in the product expansion is 1,
For the infinite product, suppose $\lambda = \frac{1}{\dd}\smat{-m}{}{}{n}$ with $m, n \in \Zb$. Then $(\lambda, W) > 0$ if and only if $m \ge -n, n \ge 0$ and $(m, n) \neq (0, 0)$.

The set of $\mu \in L'_0/L$ with $p(\mu) = \lambda$ consists then of $\frac{1}{\dd}\smat{-m}{-j}{0}{n}$ with $j \in \tfrac{1}{2}\Zb/\dd\Zb$. For such $\lambda, \mu$, we have
$$
1 - e((\lambda, z) + (\mu, \ell')) = 1 - \zeta_\dd^{j} q_1^{n/\dd} q_2^{m/\dd}.
$$
By inspecting the $q$-expansion of $F_\dd$, we notice that
$$
F_\dd(\tau) = \lp q^{-1/\dd} + \sum_{l \in \Nb,~ l \equiv -1 \bmod{\dd}} c_\dd(l)q^{l/\dd} \rp \uf_\dd  + \sum_{\mu \in L'/L,~  Q_\dd(\mu) \in \{\tfrac{1}{2\dd}, \tfrac{1}{2\dd} + \tfrac{1}{2} \}} F_{\dd, \mu}(\tau) \ef_\mu.
$$
Therefore, the only pairs of $(m, n)$ with $m < 0$ is $m = -n = -1$, and the only $\mu \in L_0'/L$ where $F_{\dd, \mu}$ could be nonzero are contained in the support of $\uf_\dd$, hence
$$
\mu = \tfrac{1}{\dd} \smat{-m}{j}{0}{n} + L \in \tfrac{1}{\dd} T^{j'} \Gamma_{\chi, \dd} + L$$
with $mn \equiv -1 \bmod{\dd}$ and $j' := nj - \tfrac{n^2 - 1}{2} \in \Zb/\dd\Zb$ by Lemma \ref{lemma:j}. The Fourier coefficient $c(-Q(\lambda), \mu)$ of the input is then $c_\dd(mn) a_\dd(j')$.
It is easy to check that $a_\dd(j') = a_\dd(j) (-1)^{(n^2 - 1)/\dd_2}$.
By Theorem 2.2 in \cite{YY19}, $\Psi_\dd(z)$ has the product expansion
$$
\Psi_\dd(z) =
\prod_{\substack{m \in \Zb_{\ge -1},~ n \in \Nb \\ mn \equiv -1 \bmod{\dd}}}
\prod_{j \in \Zb/\dd\Zb} (1 - \zeta_\dd^{j} q_1^{n/\dd} q_2^{m/\dd})^{c_\dd(mn) a_\dd(j) (-1)^{(n^2-1)/8}}
$$
Finally, applying Lemma \ref{lemma:pol} finishes the proof.
\end{proof}

\subsection{The Weber function differences as Borcherds liftings.}

Now, we are ready to state and prove the following main result of this section.
\begin{theorem}
  \label{thm:Blift}
For $\dd \mid 24$, let $\Psi_{\dd}(z_1, z_2)$ be the Borcherds product of $\tilde{F}_\dd \in M^!(\omega_\dd)$ as in Theorem \ref{theo:BorcherdsLifting}.
Then for any $s \mid 24$ and $\varepsilon \in \{\pm 1\}$, we have
\begin{equation}
  \label{eq:Blift}
  \lp \ff_2(z_1)^{24/s} - (\varepsilon\ff_2(z_2))^{24/s} \rp^{s} = \prod_{\dd \mid s}
\Psi_\dd(z_1, z_2)^{\varepsilon^{24/\dd} }.
\end{equation}
\end{theorem}
\begin{proof}
We first look at their divisors in the open Shimura varieties $X_{s} \times X_{s}$.
Suppose $\varepsilon = 1$.
  The LHS clearly has $s \cdot [X_{s}^\Delta]$ as its divisor, whereas the RHS has the divisor
\begin{align*}
\sum_{\dd \mid s}
\sum_{j \in \Zb/\dd\Zb}
a_\dd(j)
\sum_{l \in \dd\Zb/s\Zb} [X_{s}^\Delta(j + l)]
&=
\sum_{k \in \Zb/s\Zb}
\lp \sum_{\dd \mid s} a_\dd(k)
\rp
[X_{s}^\Delta(k)]
\\
&= s \cdot [X_{s}^\Delta],
\end{align*}
as
$$
\sum_{\dd \mid s} a_\dd(k)
=
\sum_{\dd \mid s}
\mu(\dd/(\dd, k)) \varphi(\dd)/\varphi(\dd/(\dd, k))
=
\begin{cases}
  s, & k = 0\\
  0, &\text{ otherwise.}
\end{cases}
$$
When $\varepsilon = -1$, the argument is the same unless $8 \mid s$. In that case, the divisor of the LHS is $s \cdot [X_{s}^\Delta(s/2)]$, whereas
\begin{align*}
\mathrm{Div}(\mathrm{RHS}) &=
\mathrm{Div}\prod_{\dd \mid s/2} \Psi_\dd(z_1, z_2)^2
- \mathrm{Div}\prod_{\dd \mid s} \Psi_\dd(z_1, z_2)\\
&= s \cdot ([X_{s}^\Delta] + [X_{s}^\Delta(s/2)] ) - s \cdot [X_{s}^\Delta]
= s \cdot [X_{s}^\Delta(s/2)].
\end{align*}

Now let
$$
g(z_1, z_2)=\frac{\prod_{\dd \mid s} \varepsilon^{24/(\dd, s)} \Psi_\dd(z_1, z_2)^2}{(\ff_2(z_1)^{24/s} - (\varepsilon \ff_2(z_2))^{24/s})^{s}}.
$$
Then it is holomorphic and has no zeros on $X_{s} \times X_{s}$.  So
$$
\hbox{Div}(g(z_1,  z_2))
 =a_{\infty, 1} ( \{\infty\} \times X_{s})  + a_{\infty, 2}( X_{s} \times \{ \infty\})
   + a_{0, 1} ( \{0 \} \times X_{s}) + a_{0, 2} (X_{s} \times \{0\})
$$
is supported on the boundary with $a_{i, j} \in \Z$. The product expansion of $\Psi_d$ and the definition of $\ff_2$ imply that $a_{\infty, 1}=a_{\infty, 2}=0$.

Next, fix $z_2 \in X_{s}$ the above argument shows that $g(z_1, z_2)$, as a function of $z_1$ on $X_{s} \cup \{ 0, \infty\} $  has only zeros or poles at the cusp $\{0\}$, which is impossible. So $g(z_1,  z_2)$ has no zeros or poles in $z_1$, and is therefore independent of $z_1$, i.e,  $g(z_1, z_2)=g(z_2)$ is purely a function of $z_2$  with no zeros or poles in $X_{s} \cup \{ \infty\}$. This implies that $g(z_1, z_2) =g(z_2) =C$ is a constant.

Finally, looking at the $q_1$-leading term of the Fourier expansion, we see $C=1$
and this proves the theorem. The last part of the proof follows from the argument in the proof of \cite[Theorem 3.4]{YY19}
\end{proof}

\section{Big CM values} \label{sect:BigCM}

\subsection{Products of CM cycles as  big CM cycles} Yang and Yin have described how to view a pair of CM points as a big CM point in \cite[Section 3.2]{YY19}, which we now briefly review for convenience and set up necessary notation. We modify a little for use in this paper.  For $j = 1, 2$, let $d_j<0$ be co-prime, fundamental discriminants satisfying \eqref{eqd}.
Denote $\kk_j =\Q(\sqrt d_j)$ with ring of integers $\OO_j=\Z[\frac{1 +\sqrt{d_j}}2]$, and class group $\Cl(d_j)$.
Let $\kk =\kk_1 \otimes_\Q  \kk_2=\Q(\sqrt{d_1}, \sqrt{d_2})$ with ring of integers $\OO_E =\OO_1 \otimes_\Z \OO_2$. Then $\kk$ is a biquadratic  CM  number field with  real quadratic subfield $F=\Q (\sqrt D)$ and $D=d_1d_2$.

 For a positive integer $\dd$,  we define $W=W_\dd= \kk$ with the $F$-quadratic form $Q_F(x) =\frac{  \dd x \bar x}{\sqrt D}$. Let $W_\Q =W$ with the $\Q$-quadratic form $Q_\Q(x) =\tr_{F/\Q} Q_F(x)$.  Let $\sigma_1=1$ and $\sigma_2 =\sigma$ be two real embeddings of $F$ with $\sigma_j(\sqrt D) =(-1)^{j-1}\sqrt D$. Then
 $W$ has signature $(0, 2)$ at $\sigma_2$ and $(2, 0)$  at $\sigma_1$ respectively, and so $W_\Q$ has signature $(2,2)$.
 Choose a $\Z$-basis of $\OO_E$ as follows
 $$
 e_1=1\otimes 1, \quad e_2 = \frac{-1+\sqrt{d_1}}2 =\frac{-1+\sqrt{d_1}}2 \otimes 1, \quad  e_3= \frac{1+\sqrt{d_2}}2 =1\otimes \frac{1+\sqrt{d_2}}2, \quad e_4= e_2 e_3.
 $$
We will drop $\otimes$ when there is no confusion.  Then it is easy to check that
\begin{equation} \label{eq:SpaceIdentification}
(W_\Q, Q_\Q) \cong (V, Q)=(M_2(\Q),   \dd\det),\quad\quad\quad \sum x_i e_i \mapsto \kzxz {x_3} {x_1} {x_4} {x_2}.
\end{equation}
We will identify $(W_\Q, Q_\Q)$ with the quadratic space $(V, Q) =(M_2(\Q),  \dd \det)$. Under this identification, the lattice $M_2(\Z)$ becomes $\OO_E$, and the lattice $L_\dd$ becomes $\Z e_1 +\Z e_2 + \Z e_3 + \Z 2 e_4 \subset \OO_E $, which we still denote by $L=L_\dd$.
Define $T$ be the maximal torus in $H$ given  by the following diagram:
\begin{equation} \label{eq:torus}
\xymatrix{
 1 \ar[r]  &\G_m \ar[r] \ar[d] &T \ar[r] \ar[d] &\hbox{Res}_{F/\Q} \SO(W) \ar[r] \ar[d] &1
\cr
 1 \ar[r]  &\G_m \ar[r] &H \ar[r] &\SO(V) \ar[r] &1.
 \cr
}
\end{equation}
Then $T$ can be identified with (\cite{HYbook}, \cite[Section 6]{BKY})
$$
T(R) =\{ (t_1, t_2)\in (E_1\otimes_\Q R)^\times \times (E_2\otimes_\Q R)^\times:\,  t_1 \bar{t}_1 = t_2 \bar{t}_2 \},
$$
for any $\Q$-algebra $R$,  and the map from $T$ to $\SO(W)$ is given by  $(t_1, t_2) \mapsto {t_1}/{\bar{t}_2}$.
 The map from $T$ to $H$ is explicitly given as follows.
Define the embeddings $\iota_j: \kk_j \rightarrow  M_2(\Q)$ by
\begin{equation} \label{eq:iota}
(e_1, e_2)  \iota_1(r)  = ( r e_{1}, r e_2), \quad  \iota_2(r) (e_{3}, e_1)^t = ( \bar{r} e_{3}, \bar{r} e_1)^t.
\end{equation}
Then $\iota= (\iota_1, \iota_2)$ gives the embedding from $T$ to $H$.
If $r_j = \alpha_j e_1  + (-1)^{j+1} \beta_j e_{j+1} \in E_j$, then
\begin{equation}
  \label{eq:iotaexplicit}
\iota_j(r_j) = \alpha_j \pmat{1}{0}{0}{1} + \beta_j \pmat{0}{\tfrac{d_j - 1}{4}}{1}{-1}.
\end{equation}

Extend the two real embeddings of $F$ into a CM type   $\Sigma=\{\sigma_1, \sigma_2\}$ of $E$ via
$$
\sigma_1(\sqrt{d_i}) =\sqrt{d_i} \in \H,\quad  \sigma_2(\sqrt{d_1}) =\sqrt{d_1}, \quad  \sigma_2(\sqrt{d_2}) =-\sqrt{d_2}.
$$
 Since $W_{\sigma_2} =W\otimes_{F, \sigma_2} \R \subset  V_\R$ has signature $(0, 2)$, it gives two points $z_{\sigma_2}^\pm$ in $\D$. In this case, the big CM cycles associated to $T$ as defined in \cite{BKY} and \cite{YY19} are given by
\begin{equation}
Z(W, z_{\sigma_2}^\pm) = \{ z_{\sigma_2}^\pm\} \times T(\Q)\backslash T(\A_f)/K_T  \in Z^2(X_K),
\end{equation}
and
\begin{equation} \label{eq:BigCMCycle}
Z(W) = Z(W, z_{\sigma_2}^\pm) +  \sigma(Z(W, z_{\sigma_2}^\pm)).
\end{equation}
For simplicity, we will denote  $z_{\sigma_2}$ for $z^+_{\sigma_2}$. The same calculation as in the proof of \cite[Lemma3.4]{YY19} gives the following result.

\begin{lemma} \label{lem3.4}  On $\H^{2} \cup (\H^-)^2$, one has  $z_{\sigma_2} = (\tau_1, \tau_2) \in \H^2$ and $z_{\sigma_2}^-=(\bar \tau_1, \bar \tau_2)\in (\H^- )^2$, where
$$
\tau_j=\frac{1 + \sqrt{d_j}}2.
$$
\end{lemma}

Let $\iota=(\iota_1,  \iota_2):  T \rightarrow H \subset {\GL_2\times \GL_2}$ be the natural embedding.
For $\dd \mid 24$, let $K_{\dd} \subset H(\Ab_f)$ be the compact open subgroup generated by $(T, T), (\Gamma_{\chi, \dd} \times \Gamma_{\chi, \dd})\otimes \hat{\Zb} \subset H(\Ab_f)$ and $(\nu(\hat{\Zb}^\times) \times \nu(\hat{\Zb}^\times)) \cap H(\Ab_f)$.
By the choice of $\Gamma_{\chi, \dd}$, we actually have the following result.
\begin{lemma}
  \label{lemma:compact}
Suppose $d_j < 0$ are discriminants satisfying \eqref{eqd} for $j = 1, 2$.
Then for any $\dd \mid 24$, the preimage $\iota^{-1}(K_{\dd})$ is independent of $\dd \mid 24$.
\end{lemma}
\begin{remark}
  \label{rmk:KT}
We will simply denote $\iota^{-1}(K_{\dd})$  by $K_T$.
\end{remark}
\begin{remark}
  The lemma does not require $d_j$ to be fundamental or co-prime.
\end{remark}
\begin{proof}
  Since $K_{24} \subset K_\dd \subset K_1$ for any $\dd \mid 24$, it suffices to check that $\iota^{-1}(K_{1}) = \iota^{-1}(K_{24})$. Furthermore, we know that $\Gamma(48) \subset \Gamma_{\chi, 24} \subset \Gamma_{\chi, 1} = \Gamma_0(2)$, so we only need to check the equality when tensoring with $\Zb/3\Zb$ and with $\Zb/16\Zb$. This then boils down to a short calculation with finite groups.

  To check the case modulo 3, it suffices to show that $\iota(\iota^{-1}(K_{ 1}) \otimes \Zb/3\Zb) \subset K_{24} \subset \GL_2(\Zb/3\Zb) \times \GL_2(\Zb/3\Zb)$.
  Since $\Gamma_{\chi, 1} \otimes \Zb/3 = \Gamma_0(2) \otimes \Zb/3 = \SL_2(\Zb_3)$, we have $\iota^{-1}(K_{ 1}) \otimes \Zb/3 = \iota^{-1}(H(\Zb/3\Zb))$. Therefore, \eqref{eq:iotaexplicit} implies
  $$
  \iota(\iota^{-1}(K_{ 1}) \otimes \Zb/3) =
  \left\{
    \left( \kzxz {\alpha_j} {\beta_j \frac{d_j-1}4} {\beta_j} {\alpha_j -\beta_j} \right)_{j = 1, 2} \in H(\Zb/3\Zb) : \alpha_j, \beta_j \in \Zb/3\Zb,
    \right\}
$$
and we need to show that this is contained in
\begin{align*}
K_{24} \otimes \Zb/3\Zb  &= \langle \Gamma_{\chi, 24} \times  \Gamma_{\chi, 24}, (T, T), \nu(\hat{\Zb}^\times) \times\nu(\hat{\Zb}^\times) \rangle \otimes \Zb/3\Zb \\
 &= \langle \Gamma_{\chi, 3} \times  \Gamma_{\chi, 3}, (T, T), \nu(\hat{\Zb}^\times) \times\nu(\hat{\Zb}^\times)\rangle \otimes \Zb/3\Zb  \\
&= \langle N_{\dd, 3}'\times N_{\dd, 3}', (T, T), (\smat{1}{}{}{-1}, \smat{1}{}{}{-1})\rangle \subset H(\Zb/3\Zb).
\end{align*}
Now given $r = (r_1, r_2)  \in \iota^{-1}(K_{1})\otimes \Zb/3\Zb$ with $\iota_j(r_j) = \kzxz {\alpha_j} {\beta_j \frac{d_j-1}4} {\beta_j} {\alpha_j -\beta_j}$, we know that
\begin{equation}
  \label{eq:det}
\delta := \det(\iota_j(r_j)) = \tr(\iota_j(r_j))^2 - \beta_j^2d_j\in (\Zb/3\Zb)^\times
\end{equation}
is independent of $j$.
If $\beta_j = 0$, then $\iota_j(r_j) = \pm \smat{1}{}{}{1} \in N_{\dd, 3}'$ and $\iota(r) \in K_{24} \otimes \Zb/3\Zb$.
If $\beta_1 = 0$ and $\beta_2 \neq 0$, then $\iota_1(r_1) \in N_{\dd, 3}'$ and $\delta = 1$, which implies $\tr(\iota_2(r_2)) = 0$ by \eqref{eq:det}. That means $\iota_2(r_2) \in N_{\dd, 3}'$ by \eqref{eq:N3'p}.
Finally suppose $\beta_j \neq 0$, then we can use $3 \nmid d_j$ to show that $\epsilon := \alpha_j \beta_j(\delta + 1)$ is independent of $j$. It is then straightforward to check that $T^{1-\epsilon} \smat{1}{}{}{\delta} \iota_j(r_j) \in N_{\dd, 3}'$. Therefore $\iota(r) \in K_{24}\otimes \Zb/3\Zb$.

To check the case modulo $16$, suppose $r=(r_1, r_2) \in \iota^{-1}(K_{1})\otimes \Zb/16\Zb$ with $r_j=\alpha_j e_1 + (-1)^{j+1} \beta_j e_{j+1}, \alpha_j, \beta_j \in \Zb/16\Zb$.
Then simple calculation shows that $\alpha_j - 1, \beta_j \in 2\Zb/16\Zb$.
Furthermore, $\det(\iota_j(r_j)) = \alpha_j(\alpha_j - \beta_j) - \beta_j^2 \tfrac{d_j - 1}{4} \in (\Zb/16\Zb)^\times$ is independent of $j$ since $\iota(r) \in H(\Ab_f)$, and $\beta_j^2 \tfrac{d_j - 1}{4} \equiv 0 \bmod{8}$ since $d_j \equiv 1 \bmod{8}$ and $\beta_j \in 2\Zb/16\Zb$.
Therefore,
$$
\det(\iota_j(r_j))^{-1} = \alpha_j^{-1}(\alpha_j - \beta_j)^{-1} +  \beta_j^2 \tfrac{d_j - 1}{4}.
$$
Now $r \in \iota^{-1}(K_{24})$ if and only if $\iota(r) \in K_{24}\otimes \Zb/16\Zb$, which is generated by $\nu( (\Zb/16\Zb)^\times) \times \nu((\Zb/16\Zb)^\times)$, $(T, T)\otimes \Zb/16\Zb$ and $(\Gamma_{\chi, 24} \times \Gamma_{\chi, 24}) \otimes \Zb/16\Zb \cong (\Gamma_{\chi, 8}/\Gamma(16) \times \Gamma_{\chi, 8}/\Gamma(16))$. From the natural surjection $\Gamma_{\chi, 8} /\Gamma(16) \rightarrow \Gamma_{\chi, 8}/\Gamma_8 = N_8' = N_{8, 2}'$, we see that the following claim will finish the proof: the element
$$
g_j := \nu(\det(\iota_j(r_j)))^{-1} T^{\det(\iota_j(r_j) - 1)/2} \iota_j(r_j)
$$
is in $N_8' = N_{8, 2}'$ for all $r_j = \alpha_j e_j + (-1)^{j+1} \beta_j e_{j+1}$ with $\alpha_j - 1, \beta_j \in 2\Zb/16\Zb$.
By dropping the subscript $j$ in $d_j, g_j, \alpha_j$ and $\beta_j$, we can write
$$
g = \alpha \pmat{1}{\frac{\alpha(\alpha - \beta) - 1 - \beta^2 (d-1)/4}{2}}{0}{\alpha^{-1}(\alpha - \beta)^{-1} + \tfrac{ \beta^2(d-1)}{4}} +
\beta
\pmat{\tfrac{\alpha(\alpha-\beta) - 1 - \beta^2 (d-1)/4}{2}}{ \tfrac{d-1}{4} - \tfrac{\alpha(\alpha-\beta) - 1 - \beta^2 (d-1)/4}{2}}{\alpha^{-1}(\alpha - \beta)^{-1}}{-\alpha^{-1}(\alpha - \beta)^{-1}},
$$
which is an element in $N_8 = N_{8, 2}$.
Denote $h = [h_0, h_1, h_2, h_3] := \kappa_{8, 2}(g) \in \Ac_{8, 2}$.
To show that $g \in N_8' = N_{8, 2}'$, it suffices check that $h \in \Ac_{8, 2}'$, i.e.
$$
h \perp [6, 4, 0, 2],~
h \perp [0, 2, 2, 0]
$$
and $h_0^2 - 1 \equiv h_1 + h_2 \bmod{16}$ by Lemma \ref{lemma:additive}.
All of these can be checked by hand (assuming $d \equiv 1 \bmod{8}$ and $\alpha - 1\equiv \beta \equiv 0 \bmod{2}$), and we leave the details to the reader.
\end{proof}
By \cite[Lemma 3.5]{YY19}, the map
\begin{equation} \label{eq:ClassMap}
p: T(\Q) \backslash T(\A_f) /K_{T} \rightarrow  \Cl(d_1) \times \Cl(d_2),~
 [t_1, t_2] \mapsto ([t_1], [t_2]) =([\mathfrak a_1],  [\mathfrak a_2])
 \end{equation}
is injective.
Here $\mathfrak a_j$ is the ideal of $\kk_j$ associated to $t_j$.
If $d_1, d_2$ are co-prime, then \cite[Lemma 3.8]{YY19} tells us that it is an isomorphism.
If $d_1d_2$ is not a perfect square, this subgroup can be identified with $\Gal(H/E)$ with $H$ the composite of the ring class fields $H_{d_j}$ associated to the order of discriminant $d_j$ (see Prop.\ 3.2 in \cite{Li18}).
This observation and the above lemma give the following corollary.

\begin{proposition} \label{prop:CMFormula}
Let $d_j < 0$ be co-prime, fundamental discriminants satisfying \eqref{eqd}.
For $[\af_j] \in \Cl(d_j)$, recall $f(\af_j)$ the class invariant defined in \eqref{eq:classinv}.
Then for any $s|24$
  \begin{equation}
    \label{eq:main1}
    4 s \sum_{[\af_j] \in \Cl(d_j),~ j = 1, 2}
\log \left| f(\af_1)^{24/s} - f(\af_2)^{24/s}\right|
 = \sum_{\dd|s}
\varepsilon^{24/\dd} \log |\Psi_\dd(Z(W))|,
  \end{equation}
where $\varepsilon := \varepsilon_{d_1} \varepsilon_{d_2} = (-1)^{(d_1 + d_2 - 2)/8}$ and $Z(W)$ is the  big CM cycle defined in (\ref{eq:BigCMCycle}).
\end{proposition}
\begin{proof}
We may assume $s=24$ for simplicity, as the other cases are the same.
By applying Shimura's reciprocity law, Prop.\ 22 in \cite{Gee} showed that class invariants $f(\af_j)$ for $[\af_j] \in \Cl(d_j)$ are conjugates of each other under the Galois group.
In particular, equation (18) loc.\ cit.\ implies
\begin{align*}
f(\af_j) =
\varepsilon_{d_j} (\zeta_{48}^{-1}\mathfrak f_2(\tau_j))^{\sigma_{t_j}}
&= \varepsilon_{d_j} \zeta_{48}^{- \sigma_{t_j}} \mathfrak f_2^{\sigma_{t_j} }(\tau_j^{\sigma_{t_j}})
= \varepsilon_{d_j}\zeta_{48}^{-t_j \bar{t}_j} \delta(t_j) \mathfrak f_2(\tau_j^{\sigma_{t_j}})
\end{align*}
where the class $t_j \in (E_j \otimes \Ab_f)^\times$ in $\Cl(d_j)$ is $[\af_j]$.
 Here $t_j \bar{t}_j$ can be understood to be an integer modulo 48, and
$$
\delta(t_j) = \frac{ (\sqrt 2)^{\sigma_{t_j \bar t_j}}}{\sqrt 2}
$$
is a 8th root of unit depending only on $t_j \bar t_j \bmod 8$, coming from the Fourier coefficients of $\mathfrak f_2$. Note that $\sqrt 2 =\zeta_8 +\zeta_8^{-1}$. Thus for $t = (t_1, t_2) \in T(\A_f)$, we have  $t_1 \bar t_1 =t_2 \bar t_2 $ and
\begin{align*}
\log|f(\af_1) - f(\af_2)|
& = \log| (\zeta_{48}^{-1}\mathfrak f_2(\tau_1))^{\sigma_{t_1}} - \varepsilon (\zeta_{48}^{-1}\mathfrak f_2(\tau_2))^{\sigma_{t_2}}|
= \log|\mathfrak f_2(\tau_1^{\sigma_{t_1}}) - \varepsilon \mathfrak f_2(\tau_2^{\sigma_{t_2}})|,
\end{align*}
which depends only on the image $p(t) = ([\af_1], [\af_2]) \in \Cl(d_1) \times \Cl(d_2)$. So by the isomorphism \eqref{eq:ClassMap}, we have
\begin{align*}
  \sum_{[\af_j] \in \Cl(d_j),~ j = 1, 2}
\log \left| f(\af_1)^{} - f(\af_2)^{}\right|
&=
\sum_{ t \in T(\Q) \backslash T(\A_f)/K_T} \log | \mathfrak f_2(\tau_1^{\sigma_{t_1}}) - \varepsilon \mathfrak f_2(\tau_2^{\sigma_{t_2}})| \\
&=
\sum_{(z, t) \in Z(W, \sigma_2^+)}\log | \mathfrak f_2(z_1) - \varepsilon \mathfrak f_2(z_2)|_{(z_1 z_2)=[(z, t)]}.
\end{align*}
As other three orbits are Galois conjugates of $Z(W, \sigma_2^+)$, the sum over other orbits are the same as this one. Now the desired identity follows from Theorem \ref{thm:Blift}.
\end{proof}

\section{Incoherent Eisenstein Series and the proof of Yui-Zagier conjecture}
\label{sec:local}

In this section, we will use the big CM value formula of Bruinier-Kudla-Yang (\cite{BKY} and \cite[Theorem 2.6]{YY19}) to prove the factorization formula  for $\Psi_\dd(Z(W))$ and the conjecture of Yui-Zaiger. To do it, we need to review the associated incoherent Eisenstein series and compute their Fourier coefficients.

\subsection{Incoherent Eisenstein series.} \label{sect:Eisenstein}

Let $F=\Q(\sqrt D)$,  $E=\Q(\sqrt{d_1}, \sqrt{d_2})$, and $W=\kk$ with $F$-quadratic form $Q_F(x) = \frac{\dd x \bar x}{\sqrt D}$ as in Section \ref{sect:BigCM}. Here $D=d_1d_2$. Let $\chi_{E/F}$ be the quadratic Hecke character of $F$ associated to $E/F$. Then  there is a $\SL_2(\A_F)$-equivariant map
\begin{equation}
\lambda=\prod \lambda_v : S(W(\A_F)) \rightarrow  I(0, \chi_{E/F}), \quad  \lambda(\phi) (g) = \omega(g) \phi(0).
\end{equation}
Here $I(s, \chi_{E/F}) =\Ind_{B_{\A_F}}^{\SL_2(\A_F) }\chi_{E/F} |\cdot|^s$ is the principal series, whose sections (elements) are smooth functions $\Phi$  on $\SL_2(\A_F)$ satisfying the condition
$$
\Phi(n(b) m(a) g,s )= \chi(a)|a|^{s+1}\Phi(g, s), \quad b \in \A_F,  \hbox{ and  } a \in \A_F^\times.
$$
Here $B =NM$ is the standard Borel subgroup of $\SL_2$. Such a section is called factorizable if $\Phi=\otimes \Phi_v$ with $\Phi_v\in I(s, \chi_v)$. It is called standard if $\Phi|_{\SL_2(\hat{\OO}_F) \SO_2(\R)^{2}} $ is independent of $s$. For a standard  section $\Phi \in I(s, \chi)$, its associated Eisenstein series is defined as
$$
E(g, s, \Phi) = \sum_{\gamma \in B_F \backslash \SL_2(F)} \Phi(\gamma g, s)
$$
for $\Re(s) \gg 0$.

For $\phi \in S(V_f) =S(W_f)$, let $\Phi_f $ be the standard section associated to $\lambda_f(\phi) \in I(0, \chi_f)$. For each real embedding $\sigma_i: F \hookrightarrow \R$,  let  $\Phi_{\sigma_i} \in I(s, \chi_{\C/\R})=I(s, \chi_{E_{\sigma_i}/F_{\sigma_i}})$ be the unique `weight one' eigenvector of $\SL_2(\R)$ given by
$$
\Phi_{\sigma_i}(n(b)m(a) k_\theta) = \chi_{\C/\R}(a) |a|^{s+1} e^{i  \theta},
$$
for $b \in \R$, $a\in \R^\times$, and $k_\theta =\kzxz {\cos\theta} {\sin \theta} {-\sin \theta} {\cos \theta} \in \SO_2(\R)$. We define  for $\vec\tau =(\tau_1, \tau_{2}) \in \H^{2}$
$$
E(\vec\tau, s, \phi) =  \norm(\vec v)^{-\frac12} E(g_{\vec\tau}, s, \Phi_f \otimes  (\otimes_{1 \le i \le 2}\Phi_{\sigma_i} )),
$$
where $\vec v =\hbox{Im}(\vec\tau)$, $\norm(\vec v) =\prod_i v_i$, and $g_{\vec\tau} = (n(u_i) m(\sqrt{v_i}))_{1\le i \le 2}$. It is a (non-holomorphic) Hilbert modular form of parallel weight $1$ for some congruence subgroup of $\SL_2(\OO_F)$. Following \cite{BKY}, we further normalize
$$
E^*(\vec\tau, s, \phi) = \Lambda(s+1, \chi_{E/F}) E(\vec\tau, s, \phi),
$$
where
\begin{equation} \label{eq:L-series}
\Lambda(s, \chi) =D^{\frac{s}2} (\pi^{-\frac{s+1}2}
\Gamma(\tfrac{s+1}2))^{2} L(s, \chi_{E/F}).
\end{equation}
According to \cite{YY19}, The Eisenstein series is incoherent in the sense of Kudla, and   $E^*(\vec\tau, 0, \phi)=0$ automatically. Write  its central derivative via Fourier expansion
\begin{equation}
E^{*, \prime}(\vec\tau, 0, \phi)=\sum_{t \in F} a(\vec v, t, \phi) q^t,  \quad q^t =e(\tr(t\tau)).
\end{equation}
Then it is known that $a(t, \phi)= a(\vec v, t, \phi)$ is independent of the imaginary part of $\vec{\tau} \in \Hb^2$ when $t $ is totally positive. Finally, when $\phi = \otimes_{\pf} \phi_{\pf} \in S(V_f)$ is factorizable, one has for $t \gg 0$
(the factor $-4$ comes from  \cite[Proposition 2.7(1)(2)]{YY19})
\begin{equation}
a( t, \phi) =-4 \frac{d}{ds} \left( \prod_{\mathfrak p}  W_{t, \pf}(s, \phi)\right)|_{s=0}
\end{equation}
where
\begin{equation}
W_{t, \pf}(s, \phi) :=
\int_{F_\pf} \omega(wn(b))(\phi_\pf)(0) |a(wn(b))|^s_{\pf} \psi_\pf(-tb) db
\end{equation}
are the local Whittaker functions.
Specializing Theorem 5.2 in \cite{BKY} give us the following result.

\begin{theorem}[BKY]
  \label{thm:bigCM}
Let $d_j < 0$ be fundamental discriminants satisfying $d_j \equiv 1 \bmod{8}$ and $3 \nmid d_j$.
For any $ 1\neq \dd \mid 24$, let $\phi_\dd \in S(V_\dd(\A_f)$ be associated to $\uf_\dd$.  Then we have
\begin{equation}
  \label{eq:bigCM}
  - \log |\Psi_{\dd}(Z(W))|^4 = C(W, K) \sum_{t \in F^\times, t \gg 0, \tr(t) = 1/\dd} a(t, \phi_\dd),
\end{equation}
where $Z(W)$ is the big CM 0-cycle associated to $d_1, d_2$  defined in \eqref{eq:BigCMCycle}, and
 $C(W, K) = \frac{\deg( Z(W, z_{\sigma_2}^\pm))}{\Lambda(0, \chi)} = 2 $.
\end{theorem}

The rest of this section is to compute $a(t, \phi_\dd)$ and prove the Yui-Zagier conjecture. Unfortunately,
 $\phi_\dd$ is not factorizable over $F$ at the places dividing $(\dd, 6)$. Instead, we have
$$
\phi_\dd = \phi_{\dd, 2} \phi_{\dd, 3}\otimes_{\pf \nmid 6} \phi_{\dd, \pf}.
$$
Then for $\pf \nmid 6$, the contribution of $W_{t, \pf}(s, \phi_\dd)$ is the same as in the case of Gross-Zagier (see \cite{YY19}).
Therefore, we are left with the local calculations at 2 and 3.
Since 2 splits completely in $E/\Qb$, we denote $\pf_1, \pf_2$ the two primes in $F$ above 2. Also denote $\pf_3, \pf_3'$ the primes in $F$ above $3$.
They are the same if and only if $\left( \frac{D}{3} \right) = -1$.
The local calculations in section \ref{subsec:local} leads to the following result.
\begin{theorem}
  \label{thm:main1}
Let $d_j < 0$ and $\dd$ be the same as in Theorem \ref{thm:bigCM}, and let $\varepsilon=\varepsilon_1\varepsilon_2 =(-1)^{\frac{d_1+d_2-2}8}$.
Suppose $t = \frac{a + \sqrt{D}}{2\dd \sqrt{D}} \in F^\times$ is totally positive with $a \in \Qb$. Then
\begin{equation}
  \label{eq:atd}
\begin{split}
  a( t, \phi_\dd ) &= -\dd_2
\varepsilon^{24/\dd}
\delta_2(\dd_2, t) \times \\
&\quad \Bigg\{    \sum_{\substack{ \pf \text{ inert in } E/F\\ \pf \nmid 3}}(1 + \ord_{\pf}(t \sqrt{D})) \rho^{(6)}(t \sqrt{D} \pf^{-1})
 \delta_3(\dd_3, t)
\log (\Nm(\pf))\\
&\qquad     +
\log 3
\sum_{\substack{ \pf \text{ inert in } E/F\\ \pf \mid 3}}
\rho^{(2)}(t \sqrt{D} \pf^{-1})
 \delta'_3(\dd_3, t)
\Bigg\} \\
\end{split}
\end{equation}
if $a \in \Zb$ and zero otherwise.
The functions $\delta_p(\dd_p, t)$ and $\delta'_3(\dd_3, t)$ are defined by
\begin{align*}
  \delta_2(1, t)
  &:= 2 (v_2(\Nm(t)) - 1),~ v_2(\Nm(t)) \ge 2,\\
\delta_2(2, t) &:=
\begin{cases}
1, & v_2(\Nm( t )) = 0,\\
v_2(\Nm(t )) - 3, & v_2(\Nm( t)) \ge 1,
\end{cases} \\
\delta_2(4, t) &:=
\begin{cases}
\mp 1, & \Nm(2t) \equiv \pm 1 \bmod{4},\\
1, & v_2(\Nm(t)) = 0, \\
v_2(\Nm(t)) - 3, & v_2(\Nm(t)) \ge 1,
\end{cases} \\
\delta_2(8, t) &:=
\begin{cases}
1, & \Nm(4t) \equiv 3 \bmod{8},\\
-1, &\Nm(4t) \equiv 7 \bmod{8},\\
\mp 1, & \Nm(2t) \equiv \pm 1 \bmod{4},\\
1, &v_2(\Nm(t)) = 0,\\
v_2(\Nm(t)) - 3, & v_2(\Nm(t)) \ge 1,
\end{cases}
                   \end{align*}
                   \begin{align*}
    \delta_3(1, t) &:= \rho_3(t),~
v_3(\Nm(t)) \ge 0,\\
\delta_3(3, t) &:=
\begin{cases}
2 - \frac{3}{4} \lp   1- \lp \frac{d_1}{3}  \rp  \rp \lp   1- \lp \frac{d_2}{3}  \rp \rp, & \Nm(3t) \equiv 1 \bmod{3},\\
-1, & \Nm(3t) \equiv 2 \bmod{3},\\
\lp 1 + \lp \frac{d_1}{3} \rp\rp v_3(\Nm(t))
+ 1 - \lp \frac{d_1}{3} \rp^{ v_3(\Nm(t)) - 1} ,& v_3( \Nm(3t)) \ge 1,
\end{cases}\\
\delta'_3(\dd_3, t) &:=
    \begin{cases}
v_3(\Nm(t)) + 1& \text{ if } \dd_3 = 1,\\
2v_3(\Nm(t)) + 3,& \text{ if } \dd_3 = 3,
    \end{cases}
\end{align*}
and zero otherwise.
Here $\rho^{(M)}(\af) := \rho(\af^{(M)})$ is the number of integral ideals of $E$ with relative norm (to $F$) $\af^{(M)}$, and $\rho_{M}(\af) := \rho(\af/\af^{(M)})$. Here  $\af^{(M)}$ is the prime to $M$ part of an ideal $\af$.
\end{theorem}

\begin{proof}

To evaluate $a(t, \phi)$, it is convenient to introduce the `Diff' set of Kudla.
  For a totally positive $t \in F^\times$, define
$$
\mathrm{Diff}(W, t) := \{ \pf: W_\pf \text{ does not represent } t\}.
$$
Then $|\mathrm{Diff}(W, t)|$ is finite and odd.
Furthermore if $\#\mathrm{Diff}(W, t) > 1$, then $a(t, \phi)$ vanishes.
This is also the case with the expression on the right hand side of \eqref{eq:atd}, since $\delta_3(\dd_3, t) = 0$ if $\pf_3, \pf_3' \in \mathrm{Diff}(W, t)$ and $\rho^{(6)}(t\sqrt{D}\pf) = 0$ for every inert $\pf$ if $\mathrm{Diff}(W, t)$ contains two primes coprime to 6.
Therefore, we can suppose that $\mathrm{Diff}(W, t) = \{\pf_0\}$ for a single prime $\pf_0$ of $F$.
In that case, every term with $\pf \neq \pf_0$ on the right hand side of \eqref{eq:atd} vanishes.
Given $t = \frac{ a + \sqrt{D}}{2\dd \sqrt{D}} \in F$ totally positive, the Fourier coefficient $a(t, \phi)$ is given by
$$
a(t, \phi_\dd) = -4 \frac{d}{ds} \lp
\frac{W^{*}_{t, 2}(s, \phi_{\dd, 2})}{\gamma(W_{2})}
\frac{W^{*}_{t, 3}(s, \phi_{\dd, 3})}{\gamma(W_{3})}
\prod_{\pf \nmid 6 \infty}
\frac{W^{*}_{t, \pf}(s, \phi_{\pf})}{\gamma(W_{\pf})} \rp \mid_{s = 0},
$$
where $\gamma(W_{\pf})$ is the Weil index of $W_\pf$ (see e.g.\ Prop.\ 2.7 in \cite{YY19}).

Recall that $\pf_1, \pf_2$ and $\pf_3, \pf_3'$ are primes in $F$ above 2 and 3 respectively.
Since $\pf_1, \pf_2$ splits in $E$, they are not in $\mathrm{Diff}(W, t)$ for any $t$.
However, $\pf_3$ and $\pf_3'$ could appear in some Diff set if they are inert in $E/F$.
Now, if $\pf_0 \nmid 3$, then we can proceed as in the proof of Theorem 1.1 in \cite{YY19} to obtain
$$
a(t, \phi_\dd) =
-2
\frac{W^{*}_{t, 2}(0, \phi_{\dd, 2})}{\gamma(W_{2})}
\frac{W^{*}_{t, 3}(0, \phi_{\dd, 3})}{\gamma(W_{3})}
\rho^{(6)}(\dd \sqrt{D} t \pf_0^{-1}) (1 + \ord_{\pf_0}(t \sqrt{D})) \log \Nm(\pf_0).
$$
By Lemma \ref{lemma:Whitt2} below and equation \eqref{eq:varepsilon}, we can replace $2\frac{W^{*}_{t, 2}(0, \phi_{\dd, 2})}{\gamma(W_{2})}$ with $\varepsilon^{24/\dd}
 \dd_2  \delta_2(\dd_2, t)$.
 By Lemmas \ref{lemma:Whitt3a}, \ref{lemma:Whitt3b} and \ref{lemma:Whitt3c} below, we can replace $\frac{W^{*}_{t, 3}(0, \phi_{\dd, 3})}{\gamma(W_{3})}$ with $\delta_3(\dd_3, t)$ and arrive at the right hand side.

 If $\mathrm{Diff}(W, t) = \{\pf_0\}$ with $\pf_0 \mid 3$, then $(\tfrac{d_j}{3}) = -1$ and we can write
$$
a(t, \phi_\dd) =
-4
\frac{W^{*}_{t, 2}(0, \phi_{\dd, 2})}{\gamma(W_{2})}
\frac{{W^{*,}}'_{t, 3}(0, \phi_{\dd, 3})}{\gamma(W_{3})}
\rho^{(6)}(\dd \sqrt{D} t).
$$
 we can again replace $2\frac{W^{*}_{t, 2}(0, \phi_{\dd, 2})}{\gamma(W_{2})}$ with $\varepsilon^{24/\dd}
 \dd_2  \delta_2(\dd_2, t)$ and apply Lemma \ref{lemma:Whitt3c} to replace $2\frac{{W^{*,}}'_{t, 3}(0, \phi_{\dd, 3})}{\gamma(W_{3})}$ with $\delta'_3(3, t)\log 3$.
This finishes the proof.
\end{proof}

In \cite{YZ}, Yui and Zagier derived the conjectural factorization of $\Nm_{H/\Qb}(f(\tau_1)^{24/s} - f(\tau_2)^{24/s})$ from the conjectural factorization of $\Nm_{H/\Qb}(\Phi_{24/s}(f(\tau_1), f(\tau_2)))$, where $\Phi_r$ the $r$\tth cyclotomic polynomial.
Since $\mathfrak{F}(m)$ is the power of a rational prime $\ell$, we can define
\begin{equation}
  \label{eq:gamma}
  \mathfrak{F}(m) = \ell^{\gamma(m)},
\end{equation}
where $\gamma(m) = \prod_{p \mid m}\gamma_p(m)$ with
\begin{equation}
  \label{eq:gammap}
  \gamma_p(m) :=
  \begin{cases}
    \ord_p(m) + 1,& \text{ if } \varepsilon(p) = 1,\\
     1,& \text{ if } \varepsilon(p) = -1 \text{ and } 2\mid \ord_p(m),\\
    \frac{\ord_p(m) + 1}{2},& \text{ if } \varepsilon(p) = -1 \text{ and } 2\nmid \ord_p(m) \text{ (i.e. } p = \ell).
  \end{cases}
\end{equation}
The conjecture is then in term expressed in terms how $\gamma_2(m)$ and $\gamma_3(m)$ decomposes, which are summarized in two tables (see page 1653 of \cite{YZ}). The theorem above is equivalent to this formulation of the conjecture. As in \cite{YZ}, one can give a conjecture with an equivalent, but simplified expression.
This is the content of Conjecture \ref{conj:YZ}, which we prove now.

\begin{proof}[Proof of Theorem \ref{thm:intro}]
  By Prop.\ \ref{prop:CMFormula} and Theorems \ref{thm:bigCM}, \ref{thm:main1}, we can write
  \begin{align*}
4s & \sum_{[\af_j] \in \Cl(d_j),~ j = 1, 2}  \log   |f(\af_1)^{24/s} - f(\af_2)^{24/s})|^4 =
- 2\sum_{\dd \mid s}
\varepsilon^{24/\dd}
 \sum_{t \in F^\times, t \gg 0 , \tr(t) = 1/\dd} a(t, \phi_{\dd})\\
&=
2\sum_{\dd \mid s}
 \sum_{t \in F^\times, t \gg 0 , \tr(t) = 1/\dd}
\dd_2
\delta_2(\dd_2, t) \times
 \Bigg\{
\log 3
\sum_{\substack{ \pf \text{ inert in } E/F\\ \pf \mid 3}}
\rho^{(2)}(t \sqrt{D} \pf^{-1})
 \delta'_3(\dd_3, t)
\\
&\qquad \qquad + \sum_{\substack{ \pf \text{ inert in } E/F\\ \pf \nmid 3}}(1 + \ord_{\pf}(t \sqrt{D})) \rho^{(6)}(t \sqrt{D} \pf^{-1})
 \delta_3(\dd_3, t)
\log (\Nm(\pf))
\Bigg\} \\
&=
2 \sum_{4\sqrt{D} \tilde{t} \in \mathcal{O}_F , \tilde{t} \gg 0 , \tr(\tilde{t}) = 1/2}
\sum_{\dd \mid s}
\dd_2
\delta_2(\dd_2, \tfrac{2\tilde{t}}{\dd}) \times
 \Bigg\{
\log 3
\sum_{\substack{ \pf \text{ inert in } E/F\\ \pf \mid 3}}
\rho^{(2)}(\tfrac{2\tilde{t}}{\dd} \sqrt{D} \pf^{-1})
 \delta'_3(\dd_3, \tfrac{2\tilde{t}}{\dd})
\\
&\qquad \qquad + \sum_{\substack{ \pf \text{ inert in } E/F\\ \pf \nmid 3}}(1 + \ord_{\pf}(\tilde{t} \sqrt{D})) \rho^{(6)}(\tilde{t} \sqrt{D} \pf^{-1})
 \delta_3(\dd_3, \tfrac{2\tilde{t}}{\dd})
\log (\Nm(\pf))
\Bigg\}
  \end{align*}
By Theorem \ref{thm:main1}, we have $\sum_{\dd_2 \mid s_2} \dd_2 \delta_2(\dd_2, \tfrac{2\tilde{t}}{\dd}) = \sum_{\dd_2 \mid s_2} \dd_2 \delta_2(\dd_2, \tfrac{2\tilde{t}}{\dd_2})$ and
\begin{align*}
\sum_{\dd_2 \mid 1} \dd_2 \delta_2(\dd_2, \tfrac{2\tilde{t}}{\dd_2}) &= 2 (    v_2(\Nm(\tilde{t})) +1) = 2 \gamma_2(\Nm(\tilde{t})),\\
\sum_{\dd_2 \mid 2} \dd_2 \delta_2(\dd_2, \tfrac{2\tilde{t}}{\dd_2}) &=
4
  \begin{cases}
    1,& \text{ if } v_2(\Nm(\tilde{t})) = 0,\\
    v_2(\Nm(\tilde{t})) - 1,& \text{ if } v_2(\Nm(\tilde{t})) \ge 1,
  \end{cases}\\
\sum_{\dd_2 \mid 4} \dd_2 \delta_2(\dd_2, \tfrac{2\tilde{t}}{\dd_2}) &=
8
  \begin{cases}
    1,& \text{ if } v_2(\Nm(\tilde{t})) \equiv -1\bmod{4}\\
& \text{ or }v_2(\Nm(\tilde{t})) = 2,\\
    v_2(\Nm(\tilde{t})) - 3,& \text{ if } v_2(\Nm(\tilde{t})) \ge 3,
  \end{cases}\\
\sum_{\dd_2 \mid 8} \dd_2 \delta_2(\dd_2, \tfrac{2\tilde{t}}{\dd_2}) &=
16
  \begin{cases}
    1,& \text{ if } v_2(\Nm(\tilde{t})) = 4,\\
& \text{ or } v_2(\Nm(\tilde{t})) \equiv 12 \bmod{16},\\
& \text{ or } v_2(\Nm(\tilde{t})) \equiv 3 \bmod{8},\\
    v_2(\Nm(\tilde{t})) - 5,& \text{ if } v_2(\Nm(\tilde{t})) \ge 5,
  \end{cases}\\
\end{align*}
From this, it is easy to check that
\begin{equation}
  \label{eq:count2}
\sum_{\dd_2 \mid s_2} \dd_2 \delta_2(\dd_2, \tfrac{2\tilde{t}}{\dd})
= 2 s_2
\sum_{\substack{r_2 \mid s_2,~ m := D\Nm(\tilde{t}/r_2) \in \Zb,\\ m \equiv 3 \bmod{s_2/r_2}}}
\gamma_{2}(m),
\end{equation}
where we write $s = s_2 s_3$ with $s_p$ the $p$-part of $s$.
Similarly, we also have
\begin{equation}
  \label{eq:count3}
  \begin{split}
    &\kappa_3(s)s_3
\sum_{\substack{r_3 \mid s_3,~ m := D\Nm(\tilde{t}/r_3) \in \Zb,\\ m \equiv d_1 + d_2 - 1 \bmod{s_3/r_3}}}
\gamma_{3}(m)
\\
&=
  \begin{cases}
    \frac{1}{2} \sum_{\dd_3 \mid s_3}  \sum_{\pf \mid 3} \rho_3(\pf^{-1}\tilde{t}/3) \delta'_3(\dd_3, \tfrac{2\tilde{t}}{\dd}) , & \text{ if }
(\tfrac{d_1}{3})  = (\tfrac{d_2}{3}) = -1 \text{ and } 2 \nmid v_3(\Nm(\tilde{t})), \\
    \sum_{\dd_3 \mid s_3}  \delta_3(\dd_3, \tfrac{2\tilde{t}}{\dd}) , & \text{ otherwise, }
  \end{cases}
  \end{split}
\end{equation}
where $\kappa_3(s) \in \{1, \tfrac{1}{2}\}$ is the constant defined in \eqref{eq:kappa3}.
So suppose $\mathrm{Diff}(W, \tilde{t}) = \{\pf_0\}$ with $\ell = \Nm(\pf_0)$. Then substituting in these gives us
\begin{align*}
&  \sum_{\dd \mid s}
\dd_2
\delta_2(\dd_2, \tfrac{2\tilde{t}}{\dd})
 \Bigg\{
\log 3
\sum_{\substack{ \pf \text{ inert in } E/F\\ \pf \mid 3}}
\rho^{(2)}(\tfrac{2\tilde{t}}{\dd} \sqrt{D} \pf^{-1})
 \delta'_3(\dd_3, \tfrac{2\tilde{t}}{\dd})
\\
& \qquad \qquad + \sum_{\substack{ \pf \text{ inert in } E/F\\ \pf \nmid 3}}(1 + \ord_{\pf}(\tilde{t} \sqrt{D})) \rho^{(6)}(\tilde{t} \sqrt{D} \pf^{-1})
 \delta_3(\dd_3, \tfrac{2\tilde{t}}{\dd})
\log (\Nm(\pf))
\Bigg\} \\
=\quad &
4s \sum_{\substack{r \mid s,~ m := D \Nm(\tilde{t}/r) \in \Zb\\ m\equiv 19D \bmod{s/r}}}
\log (\ell )
\prod_{p \mid m} \gamma_p(m)
= 4s \sum_{\substack{r \mid s,~ m := D \Nm(\tilde{t}/r) \in \Zb\\ m\equiv 19D \bmod{s/r}}} \log \mathfrak{F}(m).
\end{align*}
After writing $\tilde{t} = \frac{\sqrt{D} + a}{4\sqrt{D}}$ with $a \in \Zb$ in the summation, we obtain equation \eqref{eq:YZconj}.
\end{proof}

\subsection{Local Calculations.}
\label{subsec:local}
We first need to write $\phi_{\dd, p}$  as a linear combination of $ \otimes_{\mathfrak p|p} \phi_{\mathfrak p}$ for some $\phi_{\mathfrak p} \in S(E_{\mathfrak p})=S(W_{\mathfrak p})$.
\subsubsection{$p = 2$.}  In this subsection, we deal with the case $p=2$.
Since $d_j \equiv 1 \bmod{8}$, the prime 2 splits completely.
We fix $\delta, \delta_j \in \Zb_2^\times$ such that
\begin{equation}
  \label{eq:delta}
\delta^2 = D,~ \delta_j^2 = d_j,~ \delta_1\delta_2 = \delta.
\end{equation}
We also denote
$$
\overline{\delta_j} := - \delta_j,~ \delta' := - \delta.
$$
Note that
\begin{equation}
  \label{eq:varepsilon}
  \varepsilon_{d_1}  \varepsilon_{d_2} = \lp \frac{2}{\delta} \rp.
\end{equation}
For $i = 1, 2$, let $\pf_i$ be the two primes in $F$ above 2, and $\Pf_i, \overline{\Pf_i}$ the two primes in $E$ above $\pf_i$.
Then the local fields $E_{\Pf_i}$ and $E_{\overline{\Pf_i}}$ are isomorphic to $\Qb_2$ via the map
\begin{align*}
\sigma_i:  F_{\pf_i} &\cong \Qb_2,~ \sqrt{D} \mapsto (-1)^{i} \delta,\\
\sigma_i:  E_{\Pf_i} &\cong \Qb_2,~ \sqrt{D} \mapsto (-1)^{i} \delta,~ \sqrt{d_j} \mapsto (-1)^{(i-1)(j-1)} \delta_j, \\
\sigma_i:  E_{\overline{\Pf_i}} &\cong \Qb_2,~ \sqrt{D} \mapsto (-1)^{i} \delta,~ \sqrt{d_j} \mapsto - (-1)^{(i-1)(j-1)} \delta_j.
\end{align*}
Under these identification, $W_2 =W\otimes_\Q \Q_2 =W_{\mathfrak p_1} \oplus W_{\mathfrak p_2}$ with
$$
W_{\mathfrak p_i} = E_{\mathfrak p_i} =E_{\mathfrak P_i} \oplus E_{\overline{\mathfrak P_i}} \cong \Q_2^2,  \quad Q_{\mathfrak p_i} (y_1, y_2)  = (-1)^{i} \frac{\dd}{\delta} y_1 y_2.
$$

Now we identify the $Q_2$-quadratic space
\begin{equation}
\label{eq:map2}
  \begin{split}
    (V\otimes_\Q \Q_2, Q) & \cong  (E_{\mathfrak p_1}, Q_{\mathfrak p_1}) \oplus  (E_{ \mathfrak p_2}, Q_{\mathfrak p_2}),
\\
\smat{x_3}{x_1}{x_4}{x_2} &\mapsto (\sigma_1(x), \sigma_1(\bar{x}), \sigma_2(x), \sigma_2(\bar{x})),
  \end{split}
\end{equation}
with $x = x_1 + x_2 \tfrac{-1 + \sqrt{d_1}}{2} + x_3 \tfrac{1 + \sqrt{d_2}}{2} + x_4 \tfrac{-1 + \sqrt{d_1}}{2} \tfrac{1 + \sqrt{d_2}}{2} \in W_2$. Under this isomorphism, the lattice $L_{\dd, 2} := L_\dd \otimes \Zb_2$ is mapped onto
$$
\tilde{L} := \left\{y = (y_1, y_2, y_3, y_4) \in \Zb_2^4: \sum y_i \in 2\Zb_2\right\},
$$
The $\Q_2$-quadratic form $\tilde{Q}_\dd$ on $\tilde{L}$ is given by
$$
\tilde{Q}_\dd(y) := -\frac{\dd}{\delta} \lp{y_1 y_2} - y_3 y_4 \rp = Q_{\mathfrak p_1}(y_1, y_2) + Q_{\mathfrak p_2}(y_3, y_3) .
$$

Let
$$
L_0=(2\Z_2)^4 =2\mathcal O_{E_{\mathfrak p_1}} \oplus  2\mathcal O_{E_{\mathfrak p_2}} =\tilde M_1 \oplus \tilde M_2
$$
 with $\tilde M_i$ being the $\mathcal O_{F_{\mathfrak p_i}}$-lattice $2\mathcal O_{E_{\mathfrak p_i}}$. Then
$$
L_0 \subset \tilde L \subset \tilde L' \subset L_0'=\frac{1}{4\dd_2} L_0,
$$
and
$$
\tilde L' =\left\{ y= \frac{1}{2\dd_2} (y_1, y_2, y_3, y_4)\in \frac{1}{2\dd_2}\Z_2^4:\,  y_i + y_j \equiv 0 \bmod 2\right\}.
$$
Notice that
$$
\phi_{\tilde L} = \sum_{\substack{ y_i \in \Z/2\Z \\ \sum y_i =0}} \phi_{(y_1, y_2)+\tilde M_1} \otimes \phi_{(y_3, y_4)+\tilde M_2},
$$
where $\phi_A=\cha(A)$ for $A \subset W_2$.
To apply the general formula in \cite{YYY}, we define  $M_i =\Z_2^2$ with quadratic form  $Q_i(y_1 y_2) = (-1)^{i} \frac{4\dd}{\delta }y_1 y_2$. Then $(M_i, Q_i) \cong (\tilde{M}_i, Q_{\pf_i})$ via scaling by 2. For any $\mu \in (\Qb_2/\Zb_2)^2$, we denote
$$
\phi_\mu = \mathrm{char}(\mu + \Zb_2^2)
$$
and view it as an element in $S(M_i)$ for both $i = 1, 2$ if $\mu \in (\frac{1}{4\dd_2}\Zb_2/\Zb_2)^2$.

Now, we can apply this map to the Schwartz function $\phi_{\dd, 2} \in S(L_{\dd, 2})$ associated to $\uf_{\dd, 2}$, which we denote by $\tilde{\phi}_{\dd, 2} \in S(\tilde{L}) \subset S({L}_0) \cong S(M_1 \oplus M_2)$.
The image $\tilde{\phi}_{\dd, 2}$ will depend on the choice of $\delta \bmod{\dd_2}$.
We have listed them as follows.

\begin{lemma}
  \label{lemma:decompose2}
  For $\delta_j \in \Zb_2^\times$ and $\delta = \delta_1 \delta_2 \in \Zb_2^\times$, we have
\begin{equation}
  \label{eq:phidfactor}
  \tilde{\phi}_{\dd, 2}  =
  \begin{cases}
\phi_{\tilde{L}}
,& \dd_2 = 1,\\
    \phi_0 \otimes \phi_{1, \delta} + \phi_{1, -\delta} \otimes \phi_0, & \dd_2 = 2,\\
  2\lp (\phi_0 + \phi_{1, -\delta}) \otimes \phi_{2, \delta}
 + \phi_{2, -\delta} \otimes (\phi_0 + \phi_{1, \delta})\rp,& \dd_2 = 4,\\
\lp\frac{2}{\delta \dd_3}\rp
4 \lp
 \begin{split}
&    (\phi_0+ \phi_{1, \delta}) \otimes \phi_{3, \delta^{}} + \phi_{3, -\delta^{}} \otimes (\phi_0 + \phi_{1, -\delta})  \\
&    + \phi_{2, \delta} \otimes \phi_{3, 5\delta^{}}  + \phi_{3, -5\delta^{}} \otimes \phi_{2, -\delta}
 \end{split}
\rp,
    & \dd_2 = 8,
  \end{cases}
\end{equation}
  where for $j = 1, 2, 3, r \in (\Zb/2^j\Zb)^\times$
  \begin{align*}
    \phi_0 := \sum_{k \in \Zb/2\Zb} \phi_{\tfrac{1}{2}(k, k)} - \phi_{\tfrac{1}{2}(k, k+1)},~
  \phi_{j, r} &:= \sum_{a \in (\Zb/2^{j+1})^\times} \phi_{\mu(a; r, j)}
  - \phi_{\mu(a; r + 2^j; j)},
\end{align*}
are elements in $S({M}_i)$ with $\mu(a; r, j) := \tfrac{1}{2^{j+1}}(a, ra^{-1}) \in  (\Qb_2/\Zb_2)^2$.
\end{lemma}

\begin{remark}
Note that the support of $\uf_{24, 2}$ is the support of $\uf_{8, 2}$ after scaling by $\dd_3$.
This does not affect $\phi_{j, r}$ for $j = 1, 2$ but introduces the factor $(\frac{2}{\dd_3})$ when $j = 3$, since $\phi_{3, r c^2} = \lp \frac{2}{c} \rp \phi_{3, r}$ for any odd integer $c$.
Therefore this factor appears above when $\dd_2 = 8$.
  \end{remark}
\begin{proof}
One can use Lemma \ref{lemma:additive} to check that the cosets on the right indeed appear. Then we have all of them by counting.
\end{proof}

Now, we can apply the general Whittaker function formulas in \cite{YYY} to obtain.
\begin{lemma}
  \label{lemma:Whitt2}
Let $\delta_2(\dd_2, t)$ be defined as in Theorem \ref{thm:main1}.
Then we have
  $$
\frac{W^*_t(0, \tilde{\phi}_{\dd, 2})}{\gamma(W_2)} =
\lp \frac{2}{\delta} \rp^{24/\dd}
\frac{\dd_2}{2}  \delta_2(\dd_2, t)
  $$
for all totally positive $t \in F^\times$ with $\tr(t) = \frac{1}{\dd}$.
\end{lemma}

\begin{proof}
  This can be checked case by case. For $\dd_2 = 1$, this was already done in \cite{YY19}.
Otherwise, we can apply propositions 5.3 and 5.7 in \cite{YYY} after scaling the lattice by 2 and the quadratic form by 4 (i.e.\ variant 2 in \cite{YYY}).
We write $t_i = \sigma_i(t) \in \Qb_2$ and suppose $o(t_1) \ge o(t_2)$ with $o(t_i)$ the 2-adic valuation of $t_i \in \Qb_2$.
The case $o(t_1) \le o(t_2)$ will be exactly the same.
The tables below contain the non-zero values of $\frac{W^*_{t_i}(0, \phi_{\mu_i})}{\gamma(W_{\pf_i})}$ for $i = 1$.
\begin{table}[h]
\begin{tabular}{|c|c|c|c|}
 \hline
 \diagbox{$o(t_1)$}{$\mu_1$} & $(0, 0)$ & $(\tfrac{1}{2}, \tfrac{1}{2})$ & $(\tfrac{1}{2}, 0)$ \\
 \hline
 $1$ & 0& 1& 0\\
 \hline
 $\ge 2$ & $o(t_1) - 2$ & 0& $1$\\
 \hline
 \end{tabular}
\caption{$\dd_2 = 2, \beta = -8\dd_3 \delta^{-1} $}
\begin{tabular}{|c|c|c|c|c|c|}
 \hline
 \diagbox{$o(t_1)$}{$\mu_1$} & $(0, 0)$ & $(\tfrac{1}{2}, \tfrac{1}{2})$ & $(\tfrac{1}{2}, 0)$ & $(\tfrac{a}{4}, -\tfrac{a^{-1}\delta}{4})$ & $(\tfrac{a}{4}, \tfrac{a^{-1}(-\delta + 2)}{4})$ \\
 \hline
 $t_1 \in \dd_3 + 4\Zb_2$ & 0& 0& 0 & $1/2$ & 0 \\
 \hline
 $t_1 \in -\dd_3 + 4\Zb_2$ & 0& 0& 0 & 0 & $1/2$\\
 \hline
 $2$ & 0& 1& 0& 0 & 0\\
 \hline
 $\ge 3$ & $o(t_1) - 3$ & 0& $1$& 0 & 0\\
 \hline
 \end{tabular}
\caption{$\dd_2 = 4, \beta = -16\dd_3 \delta^{-1} $}
\begin{tabular}{|c|c|c|c|c|c|c|c|}
 \hline
 \diagbox{$o(t_1)$}{$\mu_1$} & $(0, 0)$ & $(\tfrac{1}{2}, \tfrac{1}{2})$ & $(\tfrac{1}{2}, 0)$ & $(\tfrac{a}{4}, \tfrac{a^{-1}\delta}{4})$ & $(\tfrac{a}{4}, \tfrac{a^{-1}(\delta + 2)}{4})$ &$(\tfrac{a}{8}, \tfrac{a^{-1}\delta^{-1}}{8})$ & $(\tfrac{a}{8}, \tfrac{a^{-1}(\delta^{-1} + 4)}{8})$ \\
 \hline
 $t_1 \in \tfrac{1}{2}(-\dd_3 + 8\Zb_2)$ & 0& 0& 0 & 0 & 0 & $1/4$ & 0 \\
 \hline
 $t_1 \in \tfrac{1}{2}(3\dd_3 + 8\Zb_2)$ & 0& 0& 0 & 0 & 0 & 0 & $1/4$ \\
 \hline
 $t_1 \in 2(-\dd_3 + 4\Zb_2)$ & 0& 0& 0 & $1/2$ & 0 & 0 & 0 \\
 \hline
 $t_1 \in 2(\dd_3 + 4\Zb_2)$ & 0& 0& 0 & 0 & $1/2$ & 0 & 0 \\
 \hline
 $3$ & 0& 1& 0& 0 & 0 & 0 & 0 \\
 \hline
 $\ge 4$ & $o(t_1) - 4$ & 0& $1$& 0 & 0 & 0 & 0\\
 \hline
 \end{tabular}
\caption{$\dd_2 = 8, \beta = -32\dd_3 \delta^{-1} $}
\end{table}

For $i = 2$, we write $\alpha(\mu_2, t_2) := \beta \mu_2 \overline{\mu_2} - t_2$ in the notation of \cite{YYY}.
When $\dd_2 = 2$, we have $t_2 \in \tfrac{1}{2}(\dd_3^{-1} + 4\Zb_2)$ if $o(t_1) \ge 1$, since $t_1 + t_2 = 1/(2\dd_3)$.
Then with $\beta = 8 \dd_3 \delta^{-1}$, we have $\alpha(\mu(a;\delta, 1), t_2) = \beta \tfrac{a}{4} \tfrac{a^{-1}\delta}{4} - t_2 \in 2 \Zb_2$.
When $\dd_2 = 4$, we have
$$
t_2 \in
\begin{cases}
  \tfrac{\dd_3^{-1}}{4} + 1 + 2\Zb_2,& \text{ if } o(t_1) = 0,\\
 \tfrac{\dd_3^{-1}}{4} + 2\Zb_2,& \text{ if } o(t_1) \ge 1,\\
\end{cases}
$$
since $t_1 + t_2 = 1/(4\dd_3)$.
Then with $\beta = 16 \dd_3 \delta^{-1}$,
we have $\alpha(\mu(a; \delta, 2), t_2) \in \Zb_2^\times,
\alpha(\mu(a; \delta + 4, 2), t_2) \in 2\Zb_2$ if $o(t_1) = 0$, and
$\alpha(\mu(a; \delta, 2), t_2) \in 2\Zb_2, \alpha(\mu(a; \delta + 4, 2), t_2)  \in \Zb_2^\times$ if $o(t_1) \ge 1$.
When $\dd_2 = 8$, we have
$$
t_2 \in
\begin{cases}
  \tfrac{\dd_3^{-1}}{8} + \tfrac{\dd_3}{2} + 2\Zb_2,& \text{ if }
o(t_1) = -1, \\
 \tfrac{\dd_3^{-1}}{8} + 2\Zb_2,& \text{ if } o(t_1) \ge 1,\\
\end{cases}
$$
since $t_1 + t_2 = 1/(8\dd_3)$.
Then with $\beta = 32 \dd_3 \delta^{-1}$, we have
\begin{align*}
  \alpha(\mu(a; \delta \dd_3^2, 3), t_2)
 &\in
\begin{cases}
\tfrac{1}{2} \Zb_2^\times, & \text{ if }  o(t_1) = -1, \\
2 \Zb_2, & \text{ if }  o(t_1) \ge 1, \\
\end{cases}\\
\alpha(\mu(a; 5 \delta  \dd_3^2, 3), t_2)
& \in
\begin{cases}
2\Zb_2, & \text{ if }  o(t_1) = -1, \\
\tfrac{1}{2} \Zb_2^\times, & \text{ if }  o(t_1) \ge 1, \\
\end{cases}
\end{align*}
and $\alpha(\mu(a; \delta \dd_3^2 + 8, 3), t_2), \alpha(\mu(a; 5\delta \dd_3^2 + 8, 3), t_2) \not\in 2\Zb_2$.

\begin{table}[h]
\begin{tabular}{|c|c|c|}
 \hline
 \diagbox{$o(t_1)$}{$\mu_2$} & $(\tfrac{a}{4}, \tfrac{a^{-1} \delta}{4})$ & $(\tfrac{a}{4}, \tfrac{a^{-1}(\delta+2)}{4})$ \\
 \hline
 $\ge 1$ & $1/2$& 0\\
 \hline
 \end{tabular}
\caption{$\dd_2 = 2, \beta = 8\dd_3 \delta^{-1} $}
\begin{tabular}{|c|c|c|}
 \hline
 \diagbox{$o(t_1)$}{$\mu_2$}  & $(\tfrac{a}{8}, \tfrac{a^{-1}\delta}{8})$ & $(\tfrac{a}{8}, \tfrac{a^{-1}(\delta + 4)}{8})$ \\
 \hline
 $ 0$ & $0$ &$1/4$ \\
 \hline
 $\ge 1$ & $1/4$ & 0\\
 \hline
 \end{tabular}
\caption{$\dd_2 = 4, \beta = 16\dd_3 \delta^{-1} $}
\begin{tabular}{|c|c|c|c|c|}
 \hline
 \diagbox{$o(t_1)$}{$\mu_2$}  & $(\tfrac{a}{16}, \tfrac{a^{-1}\delta\dd_3^2}{16})$ & $(\tfrac{a}{16}, \tfrac{a^{-1}(\delta\dd_3^2 + 8)}{16})$ & $(\tfrac{a}{16}, \tfrac{5 a^{-1}\delta\dd_3^2}{16})$ & $(\tfrac{a}{16}, \tfrac{a^{-1}(5 \delta\dd_3^2 + 8)}{16})$ \\
 \hline
 $o(t_1) = -1$ & $0$ &$0$ & $1/8$ & 0  \\
 \hline
 $o(t_1) \ge 1$ & $1/8$ & 0 & 0& 0\\
 \hline
 \end{tabular}
\caption{$\dd_2 = 8, \beta = 32\dd_3 \delta^{-1} $}
\end{table}
Putting these together, we see that when $\dd_2 = 2$, we have
$$
\frac{W^*_t(0, \tilde{\phi}_{\dd, 2})}{\gamma(W_2)} =
\begin{cases}
  1, & \text{ if }o(t_1 ) = 1,\\
  o(t_1) - 4, & \text{ if }o(t_1 ) \ge 2.
\end{cases}
$$
Notice that $v_2(\Nm(t)) = o(t_1t_2) = o(t_1) - 1$. This proves the lemma for $\dd_2 = 2$.
When $\dd_2 = 4$, we have
$$
\frac{W^*_t(0, \tilde{\phi}_{\dd, 2})}{\gamma(W_2)} =
\begin{cases}
\mp 1, & \text{ if } t_1 \in \pm \dd_3 + 4\Zb_2, \\
  1, & \text{ if }o(t_1 ) = 2,\\
  o(t_1) - 5, & \text{ if }o(t_1 ) \ge 3.
\end{cases}
$$
Notice that $v_2(\Nm(t)) = o(t_1t_2) = o(t_1) - 2$.
If $t_1 \in \pm \dd_3 + 4\Zb_2$, then $4 t_2 \in \dd_3^{-1} + 4\Zb_2$ and $\Nm(2t) = 4t_1t_2 \equiv \pm 1 \bmod{4}$. This proves the lemma for $\dd_2 = 4$.
Finally when $\dd_2 = 8$, we have
$$
\lp \frac{2}{\delta} \rp
\frac{W^*_t(0, \tilde{\phi}_{\dd, 2})}{\gamma(W_2)} =
\begin{cases}
1, & \text{ if } t_1 \in  \tfrac{1}{2}(- \dd_3 + 8\Zb_2), \\
-1, & \text{ if } t_1 \in  \tfrac{1}{2}(3 \dd_3 + 8\Zb_2), \\
\mp 1, & \text{ if } t_1 \in  2(\pm \dd_3 + 4\Zb_2), \\
  1, & \text{ if }o(t_1 ) = 3,\\
  o(t_1) - 6, & \text{ if }o(t_1 ) \ge 4.
\end{cases}
$$
Notice that $v_2(\Nm(t)) = o(t_1t_2) = o(t_1) - 3$.
If $t_1 \in \tfrac{1}{2}(- \dd_3 + 4\Zb_2)$, then $8 t_2 \in \dd_3^{-1} + 4 + 8\Zb_2$ and $\Nm(4t) = 16t_1t_2 \equiv 3 \bmod{8}$.
Similarly, if $t_1 \in \tfrac{1}{2}(3 \dd_3 + 4\Zb_2)$, then $8 t_2 \in \dd_3^{-1} + 4 + 8\Zb_2$ and $\Nm(4t) = 16t_1t_2 \equiv 7 \bmod{8}$.
If $t_1 \in 2(\pm \dd_3 + 4\Zb_2)$, then $8 t_2 \in \dd_3^{-1} + 8\Zb_2$ and $\Nm(2t) = 4t_1t_2 \equiv \pm 1 \bmod{4}$.
This finishes the proof of the lemma.
\end{proof}

\subsubsection{$p = 3$.}
If $\dd_3 = 1$, then $\phi_{\dd, 3} = \mathrm{Char}(\mathcal{O}_E \otimes \Zb_3)$ and the calculations have been done before.
So suppose $\dd_3 = 3$.
There are 3 cases to consider.
\begin{itemize}
\item
$\lp \frac{d_i}{3} \rp = 1$
\item
$\lp \frac{d_1}{3} \rp \neq
\lp \frac{d_2}{3} \rp$.
\item
$\lp \frac{d_i}{3} \rp = -1$
\end{itemize}
The first case is similar to the case $p = 2$ considered above.
We again fix $\delta_i \in \Zb_3^\times$ square roots of $d_i$ and denote $\delta := \delta_1 \delta_2$.
Then the analog of the map in \eqref{eq:map2} for $p =3 $ identifies $L_{\dd, 3} = M_2(\Zb_3)$ with $\tilde{L}_3 := \Zb_3^4$, which has the quadratic form $\tilde{Q}_\dd(y) = -\frac{3\dd_2}{\delta}(y_1y_2 - y_3 y_4)$.
Denote $\tilde{\phi}_{\dd, 3} \in S(\tilde{L}_3)$ the Schwartz function associated to $\uf_{\dd, 3} \in \Cb[\Ac_{\dd, 3}]$.
Then the analog of Lemma \ref{lemma:decompose2} is as follows.

\begin{lemma}
  \label{lemma:decompose3a}
For $\delta_i \in \Zb_3^\times$ and $\delta = \delta_1\delta_2 \equiv \pm 1 \bmod{3}$, we have
$$
\tilde{\phi}_{\dd, 3} = \phi_0 \otimes \phi_{\delta} + \phi_{-\delta} \otimes \phi_0 + 2\phi_{\delta} \otimes \phi_{-\delta},
$$
where
\begin{align*}
  \phi_0 &:= 2\phi_{(0, 0)} - \lp \phi_{\tfrac{1}{3}(0, 1)} + \phi_{\tfrac{1}{3}(1, 0)} + \phi_{\tfrac{1}{3}(0, 2)}+ \phi_{\tfrac{1}{3}(0, 2)} \rp,~
\phi_{\pm 1} := \phi_{\tfrac{1}{3}(1, \pm 1)} + \phi_{\tfrac{1}{3}(2, \pm 2)}.
\end{align*}
\end{lemma}

\begin{proof}
  This follows from a straightforward calculations as in the case $p = 2$.
\end{proof}

\begin{lemma}
  \label{lemma:Whitt3a}
Suppose $\lp \frac{d_i}{3} \rp = 1$.
Then we have
  $$
\frac{W^*_t(0, \tilde{\phi}_{\dd, 3})}{\gamma(W_3)} =
\begin{cases}
  2, & v_3(\Nm(t)) = -2, \\
  2v_3(\Nm(t)), & v_3(\Nm(t)) \ge -1, \\
\end{cases}
  $$
for all totally positive $t \in F^\times$ with $\tr(t) = \frac{1}{\dd}$.
\end{lemma}

\begin{proof}
  Apply Lemma \ref{lemma:decompose3a} and Prop.\ 5.3, 5.7 in \cite{YYY}.
\end{proof}

In the second case, the prime $3$ is inert in $F$ and splits into two primes $\Pf, \overline{\Pf}$ in $E$. We therefore fix $\delta \in \overline{\Qb_3}$ such that $\delta^2 = D$, and denote $F_\delta := \Qb_3(\delta)$ the quadratic extension of $\Qb_3$ with $\mathcal{O}_\delta \subset F_\delta$ its ring of integers, where $3$ is inert.
For any choice of $\delta_j \in F_\delta$ such that $\delta_j^2 = d_j$ and $\delta_1 \delta_2 = \delta$, we can identify $W\otimes \Qb_3$ with $F_\delta \times  F_\delta$ via
$$
(a_1 + b_1 \sqrt{d_1}) \otimes (a_2 + b_2 \sqrt{d_2}) \mapsto
((a_1 + b_1 \delta_1) (a_2 + b_2 \delta_2),
(a_1 - b_1 \delta_1) (a_2 - b_2 \delta_2)).
$$
This identifies the $\Qb_3$-vector spaces $V \otimes \Qb_3$ and $F_\delta \times F_\delta$.
The $\Zb_3$-lattice $L_{3\dd_2} \otimes \Zb_3$ and its dual lattice $L_{3\dd_2}' \otimes \Zb_3$ in $V \otimes \Qb_3$ are then mapped to
$$
\tilde{L}_3 := \mathcal{O}_\delta \times \mathcal{O}_\delta,~
\tilde{L}_3' := 3^{-1}\mathcal{O}_\delta \times 3^{-1}\mathcal{O}_\delta
$$
respectively.
The finite $\Zb_3$-modules $(L'_{3\dd_2} / L_{3\dd_2}) \otimes \Zb_3$ and $\mathcal{O}_\delta/3 \mathcal{O}_\delta \times \mathcal{O}_\delta/3 \mathcal{O}_\delta$ are explicitly identified via
\begin{equation}
  \label{eq:identify3}
\begin{split}
\tfrac{1}{3\dd_2} \smat{x_3}{x_1}{x_4}{x_2} \otimes \Zb_3 \mapsto
 &(\dd_2^{-1} (x_1 + x_2 - x_3 - x_4 + (x_4 - x_2) \delta_1 - (x_3 + x_4) \delta_2 + x_4 \delta), \\
&\dd_2^{-1}(x_1 + x_2 - x_3 - x_4 - (x_4 - x_2) \delta_1 + (x_3 + x_4) \delta_2 + x_4 \delta)).
\end{split}
\end{equation}
The latter can be viewed as the finite quadratic module of the $\mathcal{O}_\delta$-lattice $\mathcal{O}_E \otimes \Zb_3 \cong \mathcal{O}_\delta \times \mathcal{O}_\delta$ with the $F_\delta$-quadratic form $Q_{\dd, \delta}(y) := -\frac{3\dd_2}{\delta}y_1y_2$ for $y = (y_1, y_2) \in F_\delta \times F_\delta$.
Note that $\mathcal{O}_\delta / 3\mathcal{O}_\delta = (\Zb/3\Zb)[\delta]$ is a finite field of size 9.

Now let $\tilde{\phi}_{\dd, 3} \in S(\mathcal{O}_E \otimes \Zb_3)$ be the Schwartz function associated to $\phi_{\dd, 3} \in S(L_{\dd, 3})$ under the map in \eqref{eq:identify3}.
It is easy to check by hand the following lemma.

\begin{lemma}
  \label{lemma:decompose3b}
Let $\delta, \delta_1, \delta_2 \in \overline{\Qb_3}$ be as above. Then
\begin{equation}
  \label{eq:decompose3b}
  \tilde{\phi}_{\dd, 3} =
2 \sum_{\mu \in S_0} \phi_{\mu}  -
 \sum_{\mu \in S_1} \phi_{\mu}  -
 \sum_{\mu \in S_{-1}} \phi_{\mu}  ,
\end{equation}
where $S_j :=  \{ \mu \in (\tfrac{1}{3}\Zb/\Zb)[\delta] \times (\tfrac{1}{3}\Zb/\Zb)[\delta]: Q_{\dd, \delta}(\mu) = \frac{1}{3}(-\dd_2 + j \delta)\in (\tfrac{1}{3}\Zb/\Zb)[\delta]\}$ for $j = 0, \pm 1$.
\end{lemma}

\begin{remark}
  The size of $S_j$ is 8 for every $j$.
\end{remark}

We can now apply Prop.\ 5.3 in \cite{YYY} to find the value of the Whittaker function.
\begin{lemma}
\label{lemma:Whitt3b}
Suppose $\lp \frac{d_1}{3} \rp \neq \lp \frac{d_1}{3} \rp$.
Then we have
  $$
\frac{W^*_t(0, \tilde{\phi}_{\dd, 3})}{\gamma(W_3)} =
\begin{cases}
  2, & \Nm(3t) \equiv 1 \bmod{3}, \\
-1, & \Nm(3t) \equiv 2 \bmod{3}, \\
\end{cases}
  $$
for all totally positive $t \in F^\times$ with $\tr(t) = \frac{1}{\dd}$.
\end{lemma}

\begin{proof}
  First, $\beta = -\tfrac{3\dd_2}{\delta}$, the normalizing $L$-factor is $L(1, \chi) = \tfrac{9}{8}$ and the volume $\mathrm{vol}(\mathcal{O}_E, d_\beta x) = \tfrac{1}{9}$.
Suppose $t = \tfrac{\delta + a}{6\dd_2 \delta} \in F_\delta$.
For $\mu \in S_j$, the quantity $3\alpha(\mu, t)$ is
$$
3\alpha(\mu, t) := 3(Q_{\dd, \delta}(\mu) - t) \equiv (-\dd_2 + j\delta) - 2 (\dd_2\delta)^{-1}(\delta + a)
\equiv (j - \dd_2 a)\delta
\bmod{3}
$$
since $\delta^2 = D \equiv 2 \bmod{3}$.
Now $3\alpha(\mu, t) \equiv 0 \bmod{3}$ if and only if $3 \mid (j - \dd_2 a)$.
This happens when $3 \mid(j, a)$, in which case $\mu \in S_0$ and $\Nm(3t) \equiv 1 + a^2 \equiv 1 \bmod{3}$.
The value of $\frac{W^*_t(0, \tilde{\phi}_{\dd, 3})}{\gamma(W_3)}$ is 2.
Otherwise if $3 \nmid a$ and $3 \mid (j - \dd_2 a)$, then $\mu \in S_{\dd_2 a}$ and $\Nm(3t) \equiv 2 \bmod{3}$.
The value of $\frac{W^*_t(0, \tilde{\phi}_{\dd, 3})}{\gamma(W_3)}$ is then $-1$.
This finishes the proof.
\end{proof}

In the last case, we need to calculate both the value and derivative of the Whittaker function at $s = 0$ since $3$ splits into the product of two inert primes $\pf_1\pf_2$ in $F$.
As in the setup of the previous two cases, we fix $\delta, \delta_i \in \overline{\Qb_3}$ such that $\delta_i^2 = d_i$ and $\delta = \delta_1\delta_2 \in \Zb_3$.
Denote $\tilde{E} := \Qb_3(\delta_1) = \Qb_3(\delta_2)$ the quadratic extension of $\Qb_3$ with ring of integers $\tilde{\mathcal{O}}$.
This gives an identification
\begin{align*}
\sigma_i:  F \otimes \Qb_3 &\cong \Qb_3: \sqrt{D} \mapsto (-1)^i \delta, \\
\sigma_i: E_{\pf_i} &\cong \tilde{E}: \sqrt{d_j} \mapsto (-1)^{(i-1)(j-1)} \delta_j.
\end{align*}
Then the isomorphism in \eqref{eq:SpaceIdentification} induces $V\otimes \Qb_3 \cong W \otimes \Qb_3 = E_{\pf_1} \oplus E_{\pf_2} \cong \tilde{E} \oplus \tilde{E}$, with the quadratic form on $y \in E_{\pf_i}$ given by $Q_i(y) := (-1)^{i-1} \tfrac{3\dd_2}{\sqrt{D}} \Nm(y)$.
The lattice $L_{\dd, 3}$ is then isometric to
$$
\tilde{L}_{\dd, 3} := \tilde{\mathcal{O}} \oplus \tilde{\mathcal{O}} \subset \tilde{E} \oplus \tilde{E},
$$
whose dual lattice is $\tilde{L}_{\dd, 3}' := \tfrac{1}{3}\tilde{\mathcal{O}} \oplus \tfrac{1}{3} \tilde{\mathcal{O}} \subset \tilde{E} \oplus \tilde{E}$, with respect to the quadratic form $\tilde{Q}_{\dd, \delta}(y) := -\tfrac{3\dd_2}{\delta} (\Nm(y_1) - \Nm(y_2))$ for $y = (y_1, y_2) \in \tilde{E} \oplus \tilde{E}$.
Under this identification, the Schwartz function $\tilde{\phi}_{\dd, 3} \in S(\tilde{L}_{\dd, 3})$ associated to $\phi_{\dd, 3} \in S(L_{\dd, 3})$ has the following decomposition.
\begin{lemma}
  \label{lemma:decompose3c}
Let $\delta, \delta_1, \delta_2 \in \overline{\Qb_3}$ be as above. Then
\begin{equation}
  \label{eq:decompose3b}
  \tilde{\phi}_{\dd, 3} =
2 \sum_{\mu \in S_{1}} \phi_\mu \otimes \phi_{0}  +
2 \sum_{\mu \in S_{-1}} \phi_0 \otimes \phi_{\mu}  -
 \sum_{\mu_1 \in S_{-1}, \mu_2 \in S_{1}} \phi_{\mu_1} \otimes \phi_{\mu_2}
\end{equation}
where $S_j :=  \{ \mu \in \tfrac{1}{3}\tilde{\mathcal{O}}/\tilde{\mathcal{O}}: -\tfrac{3}{\delta}\Nm(\mu) \equiv \tfrac{j}{3} \bmod{\Zb_3}\}$ for $j = \pm 1$.
\end{lemma}

Now, we can again apply Prop.\ 5.3 in \cite{YYY} to calculate the values and derivatives of the Whittaker function.

\begin{lemma}
  \label{lemma:Whitt3c}
Suppose $\lp \frac{d_1}{3} \rp = \lp \frac{d_1}{3} \rp = -1$.
Then we have
\begin{align*}
  \frac{W^*_t(0, {\phi}_{\dd, 3})}{\gamma(W_3)} &=
\begin{cases}
-1, & v_3(\Nm(t)) = -2,\\
  2,& v_3(\Nm(t)) \ge 0 \text{ is even,}\\
  0,& \text{ otherwise,}
\end{cases} \\
  \frac{{W^{*,}}'_t(0, {\phi}_{\dd, 3})}{\gamma(W_3)} &=
\lp v_3(\Nm(t)) + \frac{3}{2} \rp \log 3,\quad \text{ if } v_3(\Nm(t)) \ge -1\text{ is odd,}
\end{align*}
for all totally positive $t \in F^\times$ with $\tr(t) = \frac{1}{\dd}$.
\end{lemma}

\begin{proof}
Denote $t_i := \sigma_i(t) \in \Qb_3$ and $o(t_i)$ its valuation.
Since $\tr(t) = \tfrac{1}{3\dd_2}$, either $o(t_i) = -1$ for both $i = 1, 2$, or $o(t_i) \ge 0$ for exactly one of $i = 1, 2$.
In the first case, it is easy to check that $W_{t_i}(s, \phi_{\mu} \otimes \phi_0)$ and $W_{t_i}(s, \phi_{0} \otimes \phi_\mu)$ are identically zero by Prop.\ 5.7 in \cite{YYY}.
If we write $t_1 = \tfrac{\delta - a}{2\dd_2 3 \delta}, t_2 = \tfrac{\delta + a}{2\dd_2 3 \delta}$ with $a \in \Zb_3$, then we must have $a \in 3\Zb_3$ since $\delta^2 = D \in 1 + 3\Zb_3$ and
$$
-2 = o(t_1) + o(t_2) =
o(t_1t_2) = -2 + o(\delta^2 - a^2) = -2 + o(1 - a^2).
$$
That means for $\mu_1 \in S_{-1}$ and $\mu_2 \in S_{1}$, we have
\begin{align*}
\alpha(\mu_1, t_1) &= - \frac{3\dd_2}{\delta}\Nm(\mu_1) - t_1 \equiv -\frac{\dd_2}{3} - \frac{\delta - a}{2\dd_2 3 \delta} \equiv 0 \bmod \Zb_3 \\
\alpha(\mu_2, t_2) &= \frac{3\dd_2}{\delta}\Nm(\mu_2) - t_2 \equiv -\frac{\dd_2}{3} - \frac{\delta + a}{2\dd_2 3 \delta} \equiv 0 \bmod \Zb_3.
\end{align*}
By Prop.\ 5.3 in \cite{YYY}, $\gamma(W_3)^{-1} W_{t_i}(0, \phi_{\mu_1} \otimes \phi_{\mu_2}) = \tfrac{1}{16}$ for any $(\mu_1, \mu_2) \in S_{-1} \times S_1$. Since $S_j$ has size 4 for $j = \pm 1$, we obtain
$$
\frac{W_{t}(0, \phi_{\dd, 3}) }{\gamma(W_3)} = -1
$$
when $v_3(\Nm(t)) = o(t_1) + o(t_2) = -2$.

In the second case, suppose $o(t_1) \ge 0$. Then Prop.\ 5.3 and 5.7 in \cite{YYY} imply that $W_{t_i}(s, \phi_{\mu_1} \otimes \phi_{\mu_2})$ vanishes identically for $(\mu_1, \mu_2) \in S_{-1} \times S_1$ and
\begin{align*}
  \frac{W^*_{t_1}(0, {\phi}_{0} )}{\gamma(W_{\pf_1})} &= \frac{1 + (-1)^{o(t_1) - 1}}{2},~
  \frac{W^*_{t_2}(0, {\phi}_{\mu})}{\gamma(W_{\pf_2})} = L(1, \chi_{\pf_2}) 3^{-1} = \frac{1}{4}, \\
  \frac{{W^{*,}}'_{t_1}(0, {\phi}_{0} )}{\gamma(W_{\pf_1})} &= \frac{2o(t_1) + 1}{2} \log(3) \text{ when } 2\mid o(t_1)
\end{align*}
  when $\mu \in S_1$ as $\alpha(\mu, t_2) = -\frac{\dd_2 \Nm(\mu)}{3\delta} - t_2 \in \frac{\dd_2}{3} - t_2 + \Zb_3 = \Zb_3$.
Since $v_3(\Nm(t)) = o(t_1t_2) = o(t_1) - 1$, we obtain the lemma when $o(t_1) \ge 0$. The case $o(t_2) \ge 0$ holds similarly.
\end{proof}

\section*{Appendix}
\label{appendix}
We record here the set $\kappa_{\dd, 2}(N'_{\dd, 2})\subset \Ac_{\dd, 2} =
\Zb/\dd_2\Zb \times
\Zb/2\dd_2\Zb \times
\Zb/2\dd_2\Zb \times
\Zb/\dd_2\Zb $.
Note that the group $N'_{\dd, 2}$ and the map $\dd_3\cdot \kappa_{\dd, 2}$ only depend on $\dd_2$.
This helps with checking Lemma \ref{lemma:additive}.
\begin{align*}
  \kappa_{1, 2}&(N'_{1, 2}) =
  \kappa_{3, 2}(N'_{3, 2}) =
\{
[0, 0]
\},~
  \kappa_{2, 2}(N'_{2, 2}) =
  \kappa_{6, 2}(N'_{6, 2}) =
\{
  [ 1 , 0 , 0 , 1 ] ,
[ 1 , 2 , 2 , 1 ]
\}.
\end{align*}
\begin{align*}
  \kappa_{4, 2}&(N'_{4, 2}) =
  \kappa_{12, 2}(N'_{12, 2}) =
 \\
\big\{
&  [ 1 , 2 , 6 , 3 ] ,
[ 1 , 6 , 2 , 3 ] ,
[ 1 , 0 , 0 , 1 ] ,
[ 1 , 4 , 4 , 1 ] ,
[ 3 , 6 , 2 , 1 ] ,
[ 3 , 2 , 6 , 1 ] ,
[ 3 , 0 , 0 , 3 ] ,
[ 3 , 4 , 4 , 3 ]
\big\}.
\end{align*}
\begin{align*}
  \kappa_{8, 2}&(N'_{8, 2}) = \\\Big\{
&[ 1 , 2 , 14 , 7 ] ,
[ 1 , 6 , 10 , 7 ] ,
[ 1 , 10 , 6 , 7 ] ,
[ 1 , 14 , 2 , 7 ] ,
[ 1 , 0 , 0 , 1 ] ,
[ 1 , 4 , 12 , 1 ] ,
[ 1 , 8 , 8 , 1 ] ,
[ 1 , 12 , 4 , 1 ] ,\\
&[ 3 , 14 , 10 , 5 ] ,
[ 3 , 2 , 6 , 5 ] ,
[ 3 , 6 , 2 , 5 ] ,
[ 3 , 10 , 14 , 5 ] ,
[ 3 , 8 , 0 , 3 ] ,
[ 3 , 12 , 12 , 3 ] ,
[ 3 , 0 , 8 , 3 ] ,
[ 3 , 4 , 4 , 3 ] ,\\
&[ 5 , 2 , 6 , 3 ] ,
[ 5 , 6 , 2 , 3 ] ,
[ 5 , 10 , 14 , 3 ] ,
[ 5 , 14 , 10 , 3 ] ,
[ 5 , 8 , 0 , 5 ] ,
[ 5 , 12 , 12 , 5 ] ,
[ 5 , 0 , 8 , 5 ] ,
[ 5 , 4 , 4 , 5 ] ,\\
&[ 7 , 14 , 2 , 1 ] ,
[ 7 , 2 , 14 , 1 ] ,
[ 7 , 6 , 10 , 1 ] ,
[ 7 , 10 , 6 , 1 ] ,
[ 7 , 0 , 0 , 7 ] ,
[ 7 , 4 , 12 , 7 ] ,
[ 7 , 8 , 8 , 7 ] ,
[ 7 , 12 , 4 , 7 ]
\Big\}.
\end{align*}
\begin{align*}
  \kappa_{24, 2}&(N'_{24, 2}) =   \kappa_{24, 2}(N'_{8, 2}) =    3^{-1} \cdot \kappa_{8, 2}(N'_{8, 2}) =
\\\Big\{
&  [ 3 , 6 , 10 , 5 ] ,
[ 3 , 2 , 14 , 5 ] ,
[ 3 , 14 , 2 , 5 ] ,
[ 3 , 10 , 6 , 5 ] ,
[ 3 , 0 , 0 , 3 ] ,
[ 3 , 12 , 4 , 3 ] ,
[ 3 , 8 , 8 , 3 ] ,
[ 3 , 4 , 12 , 3 ] ,\\
&[ 1 , 10 , 14 , 7 ] ,
[ 1 , 6 , 2 , 7 ] ,
[ 1 , 2 , 6 , 7 ] ,
[ 1 , 14 , 10 , 7 ] ,
[ 1 , 8 , 0 , 1 ] ,
[ 1 , 4 , 4 , 1 ] ,
[ 1 , 0 , 8 , 1 ] ,
[ 1 , 12 , 12 , 1 ] ,\\
&[ 7 , 6 , 2 , 1 ] ,
[ 7 , 2 , 6 , 1 ] ,
[ 7 , 14 , 10 , 1 ] ,
[ 7 , 10 , 14 , 1 ] ,
[ 7 , 8 , 0 , 7 ] ,
[ 7 , 4 , 4 , 7 ] ,
[ 7 , 0 , 8 , 7 ] ,
[ 7 , 12 , 12 , 7 ] ,\\
&[ 5 , 10 , 6 , 3 ] ,
[ 5 , 6 , 10 , 3 ] ,
[ 5 , 2 , 14 , 3 ] ,
[ 5 , 14 , 2 , 3 ] ,
[ 5 , 0 , 0 , 5 ] ,
[ 5 , 12 , 4 , 5 ] ,
[ 5 , 8 , 8 , 5 ] ,
[ 5 , 4 , 12 , 5 ]
\Big\}.
\end{align*}

Here we also include an explicit example for Theorem \ref{thm:intro}.
Let $d_1 = -31, d_2 = -127$, which have class numbers 3 and 5 respectively and satisfy $d_j \equiv 17 \bmod{24}$.
Then the minimal polynomials of the invariants $f\lp\left[1, \tfrac{-1 + \sqrt{d_j}}{2}\right]\rp$ are
\begin{equation}
  \label{eq:minpol}
  g_1(x) = x^3 + x - 1,~
  g_2(x) = x^{5} - x^{4} - 2x^{3} + x^{2} + 3x - 1.
\end{equation}
The following table lists the values of $\mathfrak{F}(m)$ for various $m$.
\begin{center}
  \begin{tabular}{|c|c|c|c|c|c|c|c|c|c|c|c|}
 \hline
$a$ & $m$ & $m \bmod 96$ &  $\mathfrak{F}(m)$ & $\mathfrak{F}(\tfrac{m}{2^2})$ & $\mathfrak{F}(\tfrac{m}{4^2})$  & $\mathfrak{F}(\tfrac{m}{8^2})$ & $\mathfrak{F}(\tfrac{m}{16^2})$  & $\mathfrak{F}(\tfrac{m}{6^2})$ & $\mathfrak{F}(\tfrac{m}{12^2})$  & $\mathfrak{F}(\tfrac{m}{24^2})$ & $\mathfrak{F}(\tfrac{m}{48^2})$ \\ \hline
1 &$ 2^{3} \cdot 3 \cdot 41 $& 24 &$ 3^8 $&$ 3^4 $&$ 1 $&$ 1 $&$ 1 $&$ 1 $&$ 1 $&$ 1 $&$ 1 $ \\ \hline
3 &$ 2 \cdot 491 $& 22 &$ 491^2 $&$ 1 $&$ 1 $&$ 1 $&$ 1 $&$ 1 $&$ 1 $&$ 1 $&$ 1 $ \\ \hline
5 &$ 2 \cdot 3 \cdot 163 $& 18 &$ 3^4 $&$ 1 $&$ 1 $&$ 1 $&$ 1 $&$ 1 $&$ 1 $&$ 1 $&$ 1 $ \\ \hline
7 &$ 2^{2} \cdot 3^{5} $& 12 &$ 3^9 $&$ 3^3 $&$ 1 $&$ 1 $&$ 1 $&$ 3^2 $&$ 1 $&$ 1 $&$ 1 $ \\ \hline
9 &$ 2^{2} \cdot 241 $& 4 &$ 241^3 $&$ 241 $&$ 1 $&$ 1 $&$ 1 $&$ 1 $&$ 1 $&$ 1 $&$ 1 $ \\ \hline
11 &$ 2 \cdot 3^{2} \cdot 53 $& 90 &$ 53^2 $&$ 1 $&$ 1 $&$ 1 $&$ 1 $&$ 1 $&$ 1 $&$ 1 $&$ 1 $ \\ \hline
13 &$ 2 \cdot 3 \cdot 157 $& 78 &$ 3^4 $&$ 1 $&$ 1 $&$ 1 $&$ 1 $&$ 1 $&$ 1 $&$ 1 $&$ 1 $ \\ \hline
15 &$ 2^{5} \cdot 29 $& 64 &$ 29^6 $&$ 29^4 $&$ 29^2 $&$ 1 $&$ 1 $&$ 1 $&$ 1 $&$ 1 $&$ 1 $ \\ \hline
17 &$ 2^{4} \cdot 3 \cdot 19 $& 48 &$ 3^{10} $&$ 3^6 $&$ 3^2 $&$ 1 $&$ 1 $&$ 1 $&$ 1 $&$ 1 $&$ 1 $ \\ \hline
19 &$ 2 \cdot 3 \cdot 149 $& 30 &$ 3^4 $&$ 1 $&$ 1 $&$ 1 $&$ 1 $&$ 1 $&$ 1 $&$ 1 $&$ 1 $ \\ \hline
21 &$ 2 \cdot 19 \cdot 23 $& 10 &$ 23^4 $&$ 1 $&$ 1 $&$ 1 $&$ 1 $&$ 1 $&$ 1 $&$ 1 $&$ 1 $ \\ \hline
23 &$ 2^{2} \cdot 3 \cdot 71 $& 84 &$ 3^6 $&$ 3^2 $&$ 1 $&$ 1 $&$ 1 $&$ 1 $&$ 1 $&$ 1 $&$ 1 $ \\ \hline
25 &$ 2^{2} \cdot 3^{2} \cdot 23 $& 60 &$ 23^3 $&$ 23 $&$ 1 $&$ 1 $&$ 1 $&$ 23 $&$ 1 $&$ 1 $&$ 1 $ \\ \hline
27 &$ 2 \cdot 401 $& 34 &$ 401^2 $&$ 1 $&$ 1 $&$ 1 $&$ 1 $&$ 1 $&$ 1 $&$ 1 $&$ 1 $ \\ \hline
29 &$ 2 \cdot 3^{2} \cdot 43 $& 6 &$ 43^2 $&$ 1 $&$ 1 $&$ 1 $&$ 1 $&$ 1 $&$ 1 $&$ 1 $&$ 1 $ \\ \hline
31 &$ 2^{3} \cdot 3 \cdot 31 $& 72 &$ 3^8 $&$ 3^4 $&$ 1 $&$ 1 $&$ 1 $&$ 1 $&$ 1 $&$ 1 $&$ 1 $ \\ \hline
33 &$ 2^{3} \cdot 89 $& 40 &$ 89^4 $&$ 89^2 $&$ 1 $&$ 1 $&$ 1 $&$ 1 $&$ 1 $&$ 1 $&$ 1 $ \\ \hline
35 &$ 2 \cdot 3 \cdot 113 $& 6 &$ 3^4 $&$ 1 $&$ 1 $&$ 1 $&$ 1 $&$ 1 $&$ 1 $&$ 1 $&$ 1 $ \\ \hline
37 &$ 2 \cdot 3 \cdot 107 $& 66 &$ 3^4 $&$ 1 $&$ 1 $&$ 1 $&$ 1 $&$ 1 $&$ 1 $&$ 1 $&$ 1 $ \\ \hline
39 &$ 2^{2} \cdot 151 $& 28 &$ 151^3 $&$ 151 $&$ 1 $&$ 1 $&$ 1 $&$ 1 $&$ 1 $&$ 1 $&$ 1 $ \\ \hline
41 &$ 2^{2} \cdot 3 \cdot 47 $& 84 &$ 3^6 $&$ 3^2 $&$ 1 $&$ 1 $&$ 1 $&$ 1 $&$ 1 $&$ 1 $&$ 1 $ \\ \hline
43 &$ 2 \cdot 3^{2} \cdot 29 $& 42 &$ 29^2 $&$ 1 $&$ 1 $&$ 1 $&$ 1 $&$ 1 $&$ 1 $&$ 1 $&$ 1 $ \\ \hline
45 &$ 2 \cdot 239 $& 94 &$ 239^2 $&$ 1 $&$ 1 $&$ 1 $&$ 1 $&$ 1 $&$ 1 $&$ 1 $&$ 1 $ \\ \hline
47 &$ 2^{4} \cdot 3^{3} $& 48 &$ 3^{10} $&$ 3^6 $&$ 3^2 $&$ 1 $&$ 1 $&$ 3^3 $&$ 3 $&$ 1 $&$ 1 $ \\ \hline
49 &$ 2^{7} \cdot 3 $& 0 &$ 3^8 $&$ 3^6 $&$ 3^4 $&$ 3^2 $&$ 1 $&$ 1 $&$ 1 $&$ 1 $&$ 1 $ \\ \hline
51 &$ 2 \cdot 167 $& 46 &$ 167^2 $&$ 1 $&$ 1 $&$ 1 $&$ 1 $&$ 1 $&$ 1 $&$ 1 $&$ 1 $ \\ \hline
53 &$ 2 \cdot 3 \cdot 47 $& 90 &$ 3^4 $&$ 1 $&$ 1 $&$ 1 $&$ 1 $&$ 1 $&$ 1 $&$ 1 $&$ 1 $ \\ \hline
55 &$ 2^{2} \cdot 3 \cdot 19 $& 36 &$ 3^6 $&$ 3^2 $&$ 1 $&$ 1 $&$ 1 $&$ 1 $&$ 1 $&$ 1 $&$ 1 $ \\ \hline
57 &$ 2^{2} \cdot 43 $& 76 &$ 43^3 $&$ 43 $&$ 1 $&$ 1 $&$ 1 $&$ 1 $&$ 1 $&$ 1 $&$ 1 $ \\ \hline
59 &$ 2 \cdot 3 \cdot 19 $& 18 &$ 3^4 $&$ 1 $&$ 1 $&$ 1 $&$ 1 $&$ 1 $&$ 1 $&$ 1 $&$ 1 $ \\ \hline
61 &$ 2 \cdot 3^{3} $& 54 &$ 3^4 $&$ 1 $&$ 1 $&$ 1 $&$ 1 $&$ 1 $&$ 1 $&$ 1 $&$ 1 $ \\ \hline
  \end{tabular}
\end{center}
By the theorem of Gross-Zagier, one obtains $J(d_1, d_2)$ by simply takes the product of all the numbers in the fourth column.
For $f_{s}(d_1, d_2)$, one takes product of the entries $\mathfrak{F}(\tfrac{m}{4r^2})$ over all the $m$'s in the table and $r \mid s$ satisfying $m \equiv 4\cdot 19 (d_1 +  d_2 - 1) \bmod{4sr}$. This congruence condition eliminates many entries, especially if $s$ is large. For example, we have
\begin{align*}
  f_{24}(d_1, d_2) &= \lp \mathfrak{F}\lp \tfrac{2^2 3^5}{4} \rp  \mathfrak{F}\lp \tfrac{2^4 3^3}{4\cdot 2^2} \rp  \mathfrak{F}\lp \tfrac{2^2 3^5}{4 \cdot 3^2} \rp \mathfrak{F}\lp \tfrac{2^4 3^3}{4 \cdot 6^2} \rp \rp^{1/2} =
3^4
\end{align*}
by Theorem \ref{thm:intro}.
One can then immediately check that this is the absolute value of the resultant of the minimal polynomials $g_1, g_2$ in \eqref{eq:minpol}.

\bibliography{YZ}{}
\bibliographystyle{amsplain}

\end{document}